\documentclass[11pt]{preprint}

\usepackage[full]{textcomp}
\usepackage[osf]{newtxtext}
\usepackage{cabin}
\usepackage{enumerate}
\usepackage{enumitem}
\usepackage{bm}
\usepackage{bbm}
\usepackage{dsfont}

\usepackage[varqu,varl]{inconsolata}
\usepackage[cal=boondoxo]{mathalfa}

\usepackage{cancel}
\usepackage{mcomment}
\usepackage{hyperref}
\usepackage{breakurl}
\usepackage{mathrsfs}
\usepackage{booktabs}
\usepackage{caption}

\usepackage{tikz}
\usetikzlibrary{decorations.pathreplacing,decorations.markings}
\usepackage{amsmath} 
\usepackage{mhequ} 
\usepackage{mhsymb} 
\usepackage{mhenvs} 
\usepackage{microtype}
\usepackage{amssymb} 
\usepackage{eufrak}

\usepackage{wasysym}
\usepackage{centernot}

\usepackage{cleveref}

\crefname{equation}{}{}

\usepackage{oldgerm}

\newcommand{\frakf}{\mathfrak{f}}
\newcommand{\frakg}{\mathfrak{g}}
\newcommand{\frakh}{\mathfrak{h}}

\newcommand{\bbC}{\mathbb{C}}

\newcommand{\bbE}{\mathbb{E}}

\newcommand{\bbN}{\mathbb{N}}

\newcommand{\bbP}{\mathbb{P}}

\newcommand{\bbR}{\mathbb{R}}

\newcommand{\bfg}{\mathbf{g}}
\newcommand{\bfh}{\mathbf{h}}

\newcommand{\bfu}{\mathbf{u}}
\newcommand{\bfv}{\mathbf{v}}
\newcommand{\bfw}{\mathbf{w}}

\newcommand{\bfE}{\mathbf{E}}

\newcommand{\bfL}{\mathbf{L}}
\newcommand{\bfM}{\mathbf{M}}

\newcommand{\bfP}{\mathbf{P}}

\newcommand{\calg}{\mathcal{g}}

\newcommand{\calA}{\mathcal{A}}
\newcommand{\calB}{\mathcal{B}}

\newcommand{\calE}{\mathcal{E}}
\newcommand{\calF}{\mathcal{F}}
\newcommand{\calG}{\mathcal{G}}
\newcommand{\calH}{\mathcal{H}}
\newcommand{\calI}{\mathcal{I}}

\newcommand{\calL}{\mathcal{L}}
\newcommand{\calM}{\mathcal{M}}
\newcommand{\calN}{\mathcal{N}}

\newcommand{\calR}{\mathcal{R}}
\newcommand{\calS}{\mathcal{S}}
\newcommand{\calT}{\mathcal{T}}

\newcommand{\calX}{\mathcal{X}}
\newcommand{\calY}{\mathcal{Y}}
\newcommand{\calZ}{\mathcal{Z}}

\newcommand{\fHat}{\hat{f}}
\newcommand{\gHat}{\hat{g}}

\newcommand{\VHat}{\hat{V}}

\newcommand{\bfMHat}{\hat{\mathbf{M}}}

\newcommand{\calLHat}{\hat{\mathcal{L}}}

\newcommand{\etaHat}{\hat{\eta}}

\newcommand{\BBar}{\overline{B}}

\newcommand{\UBar}{\overline{U}}
\newcommand{\VBar}{\overline{V}}

\newcommand{\bfEBar}{\overline{\mathbf{E}}}

\newcommand{\bfPBar}{\overline{\mathbf{P}}}

\newcommand{\etaBar}{\overline{\eta}}

\newcommand{\piBar}{\overline{\pi}}

\newcommand{\hTil}{\tilde{h}}

\newcommand{\XTil}{\tilde{X}}

\newcommand{\etaTil}{\tilde{\eta}}

\newcommand{\piTil}{\tilde{\pi}}

\newcommand{\omegaTil}{\tilde{\omega}}

\newcommand{\OmegaTil}{\tilde{\Omega}}

\newcommand{\Var}{\mathrm{Var}}

\renewcommand{\R}{\mathbb{R}}
\renewcommand{\C}{\mathbb{C}}
\renewcommand{\N}{\mathbb{N}}

\renewcommand{\E}{\mathbb{E}}

\newcommand{\BF}{\mathbf{F}}

\newcommand{\cA}{\mathcal{A}}

\newcommand{\cF}{\mathcal{F}}
\newcommand{\cG}{\mathcal{G}}
\newcommand{\cH}{\mathcal{H}}

\newcommand{\cL}{\mathcal{L}}
\newcommand{\cM}{\mathcal{M}}
\newcommand{\cN}{\mathcal{N}}

\newcommand{\cS}{\mathcal{S}}

\newcommand{\cV}{\mathcal{V}}

\newcommand{\cm}{\mathcal{m}}
\newcommand{\cg}{\mathcal{g}}

\newcommand{\cl}{\mathrm{cl}}
\newcommand{\dd}{\mathrm{d}}

\newcommand{\ST}{\mathscr{T}}

\renewcommand{\fH}{\mathfrak{H}}

\def\one{\mathrm{(I)}}
\def\two{\mathrm{(II)}}
\def\three{\mathrm{III}}
\def\four{\mathrm{IV}}

\newcommand{\1}{\mathds{1}}

\def\one{\mathrm{(I)}}
\def\two{\mathrm{(II)}}
\def\three{\mathrm{(III)}}
\def\four{\mathrm{(IV)}}

\newcommand{\half}{\frac{1}{2}}
\newcommand{\thalf}{\tfrac{1}{2}}

\newcommand{\fock}{\Gamma L^2}

\colorlet{darkblue}{blue!90!black}
\colorlet{darkred}{red!90!black}
\colorlet{darkgreen}{green!50!black}
\colorlet{darkyellow}{yellow!90!black}

\DeclareMathOperator*{\rot}{rot}
\DeclareMathOperator*{\ddiv}{div}
\newcommand{\ic}{\mathbf{1}}

\newcommand{\coup}{\gamma}

\def\genAR#1{\calA^R_+[#1]}

\def\genAN#1{\calA^N_+[#1]}

\def\ffock#1{\mu_{#1}}

\newcommand{\vertiii}[1]{{\vert\kern-0.25ex\vert\kern-0.25ex\vert #1 
    \vert\kern-0.25ex\vert\kern-0.25ex\vert}}

\title{An invariance principle for the $2d$ weakly self-repelling Brownian polymer}
\author{Giuseppe Cannizzaro, Harry Giles}

\institute{University of Warwick, CV4 7AL, UK \\ 
\hspace{-0.33cm}\email{giuseppe.cannizzaro@warwick.ac.uk, \\harry.giles@warwick.ac.uk}}

\begin{document}

\maketitle

\begin{abstract}
We investigate the large-scale behaviour of the Self-Repelling Brownian Polymer (SRBP)  in the critical dimension 
$d=2$. The SRBP is a model of self-repelling motion, 
which is formally given by the solution a stochastic differential equation driven by a standard Brownian motion 
and with a drift given by the negative gradient of its own local time. 
As with its discrete counterpart, the ``true'' self-avoiding walk (TSAW) of 
[D.J. Amit, G. Parisi, \& L. Peliti, Asymptotic behaviour of the ``true'' self-avoiding walk, Phys. Rev. B, 1983],  
it is conjectured to be logarithmically superdiffusive, i.e. to be such that its mean-square displacement 
grows as $t(\log t)^\beta$ for $t$ large and some currently unknown $\beta\in(0,1)$.

The main result of the paper is an invariance principle for the SRBP under the weak coupling scaling, 
which corresponds to scaling the SRBP diffusively and simultaneously tuning down 
the strength of the self-interaction in a scale-dependent way. 
The diffusivity for the limiting Brownian motion is explicit and its expression provides 
compelling evidence that the $\beta$ above should be $1/2$. 
Further, we derive the scaling limit of the so-called environment seen by the particle process, 
which formally solves a non-linear singular stochastic PDE of transport-type, 
and prove this is given by the solution of a stochastic linear transport equation 
with enhanced diffusivity. 
\end{abstract}

\setcounter{tocdepth}{2} 
\tableofcontents

\section{Introduction and main results}
\subsection{Introduction and related works}
\label{sec:introduction}
We study a model of self-avoiding motion known as the self-repelling Brownian polymer (SRBP) in the critical dimension, $d=2$. For general dimension $d$, the SRBP is a $\bbR^d$-valued continuous stochastic process $(X_t)_{t \ge 0}$ driven by Brownian motion and repelled by its own local time. In other words, $(X_t)_{t \ge 0}$ has a drift which pushes the process away from regions of space it has previously occupied. 

Ideally, one would like to define the SRBP according to 
\begin{equation}
\label{eq:38}
\dd X_t = \dd B_t- \gamma^2 \nabla L_t(X_t) \dd t , \qquad X_0 = 0\,,
\end{equation}
where $(L_t)_{t \ge 0}$ is the occupation measure of $(X_t)_{t \ge 0}$ defined by
\begin{equ}
 L_t(x) = \int_0^t \delta_0(x-X_s) \dd s , \qquad t\ge0, x \in \bbR^d
 \end{equ}
for $\delta_0 : \bbR^d \rightarrow \bbR$ the Dirac delta at zero, and the {\it coupling constant} $\gamma>0$ which controls the strength of the self-interaction. As written, \eqref{eq:38} is meaningless, and one is led to consider a regularised version which is given by the following SDE
\begin{equ}[e:SRBP1]
\dd X_t = \dd B_t - \gamma^2 \left( \int_0^t \nabla V (X_t-X_s) \dd s\right) \dd t, \qquad X_0 = 0
\end{equ}
where $V : \bbR^{d} \rightarrow \bbR$ is a smooth mollifier. The drift term may be rewritten in terms of the mollified occupation field $V * L_t(X_t)$, and with this in mind, the self-interaction can be described as follows: over an infinitesimal time-step $[t, t+\dd t]$, the particle updates the occupation measure by adding mass at its current location, and the occupation measure influences the particle by providing a (dynamic) scalar potential which induces the drift.

The first instances of self repelling motion date back to the early eighties \cite{AmitPariPeli83_AsymptoticBehavior}, 
when physicists introduced the ``true'' self-avoiding walk (TSAW) as a model capturing the statistics of a growing polymer. 
In short, this is the random walk governed by the non-Markovian transitions given by, for $y$ a neighbour of $X_n$
\begin{equ}
\bbP( X_{n+1} = y | \vec{X}_n) \propto e^{- \beta \big\{ \ell(y ; \vec{X}_n) - \ell(X_n ; \vec{X}_n) \big\}}
\end{equ}
where $\vec{X}_n = (X_0, ..., X_n)$ is the history of the process and $\ell$ is the occupation time, $\ell(y ; \vec{x}_n) \eqdef \sum_{m=0}^n \ic_{x_m}(y)$. 
The TSAW is a discrete cousin of the SRBP, which was independently introduced 
by probabilists shortly thereafter \cite{NorrRogeWill87_SelfavoidingRandom, DurrettRogers92_AsymptoticBehavior}. 
These, and other models of self-avoiding motion \cite{Kesten63_NumberSelf,Lawler80_SelfavoidingRandom} are notoriously difficult to study because of their self-interaction 
and long-term memory, in particular they are {\it not} Markovian. 

A first question regards the large-scale behaviour for the mean squared displacement of $(X_t)_{t \ge 0}$. Heuristically, one might expect diffusive behaviour in higher dimensions, where the self-avoiding ``constraint'' is less restrictive, leaving only the influence of the Brownian motion. Non-rigorous scaling \cite{AlderWainwrigh67_VelocityAutocorrelations, AlderWainwrigh70_DecayVelocity, ForsNelsStep77_LargedistanceLongtime} (see also the appendix of \cite{TothValko12_SuperdiffusiveBounds}) and renormalization group \cite{AmitPariPeli83_AsymptoticBehavior} arguments lead to the following dimension dependent predictions
\begin{align}
\label{eq:36}
\bbE[|X_t|^2] \sim
\begin{cases} 
t^{4/3} & d = 1 \\
t (\log t)^\beta & d = 2 \\
t & d \ge 3 
\end{cases}
\end{align}
for some $\beta \in (0,1)$. In particular, the process is conjectured to be diffusive in high dimensions $d \ge 3$, and superdiffusive in low dimensions $d \in \{1, 2\}$. 

It is in the case $d \ge 3$ where the most is known. Diffusive behaviour was rigorously shown in \cite{HorvTothVeto12_DiffusiveLimits}, and a central limit theorem for the scaled motion derived. 
More precisely, the authors prove that for every given $t\geq 0$, as $\eps$ goes to $0$,
$X^\varepsilon_t \eqdef \varepsilon X_{t/\varepsilon^2}$ converges to a Gaussian 
random variable whose variance $\sigma^2 \ge 1$ is only given implicitly. Even though not explicitly verified 
in the above-mentioned reference, we believe that the variational method of~\cite{TothValko12_SuperdiffusiveBounds} 
can be used to show 
that $\sigma^2 > 1$, implying that the self-interaction term still has an influence on the scaling limit.

In dimension $d=1$, bounds on the superdiffusivity of SRBP are given in \cite{TarrTothValk12_DiffusivityBounds}, 
but a rigorous proof of the exact rates remain open. Moreover, non-Gaussian fluctuations are expected at large scales: 
it is conjectured that under the correct superdiffusive scaling, $\varepsilon^{4/3}X_{t/\varepsilon^2}$, 
the process will converge to the true self-repelling motion \cite{TothWerner98_TrueSelfrepelling}. For the SRBP 
the conjecture is not yet settled but there are a few results in this direction for TSAW 
(see e.g.~\cite{Toth95_TrueSelfAvoiding,TothVeto11_ContinuousTime,NewmanRavishan06_ConvergenceToth}).

Dimension $d = 2$ is the critical dimension as~\eqref{eq:38} is formally scale invariant, 
and very little is known both for the TSAW and the SRBP. 
In \cite{TothValko12_SuperdiffusiveBounds}, superdiffusive bounds are established 
to the effect of $t \log\log t \lesssim \bbE[|X_t|^2] \lesssim t \log t$, which is far from identifying 
the precise value of $\beta$ in~\eqref{eq:36}. In fact, even in the physics literature there is no consensus over 
what $\beta$ should be~\cite{ObukhovPeliti83_RenormalisationTrue, PelitiPietrone87_RandomWalks}, 
although the argument outlined in the appendix of \cite{TothValko12_SuperdiffusiveBounds} 
(which we find most convincing) predicts $\beta = \frac12$. 
\medskip

The influence of the dimension $d$ can already be seen at the level of \eqref{e:SRBP1}. 
To wit, the diffusively rescaled process $(X^\varepsilon_t)_{t \ge 0}$ satisfies 
\begin{equation}
\label{eq:39}
dX^\varepsilon_t = dB^\varepsilon_t - \gamma^2 \varepsilon^{d-2} \left( \int_0^t \nabla V^\varepsilon (X^\varepsilon_t-X^\varepsilon_s) \dd s \right) \dd t, \qquad X^\varepsilon_0 = 0
\end{equation}
where $B^\varepsilon_t \eqdef \varepsilon B_{t/\varepsilon^2}$ is simply another Brownian motion, 
and $V^\varepsilon(x) \eqdef \varepsilon^{-d} V(\varepsilon^{-1} x)$ corresponds to the ``sharpening'' of the function $V$. 
This calculation reveals $d = 2$ as the critical dimension: the self-interaction term is scale invariant, 
and naively appears to be living at the same scale as the noise $(B^\varepsilon_t)_{t \ge 0}$. 
Let us also stress that, since $V^\varepsilon \rightarrow \delta_0$ as $\varepsilon\rightarrow0$, 
if $X^\varepsilon_t$ can be shown to converge, 
then the limiting object would be a natural candidate solution to \eqref{eq:38}. 

Of course, the superdiffusivity results of \cite{HorvTothVeto12_DiffusiveLimits} imply 
that the process will not converge under diffusive scaling. To tame the polymer's growth due to the self-interaction, 
we consider the so-called \textit{weak coupling scaling} which amounts to diffusively rescaling $X$ as in~\eqref{eq:39},
but simultaneously tuning down the coupling constant with $\eps>0$ as 
$\gamma = \gamma(\varepsilon)\sim1/\sqrt{\log|\eps|}$ (see~\eqref{eq:coup} below for the precise definition).  
In this context, the choice of $\gamma$ balances the logarithmic blow up of the Green's function 
in dimension $d = 2$, see Lemma \ref{lem:5}, and the same choice of coupling has been used 
in a whole host of other problems \cite{CSZ1, CSZ, Gu2020, CannErhaToni23_WeakCoupling, DG, CannGubiToni23_GaussianFluctuations, CaraSunZygo23_Critical2d}.  

The main result of this paper is that, in the weak coupling regime, the SRBP behaves diffusively and satisfies 
an (annealed) invariance principle, with a limiting Brownian motion having an {\it explicit} variance $\varsigma^2 > 1$. 
Moreover, the way in which $\varsigma^2$ depends on the coupling constant $\gamma^2$ is consistent with $t \sqrt{\log t}$ superdiffusivity, thus providing compelling evidence that the logarithmic correction of \eqref{eq:36} is $\beta=\half$. 

The conditions exploited in dimensions $d \ge 3$ break down for $d = 2$, and as such we must introduce new techniques. In fact, the classical Kipnis-Varadhan theory cannot cover $d=2$ because the solutions $(u^\lambda)$ of Poisson's equation do not have a limit point in the Sobolev space $\fH$, only a limiting norm (see the discussion following Theorem \ref{thm:1}). We expect that the methods exposed in the present work apply more generally, a first example being 
to the diffusion in the curl of the Gaussian free field (DCGFF) studied 
in~\cite{TothValko12_SuperdiffusiveBounds,CannHaunToni22_SqrtLog,ChatzigeorgiouEtAl23_GaussianFreefield}, 
albeit in the weak coupling regime (see Section \ref{a:nuisance}). In particular, let us point out that 
the non-Markovianity of the SRBP makes it unclear whether the techniques 
of~\cite{ChatzigeorgiouEtAl23_GaussianFreefield} can be applied to the present context. 

\subsection{The model and main results}
\label{sec:model-main-results}

Let us first state the precise assumption, which will be in place hereafter, on the mollifying function $V$ in~\eqref{e:SRBP1}. 

\begin{assumption}\label{a:V}
The function $V \colon \bbR^{d} \rightarrow \bbR$ is a smooth function, decaying at infinity faster than any 
polynomial and such that $\int V(x) \dd x=1$. 
Furthermore, $V$ is rotationally invariant and positive semi-definite, i.e. 
for any matrix $U \in SO(2)$, $V = V \circ U^{-1}$, and 
for any $x_1,\dots,x_n \in \bbR^d$, the matrix $(V(x_i-x_j))_{i, j = 1}^n$ is positive semi-definite. 
\end{assumption}

Let $(E, \calE, (\calE_t)_{t \ge 0}, \bbP)$ be a filtered probability space and $(B_t)_{t \ge 0}$ a 
Brownian motion on it. We define the self-repelling Brownian Polymer 
as the unique solution of the SDE
\begin{equ}[eq:sde]
\dd X_t = \dd B_t - \gamma \omega(X_t) \dd t - \gamma^2 \Big( \int_0^t \nabla V (X_t-X_s) \dd s\Big) \dd t\,,\qquad X_0 = 0 
\end{equ}
where $\gamma>0$ is the coupling constant and $\omega\colon \R^d\to \R^d$ is a smooth gradient (i.e. rotation free) 
vector field which grows at most linearly at infinity. We leave the dependence of $(X_t)_{t\ge0}$ on $\coup$ and $\omega$ implicit.

The vector field $\omega$ plays the role of an environment and in what follows we will choose it at random, 
so that~\eqref{eq:sde} defines an SRBP in {\it random environment}. One may regard $\omega$ as a random non-zero initial condition for the local time field, $V*L_0 = \omega$. Upon setting 
$\Omega$ to be the space smooth gradient 
vector fields which grow at most linearly at infinity, endowed with the cylindrical $\sigma$-algebra $\cF$ 
(see \eqref{e:Omega} for its rigorous definition), 
we take the law $\pi$ of $\omega\in\Omega$ to be that of $\nabla \xi$ 
where $\xi\eqdef \sqrt{V}\ast\Phi$ for $\Phi$ a two-dimensional Gaussian Free Field (GFF)
and $\sqrt{V}$ such that $\sqrt{V}\ast\sqrt{V}=V$ 
(which is well-defined in view of the positive semi-definiteness in Assumption~\ref{a:V}). 
In other words, $\pi$ is the law of a centred Gaussian field whose covariance function is given by 
\begin{equ}[e:CovOmega]
\int\omega_i(x)\omega_j(y)\pi(\dd \omega)\eqdef -\partial_{ij}^2 V\ast G(x-y)\,\qquad x,y\in\R^d
\end{equ}
for $G$ the Green's function of the two-dimensional Laplacian, i.e. $G(x)\eqdef - (2 \pi)^{-1} \log |x|$. 

The reason to introduce the environment $\omega$ in the definition of the SRBP is mainly technical (see below),
and is standard in this context (see~\cite{TothValko12_SuperdiffusiveBounds, HorvTothVeto12_DiffusiveLimits, TarrTothValk12_DiffusivityBounds}). 
That said, we believe that the results stated below also hold 
for $\omega = 0$. 

In the present work, we focus on the case of $d=2$ and consider the large-scale behaviour of the SRBP in the so-called weak coupling regime. That is, for $\eps>0$, setting $X^\eps_t\eqdef \eps X_{t/\eps^2}$ to be the diffusively rescaled solution of~\eqref{eq:sde}, we choose the coupling constant $\gamma$ to be given by 
\begin{equ}[eq:coup]
\gamma=\gamma(\eps)\eqdef \frac{\alpha}{\sqrt{\log(1+\varepsilon^{-2})}}
\end{equ}
where $\alpha$ is a strictly positive constant. 

We define the annealed measure $\bfP \eqdef \pi \otimes \bbP$ on the product space $(\bm{\Omega}, \bm{\calF}) \eqdef (\Omega \times E, \calF \otimes \calE)$. We are now ready to state the first main result of the present work.

\begin{theorem}
\label{thm:1}
Under the weak coupling in \eqref{eq:coup}, the finite dimensional distributions of $(X^\varepsilon_t)_{t\ge0}$ 
converge in probability with respect to $\pi$ to those of an isotropic Brownian motion with diffusivity 
$\varsigma^2(\alpha)\eqdef 1+\sigma^2(\alpha)$, where $\sigma^2(\alpha)$ is given explicitly by
\begin{equation}
\label{eq:87}
\sigma^2(\alpha) \eqdef \sqrt{4 \pi \alpha^2 + 1} - 1.
\end{equation}
Moreover, for $T\ge0$, with respect to the annealed measure $\bfP$, we have an invariance principle on $C([0, T], \bbR^2)$
\begin{equ}
 (X^\varepsilon_t)_{t \in [0, T]} \overset{d}{\rightarrow} \varsigma(\alpha)(W_t)_{t \ge 0}
 \end{equ}
where $(W_t)_{t \ge 0}$ is a standard two-dimensional Brownian motion.
\end{theorem}

\begin{remark}
By \textit{convergence in probability with respect to $\pi$} in the statement of Theorem \ref{thm:1}, it is meant that for all $\theta_1, ..., \theta_n \in \bbR^2$, and for all times $0 \le t_0 \le ... \le t_n$, it holds that as $\varepsilon \rightarrow 0$
\begin{equation}
\label{eq:AR5}
\int \Big| \bbE\Big[e^{\iota\sum_{k=1}^n \theta_k \cdot (X^\varepsilon_{t_k}-X^\varepsilon_{t_{k-1}})} \Big] - e^{-\half \varsigma^2(\alpha)\sum_{k=1}^n |\theta_k|^2 (t_k - t_{k-1})} \Big| \pi(\dd \omega)  \rightarrow 0\,.
\end{equation}
Consequently, the convergence of the marginals takes place in distribution with respect to the annealed measure $\bfP$. 
\end{remark}

We briefly discuss the significance of this result and its proof. 
In contrast to the $d \ge 3$ setting~\cite{HorvTothVeto12_DiffusiveLimits}, Theorem \ref{thm:1} yields 
an explicit expression for the limiting diffusivity. 
Since $\sigma^2(\alpha)>0$ for all $\alpha > 0$, the choice of weak coupling in~\eqref{eq:coup} is meaningful because, 
although $\gamma(\varepsilon) \rightarrow 0$ as $\varepsilon \rightarrow 0$, 
the self-interaction term does not vanish but has a non-trivial influence on the limit process. 
Specifically, as may be seen in the proof of Theorem \ref{thm:3}, in the limiting diffusivity $\varsigma^2 = 1 + \sigma^2(\alpha)$, 
the unit comes from the original Brownian term in \eqref{eq:sde}, 
while the additional diffusivity $\sigma^2(\alpha)$ is a consequence of the drift. 

The form taken by $\sigma^2(\alpha)$ is also interesting, 
most notably because of what it suggests regarding the logarithmic correction for the superdiffusivity in~\eqref{eq:36}. 
Indeed, formally ``undoing'' the weak coupling by setting $\alpha =\alpha(\eps)= \sqrt{\log(1+\varepsilon^{-2})}$ 
(corresponding to the so-called \textit{strong coupling} regime $\coup=O(1)$) and substituting it into~\eqref{eq:87} 
gives $\sigma^2(\alpha) \sim \sqrt{|\log \varepsilon|}$. 
This suggests that for large times $t$, $\bbE[|X_t|^2] \approx t \sqrt{\log t}$, 
providing strong evidence for the conjectured rate of $\beta = 1/2$. 

We emphasise that the square root appearing in the limiting diffusivity $\sigma^2(\alpha)$ is not ``the same'' square root appearing in the definition of $\gamma(\varepsilon)$ in~\eqref{eq:coup}. 
Generally speaking, one expects the choice of $\gamma(\varepsilon)$ to be somewhat universal, 
whereas the limiting diffusivity $\sigma^2(\alpha)$ to be problem specific. 
For example, in the recent work \cite{CannGubiToni23_GaussianFluctuations}, 
the coupling constant is chosen as in \eqref{eq:coup}, 
but the limiting diffusivity obeys a $2/3$-power law, 
which is consistent with the $\log^{2/3}(t)$ result of \cite{Yau04_LogtLaw}. 
We believe our methods would deliver the same $2/3$-power law in the case of the anisotropic version of the SRBP, 
as conjectured in \cite{TothValko12_SuperdiffusiveBounds}.

Even though the present setting is very different from that of a diffusion in a random environment 
(the SRBP is not even Markovian) the techniques we exploit have a similar flavour. 
We introduce a process $(\eta_t)_{t \ge 0}$, referred to as the \textit{environment seen by the particle} 
(see \cite[Chapter 9]{KomoLandOlla12_FluctuationsMarkov} or \cite{HorvTothVeto12_DiffusiveLimits}). 
Its state space is $\Omega$, 
and the process is defined, for $t\ge0$ and $x \in \bbR^2$, according to
\begin{equation}
\label{eq:83}
\eta_t(x) = \omega(x+X_t) + \coup \left(\int_0^t \nabla V (x+X_t-X_s) \dd s\right)\,.
\end{equation}
The advantage of working with $(\eta_t)_{t \ge 0}$ is that it is a Markov process on $\Omega$,  
which has $\pi(d \omega)$ in~\eqref{e:CovOmega} as an invariant measure (see Lemma \ref{lem:25}). 
The reason for introducing $\omega$ in \eqref{eq:83} is that it gives rise to the initial condition $\eta_0 = \omega$, 
and therefore with respect to the annealed measure $\bfP$, $(\eta_t)_{t \ge 0}$ is stationary.  
The environment process plays a central role in the proof of Theorem \ref{thm:1} in view of its connection 
to the SRBP, given by
\begin{equation}
\label{eq:30}
X_t = B_t - \coup \int_0^t \eta_s(0)\dd s\,.
\end{equation}
In other words, the drift term $\coup \int_0^t \eta_s(0)\dd s$ is an additive functional of the environment, 
so that we are in the setting of the Kipnis-Varadhan theory \cite{KipnisVaradhan86_CentralLimit} 
of martingale approximation (see \cite{KomoLandOlla12_FluctuationsMarkov} for a comprehensive study). 
However, our application goes beyond the classical case (see Section \ref{sec:no-weak-convergent}). 
A self-contained exposition of the way in which we use martingale approximation and 
of where and how we generalise their technique is given in Theorem \ref{thm:3}.
\medskip

Our second main theorem focuses on the environment process itself and determines its large scale behaviour. 
To the best of the authors' knowledge, this is the first time a result of this type is derived and its interest 
goes beyond the specific setting of this paper, also in view of its relation to singular stochastic PDEs. 

Indeed, in the formal setting of SRBP in \eqref{eq:38}, with $V$ replaced by the Dirac delta at $0$, 
$\eta$ solves
\begin{equation}
\label{eq:21}
\dd \eta_t = \nabla \eta_t \circ \dd B_t - \coup \left( \nabla \eta_t \cdot \eta_t(0) - \nabla \delta_0\right) \dd t
\end{equation}
where $\circ \dd B_t$ denotes Stratonovich integration. Due to the presence of $\nabla\delta_0$, 
one expects $t\mapsto \eta_t$ to be distribution valued, 
so that neither the quadratic term nor the point-evaluation at the right hand side make sense. 
Morally speaking, our goal is to give a meaning to~\eqref{eq:21} in the stationary weak coupling regime. 
Indeed, the process $(\eta_t)_{t \ge 0}$ in~\eqref{eq:83} is well-defined and satisfies 
a regularised version of \eqref{eq:21} (see~\eqref{eq:117} below). Analogous to the SRBP case, 
passing to the diffusive scaling corresponds to removing the regularisation so that 
the equation solved by the scaled process formally approaches~\eqref{eq:21}.  

To that end, let $(\eta^\varepsilon_t)_{t\ge0}$ be the $\Omega$-valued process defined according to $\eta^\varepsilon_t(x) = \varepsilon^{-1}\eta_{t/\varepsilon^2}(\varepsilon^{-1}x)$. Since we expect the limit to be only a 
generalised function, 
we embed $\Omega$ into a Hilbert space $H^*$ of distributions, 
whose precise definition will be given in Section \ref{sec:Env}, with canonical embedding given by 
$\omega \mapsto \big(g \mapsto  \int_{\bbR^2} \omega(x)\cdot g(x)\dd x \big)$. 
We shall reserve bold symbols for $H^*$-valued elements, e.g. $\bm{\eta}^\varepsilon_t$ denotes the macroscopic 
environment $\eta^\varepsilon_t$ when viewed as a distribution under the canonical embedding.

The limiting process for $(\bm{\eta}^\varepsilon_t)_{t\ge0}$ is the solution $(\bm{\etaBar}_t)_{t \ge 0}$ of the following stochastic linear transport equation (SLTE)
\begin{equation}
\label{eq:131}
\dd \bm{\etaBar}_t = \varsigma \nabla \bm{\etaBar}_t \circ \dd \BBar_t , \quad \bm{\etaBar}_0 \sim \bm{\piBar},
\end{equation}
where $\varsigma>1$ is defined as in Theorem~\ref{thm:1}, and $\bm{\piBar}$ is the law of $\nabla \Phi$, 
for $\Phi$ a GFF (see Definition \ref{def:5}), 
and is sampled independently from the Brownian motion $(\BBar_t)_{t\ge0}$. 
The equation~\eqref{eq:131} is to be interpreted as an infinite dimensional stochastic equation 
in the Hilbert space $H^*$, in the sense of \cite{DaPratoZabczyk14_StochasticEquations}. Even though 
we believe it to be classical,  we will discuss existence and uniqueness in Appendix~\ref{sec:well-posedn-mart}.
We refer to the law of $(\bm{\etaBar}_t)_{t \ge 0}$ as {\it the Brownian transportation of the gradient of the 
GFF}. Such terminology is justified as the process 
$(\bm{\etaBar}_t)_{t \in [0,T]}$ is formally given by
\begin{equ}[e:BTGGFF]
\bm{\etaBar}_t(x) = \bm{\etaBar}_0(x + \sigma \BBar_t)\,, 
\end{equ}
an expression which will be made sense of in Lemma \ref{lem:12} below. 

We are ready to state the second main result of this paper.

\begin{theorem}
\label{thm:2}
Let $\alpha,T>0$. Under the weak coupling in \eqref{eq:coup}, the scaled environment seen by the particle process 
$\eta^\eps$ converges in law on $C([0,T],H^*)$ to the solution of $\bm{\etaBar}$ of~\eqref{eq:131}, 
with $\varsigma$ as in Theorem~\ref{thm:1}. 
\end{theorem}

In view of the definition given in \eqref{eq:83}, and the convergence result of Theorem \ref{thm:1}, 
it is not surprising to see $\bm{\etaBar}_t$ as the large scale description of $\bm{\eta}_t$. 
As shown in the proof of Theorem \ref{thm:2},  
the enhanced diffusivity $\sigma^2(\alpha)$ is produced by the singular part of  \eqref{eq:21}. 
While our two results are consistent, it is not clear how to go from one to the other. Indeed,
although at macroscopic scale, the SRBP is essentially an observable of the environment 
$\eta^\varepsilon$, this observable doesn't make sense for the limit $\bm{\etaBar}_t$, 
because it would involve evaluating a distribution at a point (see \eqref{eq:30}).

At last let us mention that, even though Theorem \ref{thm:2} is very different from Theorem \ref{thm:1}, 
we prove them both with similar techniques. On the one hand, martingales may be used to approximate the SRBP, 
and on the other they may be related to the environment through 
the martingale problem associated to the limiting SPDE. 
It should be possible to use our methods to derive the limiting environment also for other models, 
for example the case of diffusions in divergence free Gaussian vector fields at and above the critical dimension. 

\subsection{Structure of the paper}
We divide the proof of Theorem \ref{thm:1} in two: firstly we deal with the martingale approximation, 
and then we perform a detailed analysis on the generator $\calL$ of $(\eta_t)_{t \ge 0}$. 
Section \ref{sec:envir-idea-proof} is devoted to the first part, 
which is wrapped up in Theorem \ref{thm:3}. Section \ref{sec:analysis-generator} and \ref{sec:appr-solv-gener} 
instead deal with the second:  
we analyse $\calL$ and derive estimates which are needed for both Theorems \ref{thm:1} and \ref{thm:2}. 
In Sections \ref{sec:invariance-principle} and~\ref{sec:Env}, we exploit such estimates and determine the invariance 
principle for the SRBP (Section~\ref{sec:invariance-principle}) 
and the scaling limit for the environment process in the second (Section~\ref{sec:Env}). 
At last, we have appendices. In Appendix \ref{sec:integral-estimates}, we collect bounds and integral computations 
which are used throughout. In Appendix \ref{sec:mart-centr-limit}, we give a triangular central limit theorem 
for martingales, which is crucial in the proof of Theorem \ref{thm:3}. 
Appendix \ref{sec:well-posedn-mart} is dedicated to the well-posedness statement 
for the stochastic linear transport equation. Finally, in Appendix \ref{a:nuisance} we provide some 
insight as to why we were forced to separately consider the so-called nuisance region together with 
a short comparison between the SRBP and the DCGFF.

\subsection*{Notation and Wiener space analysis}

Let us introduce some notation that we shall be using throughout the article. 
For elements $x_1, ..., x_n\in\R^d$, we write $x_{1:n}$ for the vector $x_{1:n} \eqdef (x_1, ..., x_n)$ and 
we extend the notation to any ordered set of indices $A \subset \{1, ..., n\}$, so that, for example, 
$x_{1:n \setminus i} \eqdef (x_1, ..., \cancel{x_i}, ..., x_n)$, $i\in\{1,\dots,n\}$. 
For $A\subset \{1,\dots,n\}$, we denote by $x_{[A]}$ the sum $ \sum_{i\in A} x_i$. 
Moreover, for $x \in \bbR^d$ and $j \in  \{1, ..., d\}$, we write $e_jx$ for the $j$'th coordinate, i.e. 
$e_jx\eqdef e_j\cdot x$ for $\cdot$ the usual scalar product in $\R^d$ and $e_j$ the $j$'th element of the standard 
basis of $\R^d$. 

For a Schwartz function $f \colon \bbR^d \rightarrow \bbR$ we define the Fourier transform $f$ by
\begin{equ}
\fHat(p) \eqdef \frac{1}{(2\pi)^{d/2}}  \int_{\bbR^d} e^{ -\iota p \cdot x } f(x) \dd x\,, 
\end{equ}
for all $p\in\R^d$. 
\medskip

We now rigorously define the random environment $\omega$ for the SRBP.
Let $\Omega$ be the space of smooth vector fields on $\R^2$ of sub-polynomial growth, i.e. 
\begin{equ}[e:Omega]
\Omega \eqdef \left\{ \omega \in C^\infty(\bbR^2, \bbR^2) : \rot \omega = 0, \, \lVert \omega \rVert_{k,i,m}  < \infty, \, \forall k,i,m\right\}
\end{equ}
where $\rot \omega\eqdef \nabla\times \omega$, and the indices $k,i,m$ respectively range over $\N^2,\,\{1,2\},\,\N$, and 
the seminorms above are defined as $ \lVert \omega \rVert_{k,i,m} \eqdef \sup_{x} (1+|x|)^{-1/m} | \partial^k \omega_i(x) | $. 
The topology generated by these seminorms turns $\Omega$ into a Fr\'{e}chet space which we endow  
with the cylindrical $\sigma$-algebra $\calF \eqdef \sigma(\omega(x) : x \in \bbR^2)$. 
Let $\pi$ be the law of the \textit{gradient of the smoothed out Gaussian free field}, i.e. 
the probability measure on $(\Omega,\cF)$ 
under which $(\omega(x) : x \in \bbR^2)$ is a Gaussian process with covariance 
given in \eqref{e:CovOmega} (see~\cite{HorvTothVeto12_DiffusiveLimits} for more on $\pi$). 

According to~\cite[Theorem 2.6]{Janson97_GaussianHilbert}, $L^2(\pi) = \oplus_{n=0}^\infty \calH_n$, 
where the subspaces $\{\calH_n\}_n$ are mutually orthogonal and, for every $n$, 
$\cH_n$ is the so-called $n$-th homogeneous Wiener chaos, while 
$\cH_{\leq n}\eqdef \oplus_{j=1}^n \calH_j$ is the $n$-th inhomogeneous Wiener chaos modulo $\cH_0$.  
The chaos $\cH_n$ is the closure in $L^2(\pi)$ 
of elements of the form 
\begin{equ}[eq:110]
X = \sum_{i_1, ..., i_n \in \{1, 2\}} \int_{(\bbR^2)^n} f_{i_{1:n}}(x_{1:n}) : \omega_{i_1}(x_1) \, \dots \, \omega_{i_n}(x_n) : \dd x_{1:n}
\end{equ}
where $: \omega_{i_1}(x_1) \, \dots \, \omega_{i_n}(x_n) : $ denotes the Wick product of the Gaussian variables $\omega_{i_1}(x_1), \,\dots,\, \omega_{i_n}(x_n)$ and  $\{f_{i_1, ..., i_n} : i_1, ..., i_n \in \{1, 2\}\}$ is a family of Schwartz functions such that 
$f_{i_{\sigma(1:n)}}(x_{1:n}) = f_{i_{1:n}}(x_{\sigma^{-1}(1:n)})$ for all permutation $\sigma\in S_n$. 

In order to work with a more manageable version of the Wiener chaoses, notice that 
the definition of the covariance in~\eqref{e:CovOmega}, Wick's rule~\cite[Theorem 3.12]{Janson97_GaussianHilbert}
and Plancherel's identity ensure that if $X,Y\in\cH_n$ admit the representation in~\eqref{eq:110} 
with respect to the families of kernels $\{f_{i_1, ..., i_n} : i_1, ..., i_n \in \{1, 2\}\}$ and 
$\{g_{i_1, ..., i_n} : i_1, ..., i_n \in \{1, 2\}\}$, then 
\begin{equ}[e:ScalarProd]
\int X(\omega) Y(\omega) \pi(\omega) = \int \frakf(p_{1:n}) \frakg(p_{1:n})  \mu_n(\dd p_{1:n} )
\end{equ}
where $\frakf$ (and similarly $\frakg$) is given by\footnote{so that $\frakf$ is the Fourier transform of the function $(-1)^n \sum_{i_1,\dots i_n}\partial_{i_1}\dots\partial_{i_n}f_{i_1, ..., i_n}$ but we omit the ``hat'' denoting Fourier transforms to lighten the notation.}
\begin{equation}
\label{eq:42}
\frakf(p_{1:n}) = (-\iota)^n \sum_{i_1, ..., i_n \in \{1, 2\}} e_{i_1} p_1 ... e_{i_n} p_n \hat f_{i_1, ..., i_n}(p_{1:n})
\end{equation}
and $\mu_n$ is defined according to
\begin{equ}[e:mun]
\ffock{n}(\dd p_{1:n}) \eqdef n! \left( \prod_{i=1}^n \frac{\VHat(p_i)}{|p_i|^2} \right) \dd p_{1:n}\,.
\end{equ}
We define $\langle\cdot,\cdot\rangle$ to be the scalar product defined by the right hand side of~\eqref{e:ScalarProd}, and we write $\lVert \cdot \rVert$ for the corresponding norm. What~\eqref{e:ScalarProd} shows is that the space $\cH_n$ is isometric 
to the Fock space $\fock_n$ given by 
\begin{equ}[e:Fock]
\fock_n \eqdef \cl\{ \frakf \in L^2_\bbC(\ffock{n}) \colon \frakf(p_{1:n}) = \frakf(p_{\sigma(1:n)}), \frakf(-p) = \overline{\frakf(p)} \} 
\end{equ}
where the closure is taken with respect to $\lVert \cdot \rVert$. We shall set $\fock_{\le n} \eqdef \oplus_{j = 1}^n \fock_j$ so as to omit the constants $\fock_0$.

Since most of our analysis will be based on operators acting on $\fock\eqdef \oplus_n \fock_n$, 
let us single out a class of them which will play an important role for us. 

\begin{definition}\label{def:DO}
An operator $\ST$ on $\fock$ is said to be {\it diagonal} with multiplier $\tau$ 
if, for all $n\in\N$, $\ST(\fock_n)\subset\fock_n$ and $\tau\colon \R^2\to\R$ is such that for all $\psi\in\fock_n$
\begin{equ}
\calT\psi(p_{1:n})=\tau(p_{[1:n]})\psi(p_{1:n})\,,\qquad  p_{1},\dots,p_n\in\R^{2}\,.
\end{equ}
We further say that the operator is  non-negative when $\tau \ge 0$. 
\end{definition}

We write $a \lesssim b$ or $a=O(b)$ to mean that there exists a constant $C>0$, such that $a \le C b$. 
If we want to emphasise the dependence of $C$ on a specific quantity $v$, we write $a\lesssim_v b$ or $a= O_v(b)$.

\section{From the environment seen by the particle to the invariance principle}
\label{sec:envir-idea-proof}

The goal of this section is to identify a set of conditions which imply the central limit theorem as well as 
the invariance principle for the SRBP, as stated in Theorem~\ref{thm:1}. 
These conditions are given in Theorem~\ref{thm:3} and are different (and for most part 
{\it strictly weaker}) than those formulated 
in~\cite[Section 2.6]{KomoLandOlla12_FluctuationsMarkov}
but, as we will see, still sufficient for the result to hold. 
In order to formulate them, we will first present some preliminary results concerning the SRBP  in~\eqref{eq:sde} and 
introduce the generator associated to the environment seen by the particle process. 

\subsection{The SRBP and the environment seen by the particle}
\label{sec:srbp-envir-seen}

In the setting of Section~\ref{sec:model-main-results}, let $(E, \calE, (\calE_t)_{t \ge 0}, \bbP)$ 
be a filtered probability space and $(B_t)_{t \ge 0}$ a Brownian motion on it. 
Let us recall that, for $\gamma>0$, the SRBP in the environment $\omega\in\Omega$ is given by 
\begin{equ}[eq:SRBP2]
\dd X_t = \dd B_t - \gamma \omega(X_t) \dd t - \gamma^2 \Big( \int_0^t \nabla V (X_t-X_s) \dd s\Big) \dd t\,,\qquad X_0 = 0\,. 
\end{equ}
It is classical to see that for any smooth vector-field $\omega$, the SDE is indeed well-posed. 
Associated to it and crucial in the study of its fluctuations, is  
the environment seen by the particle process $\eta$, whose definition was given in~\eqref{eq:83}, 
\begin{equ}[e:envNew]
\eta_t(x) = \omega(x+X_t) + \coup \left(\int_0^t \nabla V (x+X_t-X_s) \dd s\right)\,,\quad x\in\R^2\,,t\geq 0\,.
\end{equ}
Before detailing the connection between $\eta$ and $X$, let us summarise some of the properties of $\eta$
in the following lemmas. 

\begin{lemma}\label{lem:25}
The environment process $(\eta_t)_{t\ge0}$ in~\eqref{e:envNew} is an $\Omega$-valued 
Markov process which solves the vector valued SPDE
\begin{equs}
\dd \eta_t(x) &= \tfrac{1}{2} \Delta \eta_t(x) \dd t - \coup \left( \sum_{i = 1}^2 \partial_i \eta_t(x)\eta_t^i(0) - \nabla V(x)\right) \dd t + \sum_{i = 1}^2 \partial_i \eta_t(x) \dd B_t^i\\
\eta_0(x)&=\omega(x)\,,\label{eq:117}
\end{equs}
for $x\in\R^2$ and $t\geq 0$. 
Furthermore, the law $\pi$ in~\eqref{e:CovOmega} is an invariant measure for $\eta$. 
\end{lemma}
\begin{proof}
The first part of the statement is an easy consequence of It\^o's formula  and a stochastic Fubini's theorem~\cite[Theorem 4.33]{DaPratoZabczyk14_StochasticEquations}. The second instead follows arguments similar to those in~\cite[Section 3.3]{HorvTothVeto12_DiffusiveLimits}.
\end{proof}

\begin{remark}
\label{rmk:1}
Let us briefly comment on the right hand side of the SPDE in~\eqref{eq:117}. 
The Laplacian is simply an It\^{o} correction term. It is best considered grouped with the noise as, together, 
they produce the transport term in Stratonovich form (see \eqref{eq:21}) which comes from the Brownian shift. 
The nonlinear term in parenthesis combines the contribution from the drift of $(X_t)_{t \ge 0}$ 
with the growth of the dynamic profile from the occupation over an infinitesimal time step. 
As can be seen by a direct computation, the correction $-\nabla V$ corresponds to 
the Wick renormalization of the nonlinearity with respect to the measure $\pi$. The SPDE is singular because the solution is distribution valued at large scales (see Section \ref{sec:Env}), in which case the RHS becomes ill defined, both due to the nonlinearity, and because we evaluate the distribution at the origin. 
\end{remark}

Since the process $\eta$ is Markov, it has an infinitesimal generator $\calL$ whose action on cylinder functions 
can be obtained by applying It\^{o}'s formula, using~\eqref{eq:117} and singling out the drift (and martingale) part. 
In the next lemma, we see how this action translates to an action on the $L^2$ space associated to $\pi$, 
and more precisely on the Fock space $\fock$ in~\eqref{e:Fock}. 

\begin{lemma}
\label{lem:26}
The generator $\calL$ of the environment seen by the particle process $\eta$ in~\eqref{e:envNew}, 
viewed as an unbounded operator on $L^2(\pi)\cong\fock$, can be decomposed as 
$\calL = -\calS + (\calA_+ + \calA_-)$, where, for any $n\geq 1$, $\cS$, $\calA_\pm$ respectively map $\fock_n$ 
to $\fock_n$, $\fock_{n\pm1}$. For $\psi\in\fock_n$, they are defined as 
\begin{equs}[eq:ops]
\calS \psi(p_{1:n}) &= \thalf |p_{[1:n]}|^2 \psi(p_{1:n})\,, \\
\calA_+ \psi(p_{1:n+1}) &= \frac{\coup}{n+1} \sum_{i=1}^{n+1} p_i \cdot p_{[1:(n+1)\setminus i]} \, \psi(p_{1:(n+1)\setminus i} )\,, \\
\calA_- \psi(p_{1:n-1}) &= \coup n\int\frac{\VHat(q)}{|q|^2} q \cdot p_{[1:n-1]} \psi(q,p_{1:n-1}) \dd q\,, \\
\end{equs}
and satisfy $\calS|_{\calH_0} = \calA_+|_{\calH_0} = \calA_-|_{\calH_0\oplus\calH_1} = 0$. 
The operator $\cS$ is self-adjoint while $(\calA_\pm)^* = - \calA_\mp$ so that, in particular, 
$\calA \eqdef \calA_+ + \calA_-$ is skew self-adjoint.  

Moreover, upon defining the operator $\nabla_i \colon \fock\to\fock$, $i=1,2$, to be the (Fourier transform of the) 
usual derivative operator 
in the $i$-th direction, i.e. $\nabla_i \psi(p_{1:n}) = \iota (e_ip_1+...+e_ip_n) \psi(p_{1:n})$ for  
$\psi\in\fock_n$, 
the following integration by parts formula holds
\begin{equ}[eq:IBP]
\sum_{i=1}^2 \int (\nabla_i u(\omega))^2 \pi(d\omega) = 2\langle u, \calS u\rangle\,.
\end{equ}
\end{lemma}
\begin{proof}
The statement follows Wiener chaos computations analogous to those performed 
in~\cite[Section 3.3]{HorvTothVeto12_DiffusiveLimits}. 
\end{proof}

\begin{remark}
There are two differences to note when comparing our setting with \cite{HorvTothVeto12_DiffusiveLimits}. 
First, in our case, although we do not write it explicitly, the generator $\calL=\calL^\coup$ depends on the coupling constant $\coup$. 
Second, for $d \ge 3$ it makes sense to exploit more significantly the gradient nature of the potential $\omega$, 
i.e. the fact that $\omega = \nabla \xi$, by choosing, as environment profile, the process 
$(\etaTil_t)_{t \ge 0}$ defined according to 
\begin{equ}
\etaTil_t(x) = \xi(x+X_t) + \coup \left( \int_0^t V (x+X_t-X_s) ds\right) 
\end{equ}
This is the approach of \cite{HorvTothVeto12_DiffusiveLimits} and corresponds to 
the \textit{potential as seen by the particle}, the relation with the above being $\eta_t = \nabla \etaTil_t$. 
Notice that we {\it cannot} take this approach as, even if smoothed out, for the field $\xi=\sqrt{V}\ast\Phi$ 
with $\Phi$ a two-dimensional GFF, pointwise evaluation is meaningless, i.e. $\tilde\eta$ is 
not well-defined in $d=2$. 
\end{remark}

As a consequence of the previous statements, with respect to the annealed measure $\bfP\eqdef\pi \otimes \bbP$ on 
the product space $(\bm{\Omega}, \bm{\calF}) \eqdef (\Omega \times E, \calF \otimes \calE)$, 
$\eta$ is a stationary Markov process corresponding to the natural filtration 
$(\bm{\calF}_t)_{t\ge0}$, where $\bm{\calF}_t = \sigma(\eta_s : s \le t)$. 
Moreover, as mentioned in the introduction, we can write the drift of the SRBP $X$ in~\eqref{eq:SRBP2} 
as an additive functional of $\eta$. More precisely, we have  
\begin{equ}[eq:AFM]
X_t = B_t - \coup \int_0^t f(\eta_s) \dd s
\end{equ}
for $f=(f_1,f_2)\colon \Omega\to\R^2$ given by 
\begin{equ}[e:fdef]
f_i(\omega) = \omega_i(0)\,,\qquad i=1,2\,,
\end{equ}
so that in particular, $f_i\in\calH_1$ and its kernel in $\fock_1$ equals $\frakf_i(p) = -\iota e_ip / (2 \pi)$. 

The representation in~\eqref{eq:AFM} is essential for our work. As a first application, we will see how to use it 
to show tightness of the SRBP under the weak coupling scaling.

\subsection{The It\^{o} trick and tightness}
\label{sec:mart-appr}
The main tool to determine tightness is the so-called It\^o trick, first appeared in~\cite{GubinelliJara13_RegularizationNoise}, 
and since then applied in a variety of other contexts. We recall here its statement adapted to the current setting. 
A complete proof of a slightly more general formulation is given in Lemma~\ref{lem:2}. 

\begin{lemma}[It\^{o} trick]\label{lem:1}
Let $n \in \bbN$ and $h \in \calH_{\le n}$ be an element of the $n$-th inhomogeneous chaos. 
Then, for any $T > 0$, $p \ge 1,\lambda>0$, it holds that 
\begin{equ}[eq:itotrick]
\bfE\Big[ \sup_{t \in [0,T]} \Big\lvert \int_0^t h(\eta_s) ds \Big\rvert^p \Big]^{\frac1p} \lesssim_{n, p} (T^{\half}+\lambda^{\half} T)\lVert (\lambda+\calS)^{-\half}h\rVert\,.
\end{equ}
Moreover, for $p = 2$, the estimate is uniform in $n \in \bbN$.
\end{lemma}

We are now ready to state and prove the main result of this subsection. 

\begin{proposition}\label{p:Tightness}
Let $(X_t)_{t\geq 0}$ be the SRBP, i.e. the solution to~\eqref{eq:SRBP2}. 
For $t\geq 0$, set $X^\eps_t\eqdef \eps  X_{t/\eps^2}$ and take $\coup=\coup(\eps)$ as in~\eqref{eq:coup}. 
Then, under the annealed measure $\bfP$, for any $T>0$ fixed, the sequence $\{(X^\eps_t)_{t \in [0,T]}\colon \eps\in(0,1)\}$ is tight 
in the space $C([0,T], \R^2)$. 
\end{proposition}

The main step in the proof is a control over the drift at the right hand side of~\eqref{eq:AFM}, which, 
in light of the It\^o trick, reduces to a regularity estimate of the functional $f$ in~\eqref{e:fdef}. 
As the next lemma states, this is precisely the point at which the weak coupling scaling enters our analysis.  

\begin{lemma}\label{lem:5}
Let $f$ be the functional in~\eqref{e:fdef}. For $\lambda>0$, let $\coup=\coup(\lambda)$ be defined as 
\begin{equ}[eq:Weak]
\coup = \frac{\alpha}{\sqrt{\log(1+\lambda^{-1})}}
\end{equ}
for $\alpha>0$. Then, for $i=1,2$, uniformly over $\lambda$, the following estimate holds
\begin{equ}[e:hone]
\lVert(\lambda + \calS)^{-\half}\coup f_i\rVert^2 \lesssim 1\,.
\end{equ}
\end{lemma} 
\begin{proof}
For $i=1,2$, by~\eqref{e:ScalarProd} and the definition of $f$ in~\eqref{e:fdef}, 
the norm at the left hand side of~\eqref{eq:Weak} equals
\begin{equ}
\lVert(\lambda + \calS)^{-\half}\coup f_i\rVert^2 = \int\frac{\VHat(p)}{|p|^2}\frac{|-\iota \coup e_ip|^2}{\lambda+\half|p|^2} \dd p \leq \coup^2 \int_{\bbR^2} \frac{\VHat(p)}{\lambda+\half|p|^2} \dd p
\end{equ}
We split the integral over the regions $|p|\leq 1$ and its complement. For the latter, 
since $\hat V$ is integrable, we obtain a bound of order $\gamma^2$. For the former, we have 
\begin{equ}[e:UBWeak]
\coup^2 \int_{|p| \le 1}\frac{\VHat(p)}{\lambda+\half|p|^2} \dd p \lesssim \coup^2 \int_0^1 \frac{r}{\lambda+\half r^2} \dd r = \coup^2 \log(1+\tfrac{1}{\lambda}) = \alpha^2\,,
\end{equ}
where in the last step we used the definition of $\gamma$ in~\eqref{eq:Weak}. 
\end{proof}

\begin{remark}\label{rem:H-1Fail}
Notice that the standard $H_{-1}$ bound in~\cite[Theorem 2.7]{KomoLandOlla12_FluctuationsMarkov} (see also \cite{KozmaToth17_CentralLimit}) does not hold in the present setting. Indeed, it is not hard to see that $\lVert\calS^{-\half}\coup f_i\rVert^2 = \infty$ for $i=1,2$.
This is precisely the reason why the coupling constant $\coup$ was chosen as in~\eqref{eq:Weak}. 
\end{remark}

\begin{proof}[of Proposition~\ref{p:Tightness}]
To prove the result, we will use Kolmogorov's criterion \cite[Theorem 23.7]{Kallenberg21_FoundationsModern}, which requires us to control the $p$-th moment of the increments 
$(X^\eps_t)_{t\geq 0}$ for some $p > 2$. By stationarity, it is enough to show that the $p$-th moment 
of $| X^\varepsilon_t |$ is controlled by $t^{p/2}$ for every $t\in[0,T]$ and some $p>2$. Notice that 
\begin{equ}[eq:Fdef]
\bfE | X^\varepsilon_t |^p\lesssim_p \bfE | B^\varepsilon_t|^p+\bfE \Big| \eps \int_0^{t/\eps^2} \coup f(\eta_s) \dd s \Big|^p=:\bfE | B^\eps_t|^p+\bfE | F^\eps_t |^p\,,
\end{equ}
where $B^\eps_t\eqdef \eps B_{t/\eps^2}$ and $F^\eps$ is defined via the left hand side. Now, 
the first term is clearly bounded by $t^{p/2}$. For the other, upon choosing $\lambda=\eps^2$, 
the It\^o trick in Lemma~\ref{lem:1} gives
\begin{equ}
\bfE | F^\eps_t |^p\lesssim_{T,p}  t^{p/2} \lVert(\eps^2 + \calS)^{-\half}\coup f_i\rVert^p\lesssim t^{p/2}\,,
\end{equ}
 where the last step follows by Lemma~\ref{lem:5}. 
\end{proof}

\subsection{Beyond the classical conditions: Theorem \ref{thm:3}}
\label{sec:BCC}

The representation in~\eqref{eq:AFM} suggests that the central limit theorem and the invariance principle 
for the SRBP in~\eqref{eq:SRBP2} stated in Theorem~\ref{thm:1} follow once we show 
they hold for the (additive) functional of the environment $\eta$ in~\eqref{e:envNew} given by 
\begin{equ}[eq:AFM2]
 \coup \int_0^t f(\eta_s) \dd s\,.
\end{equ}
There is a large body of work on the Kipnis-Varadhan program for obtaining such type of results  
for quantities as those in~\eqref{eq:AFM2} (see \cite{KomoLandOlla12_FluctuationsMarkov} for an overview). 
Heuristically, the idea is to consider the observable given by the solution to the Poisson equation 
$-\calL u = \coup f$ and note that, by It\^o's formula we have
\begin{equ}[e:Dynkin]
u(\eta_t)-u(\eta_0)-\int_0^t \cL u(\eta_s)\dd s= M_t(u)
\end{equ}
where $(M_t(u))_t$ is the martingale given by 
\begin{equ}[e:Mart]
M_t(u) = \sum_{i=1}^2 \int_0^t \nabla_i u(\eta_s) \dd B^i_s\,,\qquad \langle M(u)\rangle_t = \sum_{i=1}^2 \int_0^t (\nabla_i u(\eta_s))^2 \dd s
\end{equ}
and $\nabla_i$ is the operator defined in Lemma~\ref{lem:26}. 
Since $u$ solves the Poisson equation,~\eqref{e:Dynkin} provides an alternative representation 
for~\eqref{eq:AFM2} in terms of boundary terms, which one expects to be negligible in the diffusive rescaling, 
and the martingale $M(u)$. In other words, such representation reduces the proof of Theorem~\ref{thm:1} 
to that of proving an analogous statement for the martingale $M(u)$. 

The problem with the above strategy is that it is hard to determine a solution for the Poisson equation since 
wthe generator $\cL$ is an unbounded operator which is not self-adjoint (and is not invertible!). Therefore, instead, it is natural to consider the above argument with $u^\lambda$ in place of $u$, where $u^\lambda$ solves the resolvent equation $(\lambda-\calL)u^\lambda = \coup f$. 
As explored in \cite[Chapter 2]{KomoLandOlla12_FluctuationsMarkov}, one wishes to identify suitable conditions for the family $u^\lambda$ in the limit $\lambda\to 0$ such that this approximation argument succeeds. In the case of the SRBP in dimension 
$d\geq 3$, the condition exploited is the so-called {\it graded sector condition} of \cite{SethVaraYau00_DiffusiveLimit} which fails at criticality, i.e. for $d=2$ (see Section~\ref{sec:grad-sect-cond}, below for its precise definition).
The novelty of our approach is to introduce an alternative family of approximate solutions
 and identify a new set of conditions which still ensure that the invariance principle holds.  
 
While the construction of the family of approximate solutions is detailed in the following sections and 
represents the bulk of the paper, in the next theorem we state the conditions mentioned above and prove 
that indeed they are sufficient for Theorem~\ref{thm:1} to hold. 

\begin{theorem}
\label{thm:3}
Let $\cV=\{v^{\lambda,n} \colon \lambda \in (0,1), n \in \bbN\}\subset \text{dom}(\calL) $ be a family of observables 
such that for every $n \in \bbN$, $\{v^{\lambda,n} \colon \lambda \in (0,1)\} \subset \fock_{\le n}$. 
For $\lambda\in(0,1)$ and $n\in\N$, define $\sigma_{\lambda,n}^2 > 0$ and the random variable $q^{\lambda,n}$
according to 
\begin{equ}[e:gsigma]
q^{\lambda,n}(\omega) \eqdef \sum_{i=1}^2 (\nabla_i v^{\lambda,n}(\omega))^2\,,\qquad \sigma_{\lambda,n}^2\eqdef \bbE_\pi[ q^{\lambda,n} ] = 2\lVert \calS^\half v^{\lambda,n} \rVert^2\,.
\end{equ}
If the family $\cV$ satisfies 
\begin{equs}
\lim_{\lambda\rightarrow0} \lambda\lVert v^{\lambda,n} \rVert^2 &= 0, \qquad \forall n \in \bbN, \label{eq:L2}\\
\lim_{n\rightarrow\infty}\limsup_{\lambda\rightarrow0}\lVert (\lambda+\calS)^{-\half} \big[(\lambda-\calL)v^{\lambda,n} - \coup f_1 \big]\rVert^2 &= 0\,,\label{eq:ApproxH1}
\end{equs}
as well as
\begin{equs}
\lim_{n\rightarrow\infty}\limsup_{\lambda\rightarrow0} | \sigma_{\lambda,n}^2 - \sigma^2(\alpha) | &=0\,, \label{eq:ApproxMean}\\
\lim_{\lambda\rightarrow0} \lambda \lVert (\lambda+\calS)^{-\half}\big(q^{\lambda,n} - \sigma_{\lambda,n}^2\big) \rVert^2 &= 0, \qquad \forall n \in \bbN,\label{eq:ApproxVariance}
\end{equs}
then the conclusions of Theorem \ref{thm:1} hold. 
\end{theorem}

\begin{remark}
We note the choice of $f_1$ in~\eqref{eq:ApproxH1} is just a convention used to minimise the number of different notations, 
and could have been replaced by $f_2$ or any other linear combination of the two. The reason why in this context 
the first coordinate suffices, is that the model is isotropic so that 
the joint law of the two coordinates can be deduced from either of the two. 
\end{remark}

We will shortly give the proof of the previous statement, but first let us briefly comment on 
the meaning of the quantities in~\eqref{e:gsigma} and the conditions~\eqref{eq:L2}-\eqref{eq:ApproxVariance}. 
As can be immediately deduced from~\eqref{e:Mart}, $q^{\lambda,n}$ is the integrand of the martingale 
in~\eqref{e:Dynkin} associated to $v^{\lambda,n}$ while $\sigma^2_{\lambda,n}$ satisfies
\begin{equ}
\bfE[M_t(v^{\lambda,n})^2] = t \sigma^2_{\lambda,n} = 2t \lVert \calS^\half v^{\lambda,n} \rVert^2
\end{equ}
where the second equality is a consequence of the integration by parts formula in~\eqref{eq:IBP}. In light of this, 
it is clear that~\eqref{eq:ApproxVariance} ensures that the variance of the quadratic variation of the martingales $M_t(v^{\lambda,n})$ 
is going to zero, so that Martingale CLT is indeed applicable, while~\eqref{eq:ApproxMean} identifies the limiting 
diffusivity. As for~\cref{eq:L2,eq:ApproxH1}, they respectively ensure that 
the boundary terms in~\eqref{e:Dynkin} vanish in the diffusive scaling and that the martingales $M_t(v^{\lambda,n})$ 
represent a good approximation for~\eqref{eq:AFM2}. 
\medskip

\begin{remark}\label{rem:DiffKLO}
There are relevant differences between ours and the setting of the SRBP in $d\geq 3$, or other more standard examples in which the Kipnis-Varadhan program has been applied. 
In these contexts, one usually considers a {\it fixed} generator $\calL$ and a {\it fixed} functional, 
while we must handle a family of generators and functionals which depend on the scaling parameter $\eps \in (0, 1)$. 
This prevents us from using certain functional analytical arguments, 
such as Mazur's theorem \cite[Lemma 2.16]{KomoLandOlla12_FluctuationsMarkov}
which ensures the existence of  
a strong limit point $u \in L^2(\pi)$ satisfying $\lim_{\lambda\rightarrow0} \lVert \calS^\half (u^\lambda - u) \rVert = 0$. 
The problem here is not merely technical, but substantial and it reflects the different nature of the problem at hand. 
Indeed, as we will detail in Section \ref{sec:no-weak-convergent}, a strong limit point simply {\it does not exist} 
but this is not needed for the invariance principle. 
\end{remark}

We will now give the proof Theorem~\ref{thm:3}. After that, in Section~\ref{sec:analysis-generator} 
we will construct the family $\{v^{\lambda,n}\}_{\lambda,n}$ and then verify that it satisfies 
conditions \cref{eq:L2,eq:ApproxH1,eq:ApproxMean,eq:ApproxVariance} in Sections~\ref{sec:appr-solv-gener} 
and~\ref{sec:invariance-principle}. 

\begin{proof}[of Theorem \ref{thm:3}]
For brevity, throughout the proof we write $\varsigma^2 = 1+\sigma^2(\alpha)$. 
Let $(\bm{\Omega}, \bm{\calF}, (\bm{\calF}_t)_{t \ge 0}, \bfP)$ be the annealed filtered space defined above~\eqref{eq:AFM}.  

We split the proof in four steps. In the first, we determine the one-time semi-quenched CLT 
for the first coordinate of the SRBP. This is the step in which all the conditions~\eqref{eq:L2}-\eqref{eq:ApproxVariance} 
are exploited. The second extends the CLT to the second coordinate while the third to multiple times. At last, 
as a consequence of the above and tightness, we obtain the annealed invariance principle. 
\medskip

\noindent\textit{Step 1.} Let $X$ be the SRBP in~\eqref{eq:SRBP2} and fix $t\geq 0$. 
By definition, the semi-quenched CLT for the first coordinate follows once we show that for all $\theta \in \bbR$ 
\begin{equation}
\label{eq:108}
\lim_{\varepsilon \rightarrow 0} \bfE\Big|\bfE[e^{\iota \theta e_1X^\varepsilon_t}|\bm{\calF}_0]-e^{-\half\varsigma^2\theta^2t}\Big| = 0\,.
\end{equation}
where $X^\eps$ is the diffusively rescaled SRBP, i.e. $X^\eps_t\eqdef \eps X_{t/\eps^2}$, and $e_1 X_t^\eps$ denotes 
its first coordinate. 

Fix $n \in \bbN$ and consider the Dynkin martingale $M(v^{\varepsilon^2,n})$ in~\eqref{e:Dynkin} 
associated to the observable $v^{\lambda,n} \in \text{dom}(\calL)$ in the statement, 
with $\lambda=\varepsilon^2$. 
Let $\varsigma^2_\varepsilon \eqdef \bfE[(M_1(v^{\varepsilon^2,n}) + e_1B_1)^2]$ and denote the 
diffusively rescaled martingale and Brownian motion as 
$M^\varepsilon_t \eqdef \varepsilon M_{t/\varepsilon^2}(v^{\varepsilon^2,n})$, $B^\eps\eqdef \eps B_{t/\eps^2}$.  
Then, we can bound
\begin{equ}
\bfE\Big|\bfE[e^{\iota \theta e_1X^\varepsilon_t}|\bm{\calF}_0]-e^{-\half\varsigma^2\theta^2t}\Big| \le \one + \two + \three
\end{equ}
where $\one, \two, \three$ are defined according to 
\begin{align*}
\one &\eqdef \bfE\Big|e^{\iota \theta e_1X^\varepsilon_t} - e^{\iota \theta (M^\varepsilon_t + e_1B^\varepsilon_t)}\Big|\,, \\
\two &\eqdef \bfE\Big|\bfE[e^{\iota \theta (M^\varepsilon_t + e_1B^\varepsilon_t)}|\bm{\calF}_0]-e^{-\half\varsigma_\varepsilon^2\theta^2t}\Big|\,, \\
\three &\eqdef \Big|e^{-\half\varsigma_\varepsilon^2\theta^2t} - e^{-\half\varsigma^2\theta^2t}\Big|\,,
\end{align*}
and we will treat each of them separately. 

First, we control $\one$. With $F_t^\eps = \eps\int_0^{t/\eps^2} \gamma f(\eta_s)\dd s$ as defined in \ref{eq:Fdef}, we have
$e_1X^\varepsilon_t = e_1F^\varepsilon_t+e_1B^\varepsilon_t$. Then, 
\begin{equ}
\one\leq \theta \big( \bfE [(M^\varepsilon_t - e_1F^\varepsilon_t)^2] \big)^\half
\end{equ}
and it suffices to show that the right hand side converges to $0$ in the double limit for $\eps\to 0$ first and 
$n\to\infty$ after. Applying Dynkin's formula to the observable $v^{\eps^2,n}$, we deduce that 
\begin{equation}
\label{eq:146}
\begin{split}
M^\varepsilon_t - e_1F^\varepsilon_t = \varepsilon v^{\varepsilon^2,n}(\eta_{t/\varepsilon^2}) - \varepsilon v^{\varepsilon^2,n}(\eta_0) - \varepsilon^3\int_0^{t/\varepsilon^2} v^{\varepsilon^2,n}(\eta_s) \dd s \\
+ \varepsilon \int_0^{t/\varepsilon^2} \big( (\varepsilon^2-\calL) v^{\varepsilon^2,n}(\eta_s) - \coup f_1\big) \dd s\,.
\end{split}
\end{equation}
For the first three terms at the right hand side, we use stationarity,~\eqref{e:ScalarProd} and Jensen's inequality, so that 
\begin{equs}
\bfE\Big|\varepsilon v^{\varepsilon^2,n}(\eta_{t/\varepsilon^2}) - \varepsilon v^{\varepsilon^2,n}(\eta_0) - \varepsilon^3\int_0^{t/\varepsilon^2} v^{\varepsilon^2,n}(\eta_s) \dd s\Big|^2\lesssim (1+t^2)\,\eps^2\|v^{\varepsilon^2,n}\|^2
\end{equs}
which vanishes thanks to~\eqref{eq:L2}. For the other term, we apply the It\^o trick~\eqref{eq:itotrick} with $p=2$ and 
$\lambda=\eps^2$, which gives
\begin{equation*}
\begin{split}
\bfE\Big| \varepsilon \int_0^{t/\varepsilon^2}  \big[ (\varepsilon^2-\calL) &v^{\varepsilon^2,n}(\eta_s) - \coup f_1\big] \dd s \Big|^2  \\
&\lesssim(t+t^2)\lVert (\varepsilon^2+\calS)^{-\half} \big[ (\varepsilon^2-\calL) v^{\varepsilon^2,n} - \coup f_1\big] \rVert^2
\end{split}
\end{equation*}
whose convergence to $0$ is guaranteed by~\eqref{eq:ApproxH1}. 

For $\two$, we apply the triangular version of the 
Martingale CLT given in Theorem~\ref{thm:clt} to the martingale $\calM^{(\varepsilon)} = M(v^{\eps^2,n}) + e_1B$, 
whose scaled version is $\calM^\eps\eqdef\eps\calM^{(\eps)}_{t/\eps^2}=M^\eps+e_1 B^\eps$, 
for which we need to verify conditions~\eqref{eq:97}-\eqref{eq:78}.  
Notice first that, by~\eqref{e:Mart} and It\^o's isometry, the quadratic variation of $\calM^{(\varepsilon)}$ is given by 
\begin{equs}[e:QV]
\langle \calM^{(\varepsilon)} \rangle_t &= \int_0^t q^{\varepsilon^2,n}(\eta_s) \dd s + 2 \int_0^t \nabla_1v^{\varepsilon^2,n}(\eta_s) \dd s + t\,,\\
 \bfE\big[ \langle \calM^{(\varepsilon)} \rangle_t \big] &= t \sigma_{\varepsilon^2,n}^2 + t
\end{equs}
where we used that $\bbE_\pi[\nabla_1 v^{\varepsilon^2,n}]=0$. 
Since $\sigma_{\varepsilon^2,n}^2$ is bounded uniformly in $\eps,n$ in view of~\eqref{eq:ApproxMean},
~\eqref{eq:97} holds. To check~\eqref{eq:82}, we brutally estimate 
\begin{equs}
\bfE[\langle \calM^{(\varepsilon)}\rangle_1^2]&\lesssim \bfE\big[\langle M(v^{\eps^2,n})\rangle_1^2\big] +1\\
&\lesssim \bfE\Big[ \int_0^1 \big(q^{\varepsilon^2,n}(\eta_s)\big)^2 ds\Big]+1 \lesssim_n \bbE_\pi[q^{\varepsilon^2,n}] +1= \sigma_{\varepsilon^2,n}^2+1
\end{equs}
where in the second step we used Jensen's inequality, in the third stationarity and in the last 
Gaussian hypercontractivity \cite[Theorem 5.10]{Janson97_GaussianHilbert}. Hence,~\eqref{eq:82} follows once again by~\eqref{eq:ApproxMean}. 
For~\eqref{eq:78}, by~\eqref{e:QV} and~\eqref{eq:itotrick} we have  
\begin{equs}
\sup_{s\leq t/\eps^2}\eps^4\Var \langle\calM^{(\varepsilon)}\rangle_s&=\sup_{s\leq t/\eps^2}\eps^4\bfE\Big|\int_0^s [q^{\varepsilon^2,n}(\eta_r)-\sigma^2_{\eps^2,n}] \dd r + 2 \int_0^s \nabla_1v^{\varepsilon^2,n}(\eta_r) \dd r\Big|^2\\
&\lesssim t\eps^2 \Big( \lVert (\varepsilon^2+\calS)^{-\half} (q^{\varepsilon^2,n} - \sigma_{\varepsilon^2,n}^2) \rVert^2 +  \lVert (\varepsilon^2+\calS)^{-\half} \nabla_1v^{\varepsilon^2,n} \rVert^2\Big)\\
&\leq  t\eps^2 \lVert (\varepsilon^2+\calS)^{-\half} (q^{\varepsilon^2,n} - \sigma_{\varepsilon^2,n}^2) \rVert^2+t\eps^2\lVert v^{\varepsilon^2,n} \rVert^2
\end{equs} 
and the right hand side converges to $0$ thanks to~\eqref{eq:ApproxVariance} and~\eqref{eq:L2}. 

At last, we are left with $\three$ which in turn is a consequence of~\eqref{e:QV} and~\eqref{eq:ApproxMean}. 
\medskip

\noindent\textit{Step 2.} Here, we use isotropy to translate the result of step $1$ to 
the joint coordinate process in $\bbR^2$. More precisely,  
for any rotation matrix $U\in SO(2)$, if $X$ solves~\eqref{eq:SRBP2} 
in the environment $\omega$ and with driver $B$, then 
$\tilde X\eqdef U X$ solves again~\eqref{eq:SRBP2} but with 
$(\tilde\omega,\tilde B)\eqdef (U\omega \circ U^{-1}, UB)$ in place of $(\omega,B)$. 
Hence, for $\theta=(\theta_1,\theta_2)\in\R^2$, letting $U \in SO(2)$ be a matrix sending 
the canonical basis element $e_1 \in \bbR^2$ to $(\theta_1/\theta,\theta_2/\theta)$, 
we obtain
\begin{equ} 
\bfE[e^{\iota \theta \cdot X^\eps_{t}}|\bm{\calF}_0]=\bfE[e^{\iota |\theta| e_1U^{-1}\tilde X^\eps_{t}}|\bm{\calF}_0]\eqlaw \bfE[e^{\iota |\theta| e_1X^\eps_{t}}|\bm{\calF}_0]
\end{equ}
where in the last step, we used rotational 
invariance of both $\omega$ and $B$. 
We can therefore use the previous step to conclude
\begin{equation}
\label{eq:3}
\lim_{\varepsilon \rightarrow 0} \bfE|\bfE[e^{\iota \theta\cdot X^\varepsilon_t}|\bm{\calF}_0]-e^{-\half\varsigma^2|\theta|^2t}| = 0\,.
\end{equation}
\medskip

\noindent\textit{Step 3.} We show semi-quenched convergence for the finite dimensional distributions, 
which is an easy consequence of the following claim. 
Let $\theta\in \bbR^2, 0 \le s \le t$ and $\{Y^\varepsilon\}_{\varepsilon > 0}$ a collection of random variables 
such that for every $\eps>0$, $Y^\eps$ is $\bm{\calF}_{s/\varepsilon^2}$-measurable, then
\begin{equation}
\label{eq:28}
\lim_{\varepsilon \rightarrow 0} \bfE|\bfE[e^{i Y^\varepsilon} \big( e^{\iota \theta\cdot (X^\varepsilon_t - X^\varepsilon_s)} - e^{-\half\varsigma^2|\theta|^2(t-s)} \big) | \bm{\calF}_0] | = 0\,.
\end{equation}
Indeed, we have 
\begin{align*}
\bfE|\bfE[e^{i Y^\varepsilon} \big( &e^{\iota \theta\cdot (X^\varepsilon_t - X^\varepsilon_s)} - e^{-\half\varsigma^2|\theta|^2(t-s)} \big) | \bm{\calF}_0] | \\
&= \bfE|\bfE[\bfE[e^{i Y^\varepsilon} \big( e^{\iota \theta\cdot (X^\varepsilon_t - X^\varepsilon_s)} - e^{-\half\varsigma^2|\theta|^2(t-s)} \big) | \bm{\calF}_{s/\varepsilon^2}] | \bm{\calF}_0] |\\
&\le \bfE|\bfE[ \big( e^{\iota \theta\cdot (X^\varepsilon_t - X^\varepsilon_s)} - e^{-\half\varsigma^2|\theta|^2(t-s)} \big) | \bm{\calF}_{s/\varepsilon^2}] |\,.
\end{align*}
By stationarity of $(\eta_t)_{t \ge 0}$ and independence of $(B^\varepsilon_t - B^\varepsilon_s)_{t \ge s}$ and $\bm{\calF}_{s/\varepsilon^2}$, under the annealed law $\bfP$ we have
\begin{equ}
\bfE[ \big( e^{\iota \theta\cdot (X^\varepsilon_t - X^\varepsilon_s)} - e^{-\half\varsigma^2|\theta|^2(t-s)} \big) | \bm{\calF}_{s/\varepsilon^2}] \eqlaw \bfE[ \big( e^{\iota \theta\cdot X^\varepsilon_{t-s}} - e^{-\half\varsigma^2|\theta|^2(t-s)} \big) | \bm{\calF}_0] 
\end{equ}
so that~\eqref{eq:28} follows from~\eqref{eq:3}. 

Now, to prove~\eqref{eq:AR5} we proceed by induction on the number of increments $n$. 
For $n=1$, the result is a consequence of~\eqref{eq:28} under the choice $Y^\varepsilon = 0$. 
For the general case, let $\theta_1,\dots,\theta_n\in\R^2$ and $0\leq t_0\leq\dots\leq t_n$ and 
note that the left hand side of~\eqref{eq:AR5} equals 
\begin{equs}
&e^{-\half \varsigma^2|\theta_n|^2(t_n-t_{n-1})} \int \Big| \bbE\Big[e^{\iota\sum_{k=1}^{n-1} \theta_k \cdot (X^\varepsilon_{t_k}-X^\varepsilon_{t_{k-1}})} \Big] - e^{-\half \varsigma^2\sum_{k=1}^{n-1} |\theta_k|^2 (t_k - t_{k-1})} \Big| \pi(\dd \omega) \\
&\quad +\int \Big| \bbE\Big[e^{\iota\sum_{k=1}^{n-1} \theta_k \cdot (X^\varepsilon_{t_k}-X^\varepsilon_{t_{k-1}})} \Big( e^{\iota \theta_n \cdot (X^\varepsilon_{t_n}-X^\varepsilon_{t_{n-1}})}- e^{\half \varsigma^2|\theta_n|^2 (t_n - t_{n-1})} \Big)\Big]\Big| \pi(\dd \omega)\,.
\end{equs}
Therefore, in the limit $\eps\to0$, the first summand vanishes by the induction hypothesis, 
while the second by~\eqref{eq:28}. 
\medskip

\noindent\textit{Step 4.} At last, the semi-quenched convergence of the finite dimensional distributions proved in 
the previous step implies convergence of the finite dimensional distributions with respect to the annealed measure $\bfP$. Combining this with the tightness proved in Proposition \ref{p:Tightness}, the invariance principle follows. 
\end{proof}

\section{A good family of observables}
\label{sec:analysis-generator}

The goal of this section is to introduce the family of observables $\{v^{\lambda,n}\}_{\lambda,n}$ 
in Theorem~\ref{thm:3}. To motivate it, let $\psi\in\fock_1$ be given, and recall that, 
for the Kipnis-Varadhan program,
one would like to consider the solution $u^\lambda$ of the resolvent equation, $(\lambda-\calL) u^\lambda = \psi$. 
In order to derive its properties, the approach followed in~\cite{KomoLandOlla12_FluctuationsMarkov} consists 
of truncating the generator at a given chaos $n\in\bbN$ sufficiently high, obtaining estimates which are uniform 
in $n$, and then passing to the limit. To be precise, for fixed $n$, one looks at 
the solution $u^{\lambda,n} \in \fock_{\le n}$ of 
\begin{equation}
\label{eq:5}
(\lambda - \Pi_n \calL \Pi_n) u^{\lambda,n} = \psi
\end{equation}
where $\Pi_n \colon \fock \rightarrow \fock_{\le n}$ is the canonical projection operator. 
Upon writing it in its chaos components,~\eqref{eq:5} reduces to a finite system of $n$ linear equations 
and its solution is given by $u^{\lambda,n}=\sum_{j=0}^n u^{\lambda,n}_j\in\fock_{\le n}$, 
where we set $u^{\lambda,n}_0=0$. 
The $j$-th component $u^{\lambda,n}_j$ can be computed recursively according to 
\begin{equation}
\label{eq:111}
u^{\lambda,n}_j = (\lambda+\calS+\calT^\lambda_{n-j})^{-1}\big( \calA_+ u^{\lambda,n}_{j-1} +\Pi_j \psi \big)
\end{equation}
where $\Pi_j\psi=0$ for all $j\neq 1$, and the operators $\calT^\lambda_j$ are defined according to 
\begin{equation}
\label{eq:109}
\calT_0\equiv 0\,,\qquad \text{and}\qquad \calT^\lambda_{j+1} = -\calA_-(\lambda+\calS+\calT^\lambda_{j})^{-1}\calA_+\,.
\end{equation}
Formally, one expects that, by taking $n\to\infty$, for every $j\in\N$, $u_j^{\lambda,n}$ 
converges to the $j$-th component  
of the solution $u^\lambda$ to the resolvent equation, and the latter should, heuristically, have the same structure as the right 
hand side of~\eqref{eq:111}
but with the operator $\calT^\lambda_{n-j}$ replaced by the operator $\calT^\lambda$ given by the fixed point of the 
relation in~\eqref{eq:109}, i.e. $\calT^\lambda$ should satisfy
\begin{equ}[e:FPO]
\calT^\lambda = -\calA_-(\lambda+\calS+\calT^\lambda)^{-1}\calA_+\,.
\end{equ} 
The idea, first explored in~\cite{CannGubiToni23_GaussianFluctuations} in the context of the Stochastic Burgers Equation 
in the critical dimension $d=2$, is that, even though it might be difficult to directly study $\calT^\lambda$ and, in particular, 
its behaviour for $\lambda$ small, it is possible to derive an approximate fixed point for~\eqref{e:FPO} 
which can then be used to define the family $\{v^{\lambda,n}\}_{\lambda,n}$ (see Definition~\ref{def:3}). 
The existence and properties of such an approximate fixed point are the content of the so-called {\it replacement lemma}, 
that we will shortly state and prove. That said, compared to~\cite{CannGubiToni23_GaussianFluctuations}, 
in the present context we face an additional difficulty. Indeed, the replacement only holds for the 
operators $\calA_\pm$ restricted to a (sufficiently large) region, which we will call the 
{\it bulk region}. The residual region, which will be referred to as the {\it nuisance region}, 
needs to be handled in an entirely different way and turns out to be small only because of 
a non-trivial cancellation of terms. 

While the analysis of the nuisance region is deferred to the next section, 
the rest of this section is organised as follows. 
At first, we derive some preliminary estimates on the operators $\calA_\pm$ (a generalised version of 
the so-called graded sector condition), then rigorously introduce bulk and nuisance region, 
state and prove the replacement lemma and conclude with the definition of the family $\{v^{\lambda,n}\}_{\lambda,n}$. 

\subsection{An (alternative) graded sector condition}
\label{sec:grad-sect-cond}

The bulk of the work in~\cite{HorvTothVeto12_DiffusiveLimits} for $d\geq 3$, 
and the reason why the Kipnis-Varadhan program can be directly applied, 
is the verification of the so-called graded sector condition \cite{SethVaraYau00_DiffusiveLimit} which requires that, when restricted to $\fock_n$, 
the operator $\cA_+$ satisfies 
$ \lVert \calS^{-\half} \calA_+ \calS^{-\half} \rVert_{\fock_n \rightarrow \fock_{n+1}} \le c(n+1)^\beta $, 
for some $\beta\in(0,1)$ (in~\cite{HorvTothVeto12_DiffusiveLimits}, $\beta=1/2$). 
As can be directly checked, in dimension $d=2$, the graded sector condition fails in that   
the operator norm is simply unbounded. To bypass the problem, we introduce a mass $\lambda>0$ 
and look at $ \lVert (\lambda+\calS)^{-\half} \calA_+ \calS^{-\half} \rVert_{\fock_n \rightarrow \fock_{n+1}}$. 
Our goal is to prove that, upon choosing $\gamma$ as in~\eqref{eq:Weak}, the previous is indeed bounded 
and satisfies the same bound as in $d\geq 3$. 
We will formulate a more general version of this result, but to do so, we need to introduce a ``local'' version  
$\calA^R$ of $\calA$. 

Let $R$ be a measurable subset of $(\R^2)^2$ and, throughout the paper, we will denote by $R(p,q)$, $p,q\in\R^2$, 
the characteristic function associated to $R$, i.e. 
\begin{equ}[e:convReg]
R(p,q)\eqdef \1_R(p,q)\,,\qquad p,q\in\R^2\,.
\end{equ}
We define the operator $\calA^R \eqdef \calA^R_+ + \calA^R_-$, where 
$\calA^R_+, \calA^R_-$ are such that $\calA^R_+|_{\calH_0} = \calA^R_-|_{\calH_0\oplus\calH_1} = 0$ and 
act on $\psi \in \fock_n$ according to 
\begin{equs}[eq:AR] 
\calA^R_+ \psi (p_{1:n+1})&\eqdef \frac{\coup}{n+1} \sum_{i=1}^{n+1} R(p_i, p_{[1:(n+1)\setminus i]})\, p_i \cdot  p_{[1:(n+1)\setminus i]} \, \psi(p_{1:(n+1)\setminus i} )\,, \\
\calA^R_- \psi (p_{1:n-1}) &\eqdef \coup n\int\frac{\VHat(q)}{|q|^2} R(q,  p_{[1:n-1]}) \,q \cdot p_{[1:n-1]}\,\psi(q,p_{1:n-1}) \dd q\,.
\end{equs}
These operators satisfy the same adjoint relationship as before, i.e. $(\calA^R_\pm)^* = -\calA^R_\mp$, 
and coincide with $\calA_\pm$ in~\eqref{eq:ops} provided $R=(\R^2)^2$. 
One should think of the regional restriction as a mean for redefining the scalar product $p \cdot q$ as $(p \cdot q) R(p, q)$ 
and we shall later choose the region $R$ in such a way that the adjusted scalar product has an improved behaviour. 

We further split the operator $\calA_+$ into its constituent parts, a procedure whose crucial role in our analysis 
will become clearer in the next section. Notice first that the definition of the operators in~\eqref{eq:AR} naturally 
extend to elements in $L^2_\bbC(\ffock{n})$ 
that are not necessarily symmetric with respect to permutation of their variables. 
Slightly abusing notations, we will still write $\lVert\cdot\rVert, \langle \cdot, \cdot \rangle$ for the norm 
and inner product on the larger space $L^2_\bbC(\mu)\eqdef\oplus_{n=0}^\infty L^2_\bbC(\ffock{n})$. 
For $\psi \in L^2_\bbC(\mu_n)$, the action of $\calA^R_+$ on $\psi$ can be decomposed as 
$\calA^R_+\psi = \sum_{i=1}^{n+1} \genAR{i}\psi$ 
where $\genAR{i}$ corresponds to the $i$'th constituent in the sum for $\calA^R_+$, i.e. 
\begin{equ}[eq:ARi]
\genAR{i} \psi \eqdef \frac{\coup}{n+1} p_i \cdot p_{[1:(n+1)\setminus i]} R(p_i, p_{[1:(n+1)\setminus i]}) \psi(p_{1:(n+1)\setminus i} )\,.
\end{equ}
The function $\genAR{i} \psi$ is not in general symmetric, even in the case where $\psi$ is.

\begin{lemma}\label{lem:DiOff}
Let $\calT$ be the extension to $L^2_\C(\mu)$ of a non-negative diagonal operator with multiplier $\tau$ 
(see Definition~\ref{def:DO}). 
Let $R$ be a measurable subset of $(\R^2)^2$ and, 
for $\lambda>0$, $n\in\N$ and $\psi\in \fock_n$, set 
$\phi^R\eqdef (\lambda+\calS+\calT)^{-\half} \calA^R_+ \calS^{-\half} \psi$ and 
\begin{equ}[e:phii]
\phi^R[i]\eqdef (\lambda+\calS+\calT)^{-\half} \genAR{i} \calS^{-\half} \psi\,,\qquad i=1,\dots,n+1\,.
\end{equ}
Then, the norm of $\phi^R$ can be decomposed as 
\begin{equ}[eq:DiOff]
\lVert\phi^R\rVert^2 = \langle\psi, \cM^\lambda \psi\rangle + \sum_{ i\ne i^\prime=1}^{n+1} \langle\phi^R[i], \phi^R[i']\rangle\,,
\end{equ}
where the second sum will be referred to as the off-diagonal term, while the first, as the diagonal. 
In the above, $\cM^\lambda$ is the diagonal non-negative operator on $L^2_\C(\mu)$ with multiplier $\cm^\lambda$ 
given by 
\begin{equ}[eq:DiKer]
\cm^\lambda(q)= 2\coup^2 \int_{\bbR^2} \frac{\VHat(p) \cos^2(\theta) R(p,q)}{\lambda+\half|p+q|^2 + \tau(p+q)} \dd p
\end{equ}
where for $p \in \bbR^d$, we write $\theta = \theta(p,q)$ for the angle between $p$ and $q$.
\end{lemma}
\begin{proof}
Thanks to the notations introduced above, $\phi^R$ can be written as 
$\phi^R = \sum_{i=1}^{n+1} \phi^R[i]$, which implies that 
\begin{equ}
\lVert\phi^R\rVert^2 = \sum_{i=1}^{n+1}\lVert\phi^R[i]\rVert^2 + \sum_{ i\ne i^\prime=1}^{n+1} \langle\phi^R[i], \phi^R[i']\rangle\,.
\end{equ}
While the second summand corresponds to the off-diagonal term, for the first note that 
$\phi^R[i]$ can be explicitly written as 
\begin{equation}\label{eq:AR8}
\phi^R[i](p_{1:n+1})=  \frac{\sqrt{2}\coup}{n+1} \frac{p_i \cdot p_{[1:n+1\setminus i]} R(p_i, p_{[1:n+1\setminus i]}) \psi(p_{1:n+1\setminus i}) }{(\lambda+\half|p_{[1:n+1]}|^2+\tau(p_{[1:n+1]}))^{\half} |p_{[1:n+1\setminus i]}|}\,.
\end{equation}
A simple change of variables shows that $\lVert\phi^R[i]\rVert^2$ is independent of $i$, so that 
the diagonal term equals
\begin{equs}
&\sum_{i=1}^{n+1}\lVert\phi^R[i]\rVert^2=(n+1)\lVert\phi^R[n+1]\rVert^2\\
&=\frac{2\gamma^2}{n+1}\int \frac{ |p_{n+1} \cdot p_{[1:n]}|^2 R(p_{n+1}, p_{[1:n]}) |\psi(p_{1:n})|^2 }{(\lambda+\half|p_{[1:(n+1)]}|^2+\tau(p_{[1:(n+1)]})) |p_{[1:n]}|^2} \mu_{n+1}(\dd p_{1:n+1})\\
&=\int \mu_n(\dd p_{1:n}) |\psi(p_{1:n})|^2 \Big(2\gamma^2\int\dd p\frac{\VHat(p) \cos^2(\theta) R(p,p_{[1:n]})}{\lambda+\half|p+p_{[1:n]}|^2 + \tau(p+p_{[1:n]})}\Big)
\end{equs}
where in the last step we expanded the scalar product, set $\theta=\theta(p,p_{[1:n]})$, and 
replaced $\mu_{n+1}$ with $\mu_n$ (see~\eqref{e:mun}). Since the quantity in parenthesis 
equals $\cm^\lambda(p_{[1:n]})$, the proof is concluded. 
\end{proof}

We are now ready to give our generalised version of the graded sector condition. 

\begin{lemma}[The Graded Sector Condition]\label{lem:32}
In the same setting of Lemma~\ref{lem:DiOff}, upon choosing $\gamma$ according to~\eqref{eq:Weak}, 
the following bounds hold uniformly over $\lambda\in(0,1)$,
\begin{align}
\lVert(\lambda+\calS+\calT)^{-\half} \calA^R_+ \calS^{-\half}\rVert_{\fock_n \rightarrow \fock_{n+1}}^2 &\lesssim n+1\,, \label{eq:AR7}\\
\lVert \calS^{-\half} \calA^R_- (\lambda+\calS+\calT)^{-\half} \rVert_{\fock_{n+1} \rightarrow \fock_n}^2 &\lesssim n+1\,. \label{eq:AR9}
\end{align}
\end{lemma}
\begin{proof}
In order to establish~\eqref{eq:AR7}, we look at~\eqref{eq:DiOff} and apply Cauchy-Schwarz to the off-diagonal 
term, from which we deduce
\begin{equs}
\lVert(\lambda+\calS+\calT)^{-\half} \calA^R_+ \calS^{-\half}\psi\rVert^2\leq (n+1) \langle\psi, \cM^\lambda \psi\rangle
\end{equs}
where $\cM^\lambda$ is the non-negative diagonal operator with multiplier $\cm^\lambda$ given in~\eqref{eq:DiKer}. 
Bounding the indicator function $R(\cdot,\cdot)$ and the cosine by $1$ and using that $\tau$ 
is non-negative, we obtain
\begin{equ}
\langle\psi, \cM^\lambda \psi\rangle\lesssim \int \mu_n(\dd p_{1:n}) |\psi(p_{1:n})|^2 \Big(\gamma^2\int\dd p\frac{\hat V(p) }{\lambda+\half|p_{[1:n]}+p|^2} \Big)\,.
\end{equ}
An argument similar to that in~\eqref{e:UBWeak} shows that uniformly in $q \in \bbR^2$ and $\lambda\in(0,1)$, we have
\begin{equation}
\label{eq:16}
\gamma^2\int\dd p\frac{\hat V(p) }{\lambda+\half|q+p|^2} \lesssim 1
\end{equation}
so that~\eqref{eq:AR7} follows at once. 
To conclude the proof it remains to show~\eqref{eq:AR9}, which in turn can be deduced from~\eqref{eq:AR7} by duality. 
\end{proof}

\subsection{The replacement lemma and the family $\{v^{n,\lambda}\}_{n,\lambda}$}
\label{sec:bulk-region-b}

As we are interested in the large scales of the environment (and of the SRBP), 
it is no surprise that the relevant behaviour comes from small Fourier modes and, more 
specifically, from those modes at which $\cS$ is small, i.e. $|p_{[1:n]}|=|\sum p_i|\approx 0$. 
There are two ways in which this quantity can be small, either all the modes are small (bulk region)
or they are order one but cancel each others out (nuisance region). 

More precisely, for $\kappa\in(0,1)$, we define the {\it bulk region} $B^\kappa$ 
and its complement, the {\it nuisance region}, $N^\kappa$ as
\begin{equ}[e:BNregions]
B^\kappa \eqdef \{(p,q)\in(\R^2)^2\colon |p+q| \ge \kappa|q| \} \quad\text{and}\quad N^\kappa \eqdef (B^\kappa)^c\,.
\end{equ}
Note that these regions are not symmetric in $p$ and $q$, and that, by triangle inequality, we have 
\begin{equs}[e:BNregionsBound]
B^\kappa &\subseteq \{(p,q)\colon|p+q| \ge \max\{ \kappa|q| , \tfrac{\kappa}{1+\kappa}|p|\} \}\\
N^\kappa &\subseteq \{(p,q)\colon|p+q| < \min\{\kappa|q|, \tfrac{\kappa}{1-\kappa}|p|\}, \, \tfrac{1}{1+\kappa} |p| < |q| < \tfrac{1}{1-\kappa}|p| \} 
\end{equs}
In particular, in the nuisance region, $|p+q| \lesssim_\kappa \min \{|p|, |q|\}$ and in the bulk region, 
$|p+q| \gtrsim_\kappa \max\{|p|, |q|\}$. 
As a convention, we mostly consider $\kappa = \tfrac{1}{3}$, 
and we write $B, N$ for $B^{\frac{1}{3}},N^{\frac{1}{3}}$ respectively. 
\medskip

Before stating the replacement lemma, we need to introduce the approximate fixed point operator of~\eqref{e:FPO} which 
is a diagonal operator (see Definition~\ref{def:DO}) given by an order one perturbation of $\calS$. 
To do so, let $g$ and $\ell^\lambda$ be the functions on $[0,\infty)$ defined according to 
\begin{equ}[e:g]
g(y)\eqdef \sqrt{4\pi y +1}-1\,,\qquad\text{and}\qquad\ell^\lambda(y)\eqdef \gamma^2 \log\Big(1+\frac{1}{\lambda+y}\Big)\,.
\end{equ}
Then, we set $\calG^\lambda$ to be operator on $\fock$ given by 
\begin{equ}[e:ApproxFPO]
\calG^\lambda\eqdef g(\ell^\lambda(\calS))
\end{equ}
so that in particular $\calG^\lambda$ is non-negative and diagonal with multiplier 
$\cg^\lambda(p)\eqdef g\circ\ell^\lambda(\tfrac12|p|^2)$, $p\in\R^2$. 

\begin{lemma}[Replacement Lemma]
\label{lem:3}
Let $\cG^\lambda$ be the non-negative diagonal operator on $\fock$ given by~\eqref{e:ApproxFPO} 
with multiplier $\cg^\lambda$. 
Then, there exists a constant $C>0$ such that for all $\lambda\in(0,1)$, $n\in\N$ 
and $\psi_1,\psi_2\in\fock_n$, we have 
\begin{equation}
\label{eq:7}
|\langle\psi_1, \calS^{-\frac12}\calR_\lambda\calS^{-\frac12}\psi_2\rangle| \leq C \coup^2 n \lVert  \psi_1 \rVert \lVert \psi_2 \rVert
\end{equation}
where $\calR_\lambda$ is the operator given by 
\begin{equ}\label{e:Rlambda}
 \calR_\lambda \eqdef  -\calA^B_-(\lambda+\calS+\calS\calG^\lambda)^{-1}\calA^B_+ - \calS\calG^\lambda 
\end{equ}
and $B$ is the bulk region defined in~\eqref{e:BNregions}. 
\end{lemma}

Before proving the statement, let us make a few remarks and see how to use it to define the family of observables 
$\{v^{\lambda,n}\}_{\lambda,n}$. 
An immediate corollary of~\eqref{eq:7} is that the operator norm of $\calR_\lambda$ in~\eqref{e:Rlambda} on $\fock_n$ 
satisfies 
\begin{equ}\label{e:eq7Op}
\lVert \calS^{-\half} \calR_\lambda \calS^{-\half} \rVert_{\fock_n \rightarrow \fock_n}^2 \lesssim \coup^2 n
\end{equ}
which, since we assume $\gamma$ to be chosen according to weak coupling, i.e. as in~\eqref{eq:Weak}, 
implies that for $n$ fixed, the left hand side vanishes as $\lambda\rightarrow0$. 
In other words, on any given chaos, $\calS\calG^\lambda$ is an {\it approximate fixed point} to~\eqref{e:FPO} 
uniformly over $\lambda$, at least when replacing $\calA_\pm$ with $\calA_\pm^B$. 
As we will see, this is enough for our purposes and, in line with the heuristic provided 
at the beginning of the section, we can give the definition of the family of observables we will be considering 
hereafter. 

\begin{definition}[Replacement Equation]
\label{def:3}
For $\lambda\in(0,1)$, $m\leq n\in\N$ and $\frakf\in\fock_{m}$, we define 
$v^{\lambda,n} = \sum_{j=m}^n v^{\lambda,n}_j\in \fock_{\leq n}$ to be the solution of the {\it replacement equation} 
with input $\frakf$, which is given by 
\begin{equation}
\label{eq:114}
v^{\lambda,n}_j = (\lambda+\calS+\calS\calG^\lambda)^{-1} \big( \calA^B_+ v^{\lambda,n}_{j-1} + \Pi_j \frakf \big)\,,\qquad j=m,\dots,n
\end{equation}
with the convention that $v^\lambda_{m-1} = 0$. 
\end{definition}

We now turn to the proof of Lemma~\ref{lem:3} which crucially relies on the fact that, in the bulk region, 
the off-diagonal terms in~\eqref{eq:DiOff} are small. This is the content of the next result. 

\begin{lemma}
\label{lem:6}
Let $\calT$ be a non-negative diagonal operator with multiplier $\tau$.  For $n \in \bbN$, $\psi_1, \psi_2 \in \fock_n$ and 
$i,j=1,\dots,n+1$, define $\phi_1^B[i],\phi_2^B[j]$ according to~\eqref{e:phii} with $\psi$ replaced by 
$\psi_1$ and $\psi_2$ respectively and with the choice $R = B$. Then, uniformly over $\lambda\in(0,1)$, we have 
\begin{equ}
\Big\lvert \sum_{i\neq j} \langle \phi_1^B[i],\,\phi_2^B[j]\rangle\Big\rvert \lesssim  \coup^2 n  \lVert\psi_1\rVert \lVert\psi_2\rVert \,.
 \end{equ}
\end{lemma}
\begin{proof}
A simple change of variables shows that 
the scalar product in the sum does not depend on the specific values of $i$ and $j$, as long as they are different. 
Hence, 
\begin{equ}
\sum_{i\neq j} \langle \phi_1^B[i],\,\phi_2^B[j]\rangle=n(n+1)\langle \phi_1^B[1],\,\phi_2^B[2]\rangle 
\end{equ}
and we are left to control the absolute value of the right hand side. 
Thanks to the explicit expression in~\eqref{eq:AR8}, we have 
\begin{equs}
n(n+1)\big\lvert\langle \phi_1^B[1],\,&\phi_2^B[2]\rangle\big\rvert \leq 2\gamma^2 \int \Phi_1[1](p_{1:n+1})\Phi_2[2](p_{1:n+1})\mu_{n+1}(\dd p_{1:n+1})\\
&\leq 2\gamma^2 \prod_{\ell=1,2}\Big( \int \frac{|p_{[1:n+1\setminus \ell]}|}{|p_{[1:n+1\setminus (\{1,2\}\setminus \ell)]}|} \Phi_\ell[\ell]^2 \mu_{n+1}(\dd p_{1:n+1}) \Big)^{1/2} 
\end{equs}
where, in the second step we applied Cauchy-Schwarz, and in the first the functions $\Phi_\ell[i]$ for $\ell=1,2$ and 
$i=1,\dots,n+1$ are given by
\begin{equ}
\Phi_\ell[i](p_{1:n+1})\eqdef \frac{|p_i| B(p_i, p_{[1:n+1\setminus i]}) }{(\lambda+|p_{1:n+1}|^2)^{\half}} |\psi_\ell(p_{1:n+1\setminus i})|\,.
\end{equ}
Since the two factors above can be treated in exactly the same way, we consider only the case $\ell=1$. 
Note that if $(p_1,p_{[2:n+1]})\in(\R^2)^2$ are such that $B(p_1, p_{[2:n+1]})=1$ then necessarily
\begin{equation*}
|p_{[1:n+1]}|^2 \gtrsim |p_1|^2+|p_{[2:(n+1)]}|^2. 
\end{equation*}
Therefore, we get 
\begin{equs}
&\int \frac{|p_{[2:n+1]}|}{|p_{[1:n+1\setminus 2]}|} \Phi_1[1](p_{1:n+1})^2 \mu_{n+1}(\dd p_{1:n+1}) \\
&\lesssim \int \frac{|p_{[2:n+1]}|}{|p_{[1:n+1\setminus 2]}|} \frac{|p_1|^2|\psi_1(p_{2:(n+1)})|^2}{\lambda+|p_1|^2+|p_{[2:n+1]}|^2}\mu_{n+1}(\dd p_{1:n+1}) \\
&= (n+1) \int\mu_{n}(\dd p_{2:n+1}) |\psi_1(p_{2:n+1})|^2  \int \frac{\VHat(p_1)|p_{[2:n+1]}| \dd p_1}{|p_{[1:n+1\setminus 2]}| (\lambda+|p_1|^2+|p_{[2:(n+1)]}|^2)} 
\end{equs}
and, by Lemma~\ref{lem:4}, it follows that the inner integral is uniformly bounded in $\lambda$ and $p_{[2:n+1]}$, 
from which the statement follows at once. 
\end{proof}

\begin{proof}[of Lemma~\ref{lem:3}]

For $\ell=1,2$ and $i=1,...,n+1$, let $\phi^B_\ell[i]\in\fock_{n+1}$ be as in the proof of Lemma \ref{lem:6} with $\calT = \calS\calG^\lambda$. 
Notice that, since $(\calA^B_+)^* = -\calA^B_-$, the left-hand side of~\eqref{eq:7} is given by 
\begin{equs}
|\langle\psi_1, \calS^{-\frac12}\calR_\lambda\calS^{-\frac12}\psi_2\rangle|&=|\langle\psi_1,\big[-\calS^{-\frac12}\calA^B_-(\lambda+\calS+\calS\calG^\lambda)^{-1}\calA^B_+\calS^{-\frac12} - \calG^\lambda\big]\psi_2\rangle|\\
&=|\langle\phi^B_1,\phi^B_2\rangle-\langle\psi_1, \calG^\lambda\psi_2\rangle|\\
&\leq|\langle \psi_1, \cM^\lambda\psi_2\rangle-\langle\psi_1, \calG^\lambda\psi_2\rangle|+\sum_{i\neq j} |\langle \phi_1^B[i],\,\phi_2^B[j]\rangle|
\end{equs}
where in the second step we used that $\calA_-=-\calA_+^*$ and the definition of $\phi^B$ in Lemma~\ref{lem:DiOff} 
with $\ST=\cG^\lambda$ and $R$ being the bulk region $B=B^{\frac13}$ in~\eqref{e:BNregions}, 
while in the last,~\eqref{eq:DiOff} for $\cM^\lambda$ the diagonal operator 
with multiplier $\cm^\lambda$ defined as in~\eqref{eq:DiKer} with $\tau$ 
given by $\tau(p)=\tfrac12|p|^2\cg^\lambda(p)$, $p\in\R^2$ (see~\eqref{e:ApproxFPO}). 
Now, Lemma~\ref{lem:6} implies that the off-diagonal term is bounded above by the right hand side of~\eqref{eq:7}. 
For the other, it suffices to show that 
\begin{equ}\label{eq:10}
\sup_{p}|\cm^\lambda(p)-\cg^\lambda(p)|\lesssim \gamma^2
\end{equ}
which in turn is proven in Lemma~\ref{lem:ReplEst}. 
Therefore, the proof of the replacement lemma is concluded. 
\end{proof}

\section{Estimates for the approximate solutions}
\label{sec:appr-solv-gener}

In this section, we derive crucial estimates on the solution to the replacement equation $v^{\lambda,n}$ 
given in Definition \ref{def:3} with {\it generic} input $\frakf\in\fock_{m}$, $m\in\N$. 
These estimates are essential not only to verify conditions~\eqref{eq:L2} and~\eqref{eq:ApproxH1} of Theorem \ref{thm:3}, 
where the family $(v^{\lambda,n} : \lambda \in (0,1), n \in \bbN)$ corresponds to the choice $\frakf = \coup f_1$, 
but also in the proof of Theorem \ref{thm:2}, for which a more general $\frakf$ will be needed. 

More precisely, our goal is to prove the following proposition. We say that a function $\frakf\in\fock_m$ satisfies the reverse triangle inequality with respect to $c\in(0,\half)$, when it holds that
\begin{equ}\label{a:RTI}
\text{supp}\{\frakf\} \subseteq \Big\{\Big|\sum_{j=1}^{m}p_j\Big| \ge c\sum_{j=1}^{m}|p_j|\Big\}\,.
\end{equ}

\begin{proposition}\label{p:MainEstimates}
For $m\in\N$, $c\in(0,\half)$ and $n\geq m$, there exists a constant $C=C(n)>0$ and $\bar\lambda=\bar\lambda(c,n)\in(0,1)$ such that for all $\lambda\in (0,\bar\lambda)$ and for all $\frakf\in\fock_m$ satisfying the reverse triangle inequality with respect to $c$, it holds that 
\begin{align}
\lambda\|v^{\lambda,n} - (\lambda + \calS + \calS\calG^\lambda)^{-1} \frakf \|^2&\lesssim \gamma^2 n \|(\lambda+\calS)^{-\half}\frakf\|^2\label{eq:L2Bound}\\
\lVert (\lambda+\calS)^{-\half} \big[ (\lambda-\calL)v^{\lambda, n} - \frakf \big] \rVert^2&\lesssim n\vertiii{ \frakf }^2 +\Big(\frac1{n}+\gamma^2 C\Big) \|(\lambda+\calS)^{-\half}\frakf\|^2\label{eq:ApproxH1Bound}
\end{align}
where $v^{\lambda,n}$ is the solution of the replacement equation in Definition~\ref{def:3} with input $\frakf$, and the seminorm on the right hand side is defined as 
\begin{equs}[e:VanNorm]
\vertiii{ \frakf }^2 \eqdef& \lVert (\lambda+\calS)^{-\half} \calA^N  (\lambda + \calS + \calS\calG^\lambda)^{-1} \frakf \rVert^2 \\
&+ \|(\lambda+\calS)^{-\half}\calA^B_-(\lambda+\calS+\calS\calG^\lambda)^{-1}\frakf\|^2.
\end{equs}
\end{proposition}

Before delving into the details, let us make a few comments. 
As we will see in Section~\ref{sec:remainder-terms}, the derivation of~\eqref{eq:L2Bound} 
is significantly simpler as it is a direct consequence of~\eqref{eq:114} and the explicit 
Fourier representation of the approximate fixed point operator $\calS\calG^\lambda$ 
(see Lemma~\ref{lem:L2control} below). 
On the other hand,~\eqref{eq:ApproxH1Bound} requires a much finer control. 
To see this, note that 
\begin{equation}
\label{eq:27}
\begin{split}
(\lambda-\calL)v^{\lambda,n}-\frakf= & - \calA^B_-(\lambda+\calS+\calS\calG^\lambda)^{-1}\frakf \\
&+ \calR_\lambda v^{\lambda,n-1}-(\calS\calG^\lambda + \calA^B_+)v^{\lambda,n}_n-\calA^Nv^{\lambda,n}
\end{split}
\end{equation}
where $\calR_\lambda$ is given by~\eqref{e:Rlambda}.  
While the first term at the right hand side only depends on $\frakf$ (and is thus inserted in the norm~\eqref{e:VanNorm}) 
the others require a special treatment. To control $\calR_\lambda v^{\lambda,n-1}$, we use the   
Replacement Lemma~\ref{lem:3}. In view of the graded sector condition~\eqref{eq:AR7}, 
$(\calS\calG^\lambda + \calA^B_+)v^{\lambda,n}_n$ is bounded but to see its smallness we need to leverage 
the fact that $v^{\lambda,n}_n\in\fock_n$, i.e. show that its norm 
decays (at least polynomially) with $n$. 

The term which creates the most difficulties though is $\calA^Nv^{\lambda,n}$, whose analysis requires a deeper  
understanding of how $\calA$ restricted to the nuisance region $N$ acts on $v^{\lambda,n}$. In the language of Lemma~\ref{lem:DiOff}, 
it is {\it not the case that the off-diagonal terms are lower order}. Contrary to what happens for $\calA^B$, 
diagonal and off-diagonals have the same order (see Appendix~\ref{a:nuisance}). That said, we will show that $\|(\lambda+\calS)^{-1/2}\calA^N\phi\|$ is indeed negligible due to cancellation of these terms, but only if $\phi$ has a suitable structure which ultimately lead us to impose condition~\eqref{a:RTI} on $\frakf$. 

\medskip

The rest of the section is organised as follows. In Section~\ref{sec:contr-nusi-regi}, we study the 
nuisance region and derive a suitable bound on $\calA^Nv^{\lambda,n}$. This is then 
used in Section~\ref{sec:apriori} to obtain weighted a priori estimates on the norms of $v^{\lambda,n}$. 
At last, in Section~\ref{sec:remainder-terms}, we put together the results in the previous sections and 
complete the proof of Proposition~\ref{p:MainEstimates}.

\subsection{Controlling the nuisance region $N$}
\label{sec:contr-nusi-regi}

The main result of this section is the following, which we shall later apply to the case $\psi = \calS^\half v^{\lambda,n}_j$.

\begin{lemma}
\label{lem:nuisance}
Let $\lambda \in (0,1)$ and $n \in \bbN, n \ge 2$. Assume $\psi \in \fock_{n-1}$ satisfies~\eqref{a:RTI} 
with respect to some $c\in(0,\tfrac12)$. 
Let $\phi^B = (\lambda+\calS+\calS\calG^\lambda)^{-\half}\calA^B_+ \calS^{-\half}\psi$ be as in 
Lemma \ref{lem:DiOff} under the choice $R = B$ and $\calT = \calS\calG^\lambda$. 
Then, for $\sigma\in\{+,-\}$ it holds that 
\begin{equation}
\label{eq:AR0}
\lVert (\lambda+\calS)^{-\half} \calA_\sigma^N (\lambda+\calS+\calS\calG^\lambda)^{-\half} \phi^B \rVert^2 \lesssim \coup^2 n^2   c^{-4} \lVert \psi \rVert^2\,.
\end{equation}
\end{lemma}

The bound in \eqref{eq:AR0} resembles a double application of the Graded Sector Condition Lemma \ref{lem:32} 
(hence the quadratic term $n^2$ at the right hand side), 
but with one crucial difference: we have gained a factor depending on the vanishing coupling constant $\coup$ 
which ensures that the contribution is indeed negligible. 

The reverse triangle inequality condition is not restrictive since, as we will see in Corollary~\ref{cor:11}, 
for $v^{\lambda,n}$ this is naturally induced via the repeated restriction to the bulk region $B$ 
in its recursive definition (see~\eqref{eq:114}) and the corresponding assumption placed on $\frakf$.

\begin{proof}[of Lemma~\ref{lem:nuisance} for $\sigma=+$]
Let $\chi^N \eqdef (\lambda+\calS)^{-\half} \calA^N_+ (\lambda+\calS+\calS\calG^\lambda)^{-\half} \phi^B$. 
We now expand both $\calA^N_+$ and $\calA^B_+$, in the definition of $\phi^B$, in terms of their 
constituents. More precisely, we write $\chi^N = \sum_{i,j} \chi^N[i, j]$, the sum running over 
$i\neq j\in\{1,\dots,n+1\}$, with $\chi^N[i, j]$ the term in which $p_i$ was created by $\calA^N_+$ and $p_j$ by 
$\calA^B_+$ and defined according to
\begin{equ}[e:chiNij]
\chi^N[i, j] \eqdef (\lambda+\calS)^{-\half} \genAN{i}(\lambda+\calS+\calS\calG^\lambda)^{-\half} \phi^B[j_i]
\end{equ}
where $j_i\eqdef j$ if $i>j$ and $j_i\eqdef j-1$ if $i\leq j$, and $\phi^B[\cdot]$ given by~\eqref{e:phii}. 
For later use, note that $\chi^N[i,j]$ and $\phi^B[j_i]$ have an explicit representation, which is 
\begin{align}
\chi^N[i, j](p_{1:n+1}) &= \frac{\coup}{n+1} \frac{N(p_i, p_{[1:n+1\setminus i]})(p_i \cdot p_{[1:n+1\setminus i]})\phi^B[j_i](p_{1:n+1\setminus i})}{(\lambda+\half|p_{[1:n+1]}|^2)^{\half}(\lambda+\half|p_{[1:n+1\setminus i]}|^2[1+\cg^\lambda])^\half}  \label{eq:54} \\
\phi^B[j_i](p_{1:n+1\setminus i}) &=  \frac{\sqrt{2}\coup}{n} \frac{B(p_j, p_{[1:n+1\setminus i,j]})(p_j \cdot p_{[1:n+1\setminus i,j]})\psi(p_{1:n+1\setminus i,j})}{(\lambda+\half|p_{[1:n+1\setminus i]}|^2[1+\cg^\lambda])^\half|p_{[1:n+1\setminus i,j]}|}  \label{eq:55}
\end{align}
where, to shorten the notation, we omitted the argument of $\cg^\lambda$. Furthermore, the following basic 
bounds, which will be used throughout the proof, hold
\begin{align}
|\chi^N[i, j](p_{1:n+1})| &\lesssim \frac{\coup}{n+1} \frac{N(p_i, p_{[1:n+1\setminus i]})|p_i| |\phi^B[j_i](p_{1:n+1\setminus i})|}{(\lambda+|p_{[1:n+1]}|^2)^\half}\,, \label{eq:56} \\
|\phi^B[j_i](p_{1:n+1\setminus i})| &\lesssim  \frac{\coup}{n} \frac{B(p_j, p_{[1:n+1\setminus i,j]})|p_j| |\psi(p_{1:n+1\setminus i,j})|}{(\lambda+|p_j|^2 + |p_{[1:n+1\setminus i,j]}|^2)^\half}\,. \label{eq:57}
\end{align}
where \eqref{eq:57} follows from invoking the bulk region $B$.

Our goal is to estimate the norm of $\chi^N$. By a simple change of variables, we have 
\begin{equation}
\label{eq:15}
\lVert \chi^N \rVert^2 = n(n+1)\big[ \one + \two + \three + \four \big]
\end{equation}
where
\begin{align*}
\one &= \sum_{j = 3}^{n+1}\langle\chi^N[1,2], \chi^N[1,j] \rangle\,,\qquad\qquad\two = \sum_{j = 3}^{n+1}\langle\chi^N[1,2], \chi^N[2,j] \rangle\,, \\
\three &= \sum_{j =3}^{n+1}\sum_{k\neq j, k=1}^{n+1} \langle\chi^N[1,2], \chi^N[j,k] \rangle\,, \,\,\, \four = \langle\chi^N[1,2], \chi^N[1,2]+\chi^N[2,1]\rangle. 
\end{align*}
We will bound each of the terms $\one-\four$ separately and ultimately deduce that~\eqref{eq:AR0} holds for $\calA^N_+$. 
\medskip

\noindent{\it Bound on $\one$ and $\two$.} We claim that, uniformly over $\lambda\in(0,1)$ and $n\in\N$, we have  
\begin{equation} \label{eq:44}
\one\vee\two\lesssim  \frac{\coup^2}{n} \lVert \psi \rVert^2\,.
\end{equation}
To see~\eqref{eq:44} for $\one$, we bound $\chi^N[1,2]$ and $ \chi^N[1,j]$ in their inner product by~\eqref{eq:56} and 
the indicator of the nuisance region by $1$, thus obtaining
\begin{equs}
&|\langle\chi^N[1,2], \chi^N[1,j] \rangle|\\
&\lesssim \frac{\gamma^2}{(n+1)^2}\int \frac{|p_1|^2 |\phi^B[1](p_{2:n+1})\phi^B[j-1](p_{2:n+1})| }{(\lambda+|p_{[1:n+1]}|^2)} \mu_{n+1}( \dd p _{1:n+1})\,,\\
&\leq \frac{1}{(n+1)}\int  |\phi^B[1](p_{2:n+1})\phi^B[j-1](p_{2:n+1})|\mu_n(\dd p_{1:n}) 
\end{equs}
where we used the definition of $\mu_n$ in~\eqref{e:mun}, and bounded the integral over $p_1$ as in~\eqref{eq:16}. 
Since $j> 2$, the last expression can be bounded as the off-diagonal term in Lemma~\ref{lem:6} 
and~\eqref{eq:44} easily follows upon summing up over $j$. 

We now turn to $\two$. Let $R\subset \R^{2(n+1)}$ be the region in which both 
$N(p_1, p_{[2:n+1]})$ and $N(p_2, p_{[1:n+1\setminus 2]})$ are $1$. Then, applying~\eqref{eq:56} and~\eqref{eq:57} 
to $\chi^N[1,2]$ and $ \chi^N[2,j]$ we see that their scalar product is bounded by 
\begin{equ}
\frac{\coup^4}{n^2(n+1)^2} \int_R \frac{|p_{[3:n+1]}| |p_{[1:n+1\setminus 2,j]}|}{\lambda+|p_{[1:n+1]}|^2} \Phi(p_{1:n+1}) \Phi^\prime(p_{1:n+1}) \mu_{n+1}(\dd p _{1:n+1})
\end{equ}
where $\Phi$ and $\Phi'$ are given by
\begin{equs}
\Phi(p_{1:n+1}) &\eqdef \frac{|p_1||p_2| |\psi(p_{3:n+1})|}{|p_{[1:n+1\setminus 2,j]}| (\lambda+|p_2|^2+|p_{[3:n+1]}|^2)^\half}\,, \\
\Phi^\prime(p_{1:n+1}) &\eqdef \frac{|p_2||p_j| |\psi(p_{1:n+1\setminus 2,j})|}{|p_{[3:n+1]}|(\lambda+|p_j|^2+|p_{[1:n+1\setminus 2,j]}|^2)^\half}\,.
\end{equs}
By Cauchy-Schwartz, we obtain two terms. The square of the term associated to $\Phi^\prime$ is 
\begin{equs}
&\frac{\coup^4}{n^2(n+1)^2} \int_R \frac{|p_2|^2|p_j|^2|p_{[1:n+1\setminus 2,j]}||\psi(p_{1:n+1\setminus 2,j})|^2\mu_{n+1}(\dd p _{1:n+1})}{(\lambda+|p_{[1:n+1]}|^2)|p_{[3:n+1]}|(\lambda+|p_j|^2+|p_{[1:n+1\setminus 2,j]}|^2)} \\
&\lesssim \frac{\coup^2}{n(n+1)} \int |\psi(p_{1:n-1})|^2\mu_{n-1}(\dd p _{1:n-1})\int\frac{ |p_{[1:n-1]}|\hat V(q)\dd q}{|q+r|(\lambda+|q|^2+|p_{[1:n-1]}|^2)} \\
&\lesssim \frac{\gamma^2}{n^2} \|\psi\|^2\label{eq:62}
\end{equs}
where in the first step, we used the definition of $\mu_n$ in~\eqref{e:mun}, applied~\eqref{eq:16} to the integral over $p_2$ 
and then changed variables ($p_j\mapsto q$ and $p_{1:n+1\setminus 2,j}\mapsto p_{1:n-1}$) and, to shorten the 
notation, set $r = p_{[2:n-1]}$, while in the second we applied Lemma \ref{lem:4} to the inner integral.  

The square of the term involving $\Phi$ is given by 
\begin{equation}
\label{eq:61}
\frac{\coup^4}{n^2(n+1)^2} \int_R \frac{|p_1|^2|p_2|^2|p_{[3:n+1]}||\psi(p_{3:n+1})|^2\mu_{n+1}(\dd p _{1:n+1})}{(\lambda+|p_{[1:n+1]}|^2)|p_{[1:n+1\setminus 2,j]}|(\lambda+|p_2|^2+|p_{[3:n+1]}|^2)}\,.
\end{equation}
To bound it we would like to lowerbound $|p_2|^2 + |p_{[3:n+1]}|^2 \gtrsim |p_1|^2 + |p_{[3:n+1]}|^2$ so as to give an upper bound for the last denominator.
This is clearly possible for $|p_{[3:n+1]}| > \frac{1}{3}|p_1|$. If instead $|p_{[3:n+1]}| \le \frac{1}{3}|p_1|$, 
we can use the fact that, by definition, $N(p_1,p_{2:n+1})=1$ on $R$, which guarantees that 
$|p_1+p_2|-|p_{[3:n+1]}| \le \frac{1}{2}|p_1|$, from which we deduce $||p_1|-|p_2|| \le \frac{5}{6}|p_1|$ 
and therefore conclude $|p_2|\gtrsim |p_1|$. 
As a consequence,~\eqref{eq:61} is bounded above by a constant times 
\begin{equ}
\frac{\coup^4}{n^2(n+1)^2} \int_R \frac{|p_1|^2|p_2|^2|p_{[3:n+1]}||\psi(p_{3:n+1})|^2\mu_{n+1}(\dd p _{1:n+1})}{(\lambda+|p_{[1:n+1]}|^2)|p_{[1:n+1\setminus 2,j]}|(\lambda+|p_1|^2+|p_{[3:n+1]}|^2)}\,.
\end{equ}
First, we integrate in $p_2$, using~\eqref{eq:16}, and then over $p_1$, arguing as in~\eqref{eq:62}. 
Putting these bounds together, we have shown that also $\two$ satisfies~\eqref{eq:44}. 
\medskip

\noindent{\it Bound on $\three$.} We claim that, uniformly over $\lambda\in(0,1)$ and $n\in\N$, we have
\begin{equation} \label{eq:59}
\three\lesssim c^{-4} \coup^2 \lVert \psi \rVert^2\,,
\end{equation}
where $c$ is the constant for which~\eqref{a:RTI} holds. Thanks to~\eqref{eq:56} and~\eqref{eq:57}, 
we can upper bound the scalar product between $\chi^N[1,2]$ and $\chi^N[j,k]$, for $j\in\{3,\dots, n+1\}$ 
and $k\in\{1,\dots,n+1\}\setminus \{j\}$, via 
\begin{equs}
\langle\chi^N[1,2], \chi^N[j,k] \rangle\lesssim \frac{\coup^4}{n^2(n+1)^2} \int_{R\cap R'} \frac{\Phi(p_{1:n+1}) \Phi^\prime(p_{1:n+1}) }{\lambda+|p_{[1:n+1]}|^2}  \mu_{n+1}(\dd p _{1:n+1})
\end{equs}
where $R\subset \R^{2(n+1)}$ is the region in which 
\begin{equation}
\label{eq:65}
N(p_1, p_{[2:n+1]}) B(p_2, p_{[3:n+1]}) \1_{\{|p_{[3:n+1]}| \ge c \sum_{i=3}^{n+1}|p_i|\}}\equiv 1
\end{equation}
with $R^\prime$ defined analogously, and $\Phi$ is defined as 
\begin{equs}
\Phi(p_{1:n+1}) = \frac{|p_1||p_2| |\psi(p_{3:n+1})|}{(\lambda+|p_2|^2+ |p_{[3:n+1]}|^2)^\half}
\end{equs}
while $\Phi'$ has the same expression as $\Phi$ but with $1,2$ replaced by $j,k$ and $[3:n+1]$ by $[1:n+1\setminus j,k]$. 
Note that our choice of $R$ is justified thanks to the reverse triangle inequality assumption placed on $\psi$. 
Using $2\Phi\Phi'\leq \Phi^2 + (\Phi')^2$, we obtain two terms which can be similarly 
bounded, so we will only focus on the one depending on $\Phi$. 
This is   
\begin{equation}\label{eq:64} 
\frac{\coup^4}{n(n+1)} \int_{R \cap R^\prime}  \frac{|\psi(p_{3:n+1})|^2 \VHat(p_2)\VHat(p_1) \dd p_1 \dd p_2\mu_{n-1}(\dd p _{3:n+1})}{(\lambda+|p_{[1:n+1]}|^2)(\lambda+|p_2|^2+|p_{[3:n+1]}|^2)}\,.   
\end{equation}
On $R$, the reverse triangular inequality ensures that 
$|p_{[3:n+1]}| \ge  c |p_j|$, and similarly, on $R'$, 
$c |p_2| \le  |p_{[1:n+1\setminus j, k]}| \lesssim |p_{[1:n+1\setminus j]}| \le |p_{[1:n+1]}| + |p_j|$. Since further, $N(p_j, p_{[1:n+1\setminus j]})=1$, 
we deduce that $|p_2|\lesssim c^{-1} |p_j|$. Therefore,~\eqref{eq:64} can be bounded above 
by 
\begin{equs}
&\frac{\coup^4}{n(n+1)} \int \frac{\1_{\{|p_2|\lesssim c^{-1}|p_j|\}}}{\lambda + c^2|p_j|^2}  \frac{|\psi(p_{3:n+1})|^2 \VHat(p_2)\VHat(p_1) \dd p_1 \dd p_2\mu_{n-1}(\dd p _{3:n+1})}{(\lambda+|p_{[1:n+1]}|^2)}\\
&\lesssim  \frac{\coup^2}{n(n+1)} \int  \frac{ |\psi(p_{3:n+1})|^2}{\lambda+c^2|p_j|^2} \Big(  \int_{|p_2| \lesssim c^{-1} |p_j|} \dd p _2\Big) \mu_{n-1}(\dd p _{3:n+1}) \lesssim c^{-4}\frac{\gamma^2}{n^2}\|\psi\|^2
\end{equs}
where we first integrated in $p_1$ using~\eqref{eq:16}. Hence,~\eqref{eq:59} easily follows.  
\medskip

\noindent{\it Bound on $\four$.} Among those treated so far, this is the most interesting term since it is 
the one where a non-trivial cancellation takes place and therefore a finer analysis, which 
uses the explicit form of $\chi^N[1,2]$ and $\chi^N[2,1]$ is needed. 
By~\eqref{eq:54},~\eqref{eq:55} and the explicit expression for the measure $\mu^N$ in~\eqref{e:mun}, 
we can write
\begin{equ}[e:BoundFour]
\four  \lesssim \frac{\coup^2}{n(n+1)} \int \mu_{n-1}(\dd p_{3:n+1}) |\psi(p_{3:n+1})|^2 \calI(p_{3:n+1})  
\end{equ}
where $\calI(p_{3:n+1}) $ is defined according to  
\begin{equs}[e:IBoundFour]
\calI(p_{3:n+1}) &\eqdef \coup^2\int \dd p_{1:2} \frac{\hat V(p_1)\hat V(p_2)}{|p_1||p_2||p_{[3:n+1]}|} \frac{N(p_1,p_{[2:n+1]}) B(p_2, p_{[3:n+1]})}{(\lambda +|p_{[1:n+1]}|^2)(\lambda +|p_2|^2 +|p_{[3:n+1]}|^2)^\half}\\
&\times \Big|\frac{N(p_1,p_{[2:n+1]}) B(p_2, p_{[3:n+1]}) (p_1\cdot p_{[2:n+1]})(p_2\cdot p_{[3:n+1]})}{\lambda +\half |p_{[2:n+1]}|^2[1+\cg^\lambda(p_{[2:n+1]})]}\\
&\qquad+\frac{N(p_2,p_{[1:n+1\setminus 2]}) B(p_1, p_{[3:n+1]}) (p_2\cdot p_{[1:n+1\setminus2]})(p_1\cdot p_{[3:n+1]})}{\lambda +\half |p_{[1:n+1\setminus 2]}|^2[1+\cg^\lambda(p_{[1:n+1\setminus 2]})]}\Big| 
\end{equs}
and~\eqref{e:BoundFour} follows by applying~\eqref{eq:56} and~\eqref{eq:57} to the first $\chi^N[1,2]$.  
As a consequence of Lemma~\ref{lem:35} in Appendix~\ref{app:EstNregion}, $\calI$ is bounded uniformly 
over $p_{3:n+1}\in \R^{2(n+1)}$, which then implies that 
\begin{equation}\label{eq:19}
\four \lesssim \frac{\gamma^2}{n^2}\|\psi\|^2\,.
\end{equation}

\noindent{\it Conclusion.} By collecting the bounds obtained in~\eqref{eq:44},~\eqref{eq:59} and~\eqref{eq:19}, 
and plugging them into~\eqref{eq:15},~\eqref{eq:AR0} for $\sigma=+$ follows at once. 
\end{proof}

Before proving Lemma~\ref{lem:nuisance} for $\sigma=-$, we show the following basic lemma, which ensures that 
the reverse triangle inequality is propagated by the operator $\calA^B_+$. 

\begin{lemma}\label{lem:RTI}
If $\psi\in\fock_{n-1}$ satisfies the reverse triangle with respect to $c\in(0,\tfrac12)$ (see~\eqref{a:RTI}), 
then so does $\phi^B$ with respect to $\tfrac{c}6$. 
\end{lemma}
\begin{proof}
Recall the definition of $\phi^B$ in Lemma~\ref{lem:DiOff}, and the fact that $\phi^B = \sum_{j=1}^n \phi^B[j]$, 
with $\phi^B[j]$ as in~\eqref{e:phii} and the bulk region in~\eqref{e:BNregions} with $\kappa=\tfrac13$. 
Then, for each $j$ we have 
\begin{align*}
\text{supp}\{\phi^B[j] \} &\subseteq \Big\{|p_{[1:n]}| \ge \tfrac16|p_{[1:n\setminus j]}|+\tfrac18|p_j| \,,\,|p_{[1:n \setminus j]}| \ge c \sum_{i=1, i \ne j}^n |p_i| \Big\}
\end{align*}
so that, by the triangle inequality and the fact that $\frac{c}{6} \le \frac{1}{8}$, the statement follows. 
\end{proof}

\begin{proof}[of Lemma~\ref{lem:nuisance} for $\sigma=-$] 

By definition of $\calA^N_-$ in~\eqref{eq:AR}, we have the bound
\begin{equation*}
|(\lambda+\calS)^{-\half} \calA^N_- (\lambda+\calS+\calS\calG^\lambda)^{-\half} \phi^B (p_{1:n-1})| \lesssim \coup n  \int_R \frac{\VHat(q)}{|q|^2} \frac{|q| |\phi^B(q, p_{1:n-1})|}{|q+p_{[1:n-1]}|} \dd q
\end{equation*}
where, thanks to Lemma~\ref{lem:RTI}, we can take $R\subset \R^{2n}$ as the region in which  
$N(q, p_{[1:n-1]}) = 1$ intersected with 
$\{ |q+p_{[1:n-1]}| \ge \tfrac{c}{6} \big( |q| +  \sum_{j=1}^{n-1} |p_j| \big) \}$.  
Defining $R'$ analogously to $R$, we get 
\begin{equs}
&\| (\lambda+\calS)^{-\half} \calA_-^N (\lambda+\calS+\calS\calG^\lambda)^{-\half} \phi^B \|^2\label{e:A-StL}\\
&\leq \coup^2 n^2\int_{R \cap R^\prime} \mu_{n-1}(p_{1:n-1})\frac{ |\phi^B(q^\prime, p_{1:n-1})|}{|q+p_{[1:n-1]}||q^\prime|}  \frac{|\phi^B(q, p_{1:n-1})|}{|q^\prime+p_{[1:n-1]}||q|}\VHat(q)\VHat(q^\prime)\dd q \dd q^\prime\\
&\lesssim \coup^2 n^2\int \mu_{n-1}(p_{1:n-1})\int_R\frac{\VHat(q)\dd q}{|q+p_{[1:n-1]}|^2}  \int_{R^\prime}  |\phi^B(q^\prime, p_{1:n-1})|^2\frac{\VHat(q^\prime)\dd q^\prime}{|q^\prime|^2}
\end{equs}
where in the last step we used $2ab\le a^2+b^2$. Now, thanks to the region $R$, 
the integral over $q$ can be estimated as 
 \begin{align*}
\int_R\frac{\VHat(q)\dd q }{|q+p_{[1:n-1]}|^2} \lesssim \frac{1}{c^2\Big( \sum_{i=1}^{n-1} |p_i| \Big)^{2}} \int _{\{|q+p_{[1:n-1]}| \le \tfrac{1}{3}|p_{[1:n-1]}| \}} \dd q \lesssim c^{-2}\,.
\end{align*}
From this and the definition of $\mu_n$ in~\eqref{e:mun}, we deduce that~\eqref{e:A-StL} is bounded above by 
\begin{equ}
c^{-2} \coup^2 n\|\phi^B\|
\lesssim c^{-2}\coup^2 n^2\|\psi\|^2
\end{equ}
where the last step is a consequence of~\eqref{eq:AR9}. 
\end{proof}

We conclude this subsection by showing that the solution of the replacement equation 
$v^{\lambda,n}$ in Definition~\ref{def:3} satisfies the assumptions of Lemma~\ref{lem:nuisance}.

\begin{corollary}\label{cor:11}
Let $m\in\N$, $\frakf\in\fock_m$ and, for $\lambda\in(0,1)$ and $n\geq m$, let $v^{\lambda,n}$ be the solution 
to the replacement equation in Definition~\ref{def:3}. If $\frakf$ satisfies~\eqref{a:RTI} for some 
$c\in(0,\tfrac12)$, then so does $v^{\lambda,n}_j$, for every $j=m,\dots,n$, with $c$ replaced by $c \,6^{m-j}$. 

In particular, there exists a constant $C_N=C_N(c,n)>0$ such that 
\begin{equation}
\label{eq:17}
\lVert (\lambda+\calS)^{-\half} \calA^N v^{\lambda,n} \rVert^2 \leq C_N \coup^2 \lVert \calS^\half v^{\lambda, n} \rVert^2 +  \vertiii{ \frakf }^2
\end{equation}
where the norm on $\frakf$ at the right hand side is that in~\eqref{e:VanNorm}. The $C_N$ above 
satisfies $C_N\lesssim c^{-4} n^2 6^{4n}$. 
\end{corollary}
\begin{proof}
The proof follows by induction over $j$. By~\eqref{eq:114}, for $j=m$, we clearly have 
$\supp\{v^{\lambda,n}_m\}=\supp\{\frakf\}$. Assume the result holds for all $m\leq i\leq j$. 
Set $\psi\eqdef\calS^{\frac12}v^{\lambda,n}_j$ and note that the definition of $\phi^B$ in Lemma~\ref{lem:DiOff} 
implies that 
\begin{equ}[e:vphi]
(\lambda+\calS+\calS\calG^\lambda)^{\half}v^{\lambda,n}_{j+1} = \phi^B\,.
\end{equ} 
Now, $\lambda+\calS+\calS\calG^\lambda$ is a diagonal operator, hence it is does not alter 
the support of the functions to which it is applied. Therefore,   
$\supp\{v^{\lambda,n}_{j+1}\}=\supp\{\phi^B\}$ and the reverse triangle inequality is then 
implied by Lemma~\ref{lem:RTI}. 

Concerning~\eqref{eq:17}, we apply Lemma~\ref{lem:nuisance} to each of the $v^{\lambda,n}_{j}$. 
The term $\vertiii{ \frakf }$ covers the case $j=m$. For $j>m$, 
using the notation in~\eqref{e:vphi} we have 
\begin{equs}
\|(\lambda+\calS)^{-\half}\calA^N v^{\lambda,n}_{j+1}\|&=\|(\lambda+\calS)^{-\half}\calA^N (\lambda+\calS+\calS\calG^\lambda)^{-\half} \phi^B\|\\
&\lesssim \gamma^2 j^2 c^{-4} 6^{4(m-j)}\|\calS^\half v^{\lambda,n}_j\|\leq C_N \gamma^2 \|\calS^\half v^{\lambda,n}_j\|
\end{equs}
so that, adding the above bounds over $j$,~\eqref{eq:17} follows at once. 
\end{proof}

\subsection{A priori estimates}\label{sec:apriori}

In order to prove Proposition~\ref{p:MainEstimates}, we need a priori estimates on the solution $v^{\lambda,n}$
which guarantee a polynomial decay in the chaos. These estimates 
are better expressed in terms of the so-called number operator $\calN : \fock \rightarrow \fock$ 
which is defined as $\calN \psi = n \psi$ for $\psi\in\fock_n$. 

First we state the desired property for the solution $u^{\lambda, n}$ of the truncated resolvent equation in~\eqref{eq:5}. 
The proof of the next lemma is completely analogous to that in~\cite[Proposition 2.8]{CannGubiToni23_GaussianFluctuations} and therefore omitted. 

\begin{lemma}
\label{lem:18}
For $m,k\in\N$ there exists a constant $C = C(m,k)>0$ such that for all $\lambda\in(0,1)$, $\frakf \in \fock_m$, and $n\geq m$ it holds that
\begin{equation}
\label{eq:41}
\lVert \calN^k (\lambda + \calS)^\half u^{\lambda, n} \rVert^2 \le C \lVert (\lambda + \calS)^{-\half} \frakf \rVert^2\,.
\end{equation}
where $u^{\lambda,n}$ is the solution of the truncated resolvent equation in~\eqref{eq:5} with input $\frakf$. 
\end{lemma}

In the next lemma, we port over the estimate of Lemma \ref{lem:18} from $u^{\lambda, n}$ to 
the solution $v^{\lambda, n}$ of the replacement equation in Definition~\ref{def:3}. 

\begin{lemma}\label{lem:17}
For $m,k\in\N$ there exists a constant $C = C(m,k)>0$ such that for all $n\geq m, c \in (0,1)$ there exists $\bar\lambda=\bar\lambda(k,n,c)\in(0,1)$, such that 
for every $\lambda<\bar\lambda$, and for every $\frakf \in \fock_m$ satisfying the assumptions of Corollary~\ref{cor:11}, it holds that
\begin{equation}
\label{eq:77}
\lVert \calN^k (\lambda + \calS)^\half v^{\lambda, n} \rVert^2 \lesssim C \big( \lVert (\lambda + \calS)^{-\half} \frakf\rVert^2 + (n^{2k}-1) \vertiii{ \frakf }^2 \big)\,,
\end{equation}
where $v^{\lambda, n}$ be the solution of the replacement equation in Definition~\ref{def:3} with input $\frakf$, and $\vertiii{ \cdot }$ is defined in~\eqref{e:VanNorm}.
\end{lemma}
\begin{proof}
We begin by showing~\eqref{eq:77} holds if $k=0$. 
Testing both sides of~\eqref{eq:27} with $v^{\lambda,n}$ and using the Replacement Lemma \ref{lem:3} with 
$\psi_i=\calS^\half v^{\lambda,n-1}$, $i=1,2$, we obtain 
\begin{equation}
\label{eq:24}
\lVert (\lambda+\calS)^{\half} v^{\lambda,n} \rVert^2 \le \langle v^{\lambda,n} , \frakf \rangle + C n \coup^2 \lVert \calS^\half v^{\lambda,n-1} \rVert^2\,.
\end{equation}
The above bound holds since by orthogonality of Fock spaces with different indices, 
$\langle v^{\lambda,n}_n, \calR_\lambda v^{\lambda,n-1} \rangle$, $\langle v^{\lambda,n}, \calA^B_+v^{\lambda,n}_n\rangle$, 
$\langle v^{\lambda,n}, \calA^B_+v^{\lambda,n}_n\rangle$ vanish, 
while $\langle v^{\lambda,n}, \calA^N v^{\lambda,n}\rangle = 0$ since $\calA^N$ is antisymmetric 
and $\langle v^{\lambda,n}, \calS\calG^\lambda v^{\lambda,n}_n\rangle = \langle v^{\lambda,n}_n, \calS\calG^\lambda v^{\lambda,n}_n\rangle \ge 0$. 

For the first summand at the right hand side of~\eqref{eq:24}, we apply Cauchy-Schwarz so that 
\begin{equ}
|\langle v^{\lambda,n} , \frakf \rangle| \le \half\lVert (\lambda+\calS)^{\half} v^{\lambda,n} \rVert^2 + \half\lVert (\lambda+\calS)^{-\half} \frakf \rVert^2\,.
\end{equ}
For the other, we bound $\lVert \calS^\half v^{\lambda,n-1} \rVert\leq \lVert (\lambda+\calS)^\half v^{\lambda,n} \rVert$ 
and then choose $\lambda$ sufficiently small so that $C n \coup^2<1/2$. 
In this way, we can reabsorb it into the left hand side and derive~\eqref{eq:77} for $k=0$. 
\medskip

We turn to the case $k>0$. Let $u^{\lambda, n}$ be the solution of the 
truncated resolvent equation in~\eqref{eq:5} with input $\frakf$. We first want to control 
$\lVert(\lambda + \calS)^\half (v^{\lambda, n} - u^{\lambda, n}) \rVert^2$. 
Using~\eqref{eq:27}, we compute 
\begin{equation*}
\begin{split}
(\lambda - \calL) (v^{\lambda, n} - u^{\lambda, n}) = &-\calA^B_-(\lambda+\calS+\calS\calG^\lambda)^{-1}\frakf + \calR_\lambda v^{\lambda,n-1} \\ &- \calS\calG^\lambda v^{\lambda,n}_n - \calA_+(v^{\lambda,n}_n- u^{\lambda,n}_n) - \calA^N v^{\lambda,n-1}
\end{split}
\end{equation*}
so that, by testing both sides by $v^{\lambda, n} - u^{\lambda, n}$ and using orthogonality of 
Fock spaces with different indices, we get 
\begin{equs}[e:TestDifference]
\lVert(\lambda + \calS)^\half &(v^{\lambda, n} - u^{\lambda, n}) \rVert^2\\
\leq& \langle u^{\lambda,n}_{m-1}, \calA^B_-(\lambda+\calS+\calS\calG^\lambda)^{-1}\frakf\rangle + \langle v^{\lambda, n} - u^{\lambda, n}, \calR_\lambda v^{\lambda,n-1}\rangle\\
&-\langle v^{\lambda, n} - u^{\lambda, n}, \calS\calG^\lambda v^{\lambda,n}_n\rangle+\langle v^{\lambda, n} - u^{\lambda, n}, \calA^N v^{\lambda,n-1} \rangle
\end{equs}
and we separately bound each of the terms at the right hand side. 
For the first, we apply Cauchy-Schwarz and obtain a bound of the form 
\begin{equ}[e:oneWE]
 \|(\lambda +\calS)^{\half} u^{\lambda,n}\|\|(\lambda+\calS)^{-\half}\calA^B_-(\lambda+\calS+\calS\calG^\lambda)^{-1}\frakf\|\lesssim \|(\lambda+\calS)^{-\half}\frakf\|\vertiii{\frakf}
\end{equ}
where we applied~\eqref{eq:41} and the definition of the norm in~\eqref{e:VanNorm}. 
Thanks to the Replacement Lemma~\ref{lem:3}, the second is bounded by 
\begin{align}
C&\coup^2n \lVert(\lambda + \calS)^\half (v^{\lambda, n} - u^{\lambda, n}) \rVert \lVert(\lambda + \calS)^\half v^{\lambda, n} \rVert \nonumber\\
&\lesssim \coup^2n \big( \lVert(\lambda + \calS)^\half (v^{\lambda, n} - u^{\lambda, n}) \rVert^2 + \lVert(\lambda + \calS)^\half \frakf \rVert^2 \big) \label{eq:80}
\end{align}
where, in the last inequality, we used~\eqref{eq:77} with $k=0$. 
The third term satisfies 
\begin{align}
-\langle v^{\lambda, n} - u^{\lambda, n}, \calS\calG^\lambda v^{\lambda,n}_n\rangle \le |\langle u^{\lambda, n}, \calS\calG^\lambda v^{\lambda,n}_n\rangle| &\lesssim \lVert(\lambda + \calS)^\half u^\lambda_n \rVert \lVert(\lambda + \calS)^\half v^{\lambda,n}_n \rVert \nonumber\\
&\lesssim_k n^{-2k} \lVert(\lambda + \calS)^\half \frakf \rVert^2\label{e:threeWE}
\end{align}
where the last line comes from invoking both~\eqref{eq:77} with $k=0$ and Lemma \ref{lem:18}. 
Finally, the fourth term is bounded above by 
\begin{equ}
 \frac{1}{16} \lVert(\lambda + \calS)^\half (v^{\lambda, n} - u^{\lambda, n}) \rVert^2 + 4 \lVert(\lambda + \calS)^{-\half} \calA^N v^{\lambda,n} \rVert^2\,.
\end{equ}
To control the second summand, we apply first~\eqref{eq:17} and then~\eqref{eq:77} with $k=0$, so that 
\begin{equs}[e:nuisa]
\lVert(\lambda + \calS)^{-\half} \calA^N v^{\lambda,n} \rVert^2&\leq C_N\gamma^2 \lVert(\lambda + \calS)^{\half} v^{\lambda,n} \rVert^2+\vertiii{ \frakf }^2\\
&\leq C_N\gamma^2 \|(\lambda+\calS)^{-\half}\frakf\|^2 + \vertiii{ \frakf }^2\,.
\end{equs}
Therefore, the fourth term is bounded by
\begin{equ}[e:fourWE]
 \frac{1}{16} \lVert(\lambda + \calS)^\half (v^{\lambda, n} - u^{\lambda, n}) \rVert^2+C_N\gamma^2 \|(\lambda+\calS)^{-\half}\frakf\| + \vertiii{ \frakf }^2\,.
\end{equ}
By plugging~\eqref{e:oneWE},~\eqref{eq:80},~\eqref{e:threeWE} and~\eqref{e:fourWE} into~\eqref{e:TestDifference}, 
and taking $\lambda$ sufficiently small so to reabsorb all the terms containing 
$\lVert(\lambda + \calS)^\half (v^{\lambda, n} - u^{\lambda, n}) \rVert^2$, we conclude that the latter satisfies
\begin{equs}[e:BoundDifference]
\lVert(\lambda + &\calS)^\half (v^{\lambda, n} - u^{\lambda, n}) \rVert^2\\
&\lesssim (n^{-2k}+C_N\coup^2)\|(\lambda+\cS)^{\half}\frakf\|^2+\vertiii{ \frakf }^2+\|(\lambda+\cS)^{\half}\frakf\|\vertiii{ \frakf }\,.
\end{equs}
Simply bounding 
\begin{equ}
\|\cN^k(\lambda + \calS)^\half v^{\lambda, n} \rVert^2\leq \lVert \cN^k(\lambda + \calS)^\half (v^{\lambda, n} - u^{\lambda, n}) \rVert^2 + \lVert \cN^k(\lambda + \calS)^\half u^{\lambda, n} \rVert^2\,,
\end{equ}
using~\eqref{eq:41} and~\eqref{e:BoundDifference} together with choosing $\lambda$ sufficiently 
small depending on $n$,~\eqref{eq:77} follows. 
\end{proof}

\subsection{Proof of Proposition \ref{p:MainEstimates}}
\label{sec:remainder-terms}

We are now ready to collect the estimates obtained so far and prove Proposition~\ref{p:MainEstimates}. 
First, we state the next lemma which will immediately imply the bound~\eqref{eq:L2Bound}. 

\begin{lemma}\label{lem:L2control}
For $\lambda \in (0,1)$, $n \in \bbN$ and $\psi \in \fock_{n}$ it holds that 
\begin{equation}
\label{eq:124}
\lVert \sqrt{\lambda} (\lambda+\calS+\calS\calG^\lambda)^{-\half} \phi^B \rVert^2 \lesssim n \coup^2 \lVert \psi \rVert^2\,.
\end{equation}
where $\phi^B$ is as in Lemma \ref{lem:DiOff} under the choice $R = B$ and $\calT = \calS\calG^\lambda$. 
\end{lemma}
\begin{proof}
Consider the decomposition $\phi^B = \sum_{j=1}^{n+1} \phi^B[j]$, 
with $\phi^B[j]$ as in~\eqref{e:phii}. Applying Cauchy-Schwarz to the off-diagonal terms, we see that 
\begin{equ}
\lVert \sqrt{\lambda} (\lambda+\calS+\calS\calG^\lambda)^{-\half} \phi^B \rVert^2\lesssim n^2\lVert \sqrt{\lambda} (\lambda+\calS+\calS\calG^\lambda)^{-\half} \phi^B[n+1] \rVert^2
\end{equ}
where we also used that, by a simple change of variables 
$\lVert \sqrt{\lambda} (\lambda+\calS+\calS\calG^\lambda)^{-\half} \phi^B[j] \rVert$ is independent of $j$. 
Using the definition of $\phi^B[n+1]$ in~\eqref{eq:AR8} and $\mu_{n+1}$ in~\eqref{e:mun}, 
and the fact that $\calS\calG^\lambda\geq 0$, we obtain 
\begin{equs}
n^2&\lVert \sqrt{\lambda} (\lambda+\calS+\calS\calG^\lambda)^{-\half} \phi^B[n+1] \rVert^2\\
&\lesssim n\gamma^2 \int \mu_n(\dd p_{1:n}) |\psi(p_{1:n})|^2 \Big(\lambda\int \dd p_{n+1}\frac{V(p_{n+1})}{(\lambda+|p_{[1:n+1]}|^2)^2}\Big)\lesssim n\gamma^2 \|\psi\|^2
\end{equs}
the last bound being a consequence of the fact that the quantity in parenthesis is bounded uniformly over $\lambda$. 
\end{proof}

We are now ready to complete the proof of Proposition~\ref{p:MainEstimates}. 

\begin{proof}[of Proposition~\ref{p:MainEstimates}]
Let us begin by showing~\eqref{eq:L2Bound}. For this, notice that, by~\eqref{eq:114}, 
\begin{equ}[e:j=m]
v^{\lambda,n}_m=(\lambda + \calS+\calS\calG^\lambda)^{-1}\frakf
\end{equ} 
so that we only need to focus on $v^{\lambda,n}_j$ for $j>m$. 
As in the proof of Corollary~\ref{cor:11}, 
upon setting $\psi\eqdef\calS^\half v^{\lambda,n}_j$,~\eqref{e:vphi} holds. Therefore, 
\begin{equ}
\lambda \|v^{\lambda,n}_{j+1}\|^2=\lVert \sqrt{\lambda} (\lambda+\calS+\calS\calG^\lambda)^{-\half} \phi^B \rVert^2 \lesssim j \coup^2 \lVert \psi \rVert^2= j\gamma^2 \|\calS^\half v^{\lambda,n}_j\|^2
\end{equ}
where the bound in the middle follows by Lemma~\ref{lem:L2control}. Summing both sides over $j=m,...,n-1$ we get 
\begin{equs}
\lambda \|v^{\lambda,n}-(\lambda + \calS+\calS\calG^\lambda)^{-1}\frakf\|^2\lesssim 
 \gamma^2 n \|(\lambda+\calS)^\half v^{\lambda,n}\|^2\lesssim  \gamma^2 n \|(\lambda+\calS)^{-\half} \frakf\|^2
\end{equs}
the last step being due to~\eqref{eq:77} with $k=0$. Hence,~\eqref{eq:L2Bound} is established.
\medskip

We now turn to~\eqref{eq:ApproxH1Bound}. By~\eqref{eq:27} and~\eqref{e:nuisa}, we 
immediately have 
\begin{equs}
\lVert (\lambda+\calS)^{-\half} \big[ &(\lambda-\calL)v^{\lambda, n} - \frakf \big] \rVert^2\lesssim \vertiii{ \frakf }^2 + C_N\gamma^2 \|(\lambda+\calS)^{-\half}\frakf\|^2\\
&+\lVert (\lambda+\calS)^{-\half}\calR_\lambda v^{\lambda,n-1}\rVert^2+\lVert (\lambda+\calS)^{-\half}(\calS\calG^\lambda + \calA^B_+)v^\lambda_n\rVert^2\,.
\end{equs}
Thanks to Lemma~\ref{lem:32} and the fact that $\calG^\lambda$ is a bounded operator, 
the final term is bounded above by 
\begin{equs}
\lVert (\lambda+\calS)^{-\half} (\calS\calG^\lambda + \calA^B_+)v^{\lambda, n}_n \rVert^2 &\lesssim (n+1) \lVert (\lambda+\calS)^\half v^{\lambda, n}_n\rVert^2 \\
&\lesssim \frac{1}{n} \lVert (\lambda + \calS)^{-\half} \frakf\rVert^2 + (n+1)\vertiii{\frakf}^2
\end{equs}
where in the last line we used~\eqref{eq:77} with $k=1$. For the penultimate term we use~\eqref{e:eq7Op}, so that 
\begin{equs}
\lVert (\lambda+\calS)^{-\half} \calR_\lambda v^{n-1} \rVert
\lesssim \coup^2 n \lVert\calS^\half v^{n-1} \rVert\lesssim \coup^2 n\lVert (\lambda + \calS)^{-\half} \frakf\rVert
\end{equs}
where the last step comes from~\eqref{eq:77} with $k=0$. 
It is then immediate to see that, by collecting all the bounds,~\eqref{eq:ApproxH1Bound} follows. 
\end{proof}

\section{The invariance principle}
\label{sec:invariance-principle}

The aim of this section is to prove Theorem~\ref{thm:1}, i.e. 
the invariance principle for the weakly self-repelling Brownian polymer 
in~\eqref{eq:sde} with the choice of $\gamma$ in~\eqref{eq:coup}, 
for which we will appeal to Theorem~\ref{thm:3}. 
Throughout the section, the family of observables $(v^{\lambda,n}\colon \lambda\in(0,1), n\in\N)$ we consider is 
given by the solution of the replacement equation in Definition~\ref{def:3} with input function $\frakf\eqdef\gamma f_1$, 
the latter being defined in~\eqref{e:fdef}. 

In the next subsections we verify conditions \cref{eq:L2,eq:ApproxH1,eq:ApproxMean} by deriving suitable bounds on the right hand sides of~\eqref{eq:L2Bound} and~\eqref{eq:ApproxH1Bound}. Then we verify~\eqref{eq:ApproxVariance}, for which some additional control is needed. 
We will put these elements together and complete the proof of the main result of the paper. At last, 
we provide some insights as to how (some of) the above mentioned conditions are weaker then 
those in~\cite[Theorem 2.7]{KomoLandOlla12_FluctuationsMarkov}.

\subsection{Limiting diffusivity}

First, we must control the seminorms in~\eqref{e:VanNorm} of $\gamma f_1$. 

\begin{lemma}
\label{lem:9}
For $\lambda \in (0,1)$, and $\frakf = \coup f_1$, it holds that 
\begin{equ}[e:BoundVanNorm]
\Big(\lambda \lVert (\lambda + \calS + \calS\calG^\lambda)^{-1} \frakf \rVert^2\Big)\vee \vertiii{ \frakf }^2 \lesssim \coup^2\,,
\end{equ}
where the norm $\vertiii{\cdot}$ is defined according to~\eqref{e:VanNorm}. 
\end{lemma}
\begin{proof} 
Let us begin by controlling the first term in the maximum. A simple computation gives that 
\begin{equ}
\lambda \lVert (\lambda + \calS + \calS\calG^\lambda)^{-1} \frakf \rVert^2 \lesssim \coup^2 \lambda\int \frac{\VHat(p)}{|p|^2} \frac{|p|^2}{(\lambda+\half|p|^2)^2} \dd p  \lesssim \coup^2\,.
\end{equ}
Turning to the second term, note that since $\frakf \in \fock_1$, Lemma~\ref{lem:26} and the definition 
of $\calA^R_-$ in~\eqref{eq:AR} imply that 
$\calA^R_- (\lambda + \calS + \calS\calG^\lambda)^{-1} \frakf = 0$. Hence, we only need to control 
\begin{align}
& \vertiii{ \frakf }^2=\lVert (\lambda+\calS)^{-\half} \calA^N_+ (\lambda + \calS + \calS\calG^\lambda)^{-1} \frakf \rVert^2 \label{eq:45} \\
&\lesssim \coup^2 \int  \frac{\dd p _{1:2}\VHat(p_1)\VHat(p_2)}{\lambda + \half|p_1+p_2|^2} \Big\lvert\frac{N(p_1,p_2) \frakf(p_2)}{\lambda+\half|p_2|^2[1+\calg^\lambda(p_2)]} + \frac{N(p_2,p_1)\frakf(p_1)}{\lambda+\half|p_1|^2[1+\calg^\lambda (p_1)]}\Big\rvert^2 \,. \nonumber
\end{align}
We want to replace $N(p_2,p_1)$ with $N(p_1,p_2)$. For this, we add and subtract the corresponding term 
inside the modulus in the previous expression and bound
\begin{equs}[e:Intermediate]
\coup^2 \int  \frac{\VHat(p_1)\VHat(p_2)}{\lambda + \half|p_1+p_2|^2} &\frac{\tilde N(p_1,p_2)|\frakf(p_1)|^2}{(\lambda+\half|p_1|^2[1+\calg^\lambda (p_1)])^2}\dd p _{1:2}\\
&\lesssim \coup^4 \int  \frac{\VHat(p_1)\VHat(p_2)}{\lambda + \half|p_1+p_2|^2} \frac{\tilde N(p_1,p_2)|p_1|^2}{(\lambda+\half|p_1|^2)^2}\dd p _{1:2}
\end{equs}
where $\tilde N(p_1,p_2)=|N(p_2,p_1)-N(p_1,p_2)|$. Notice that $\tilde N(p_1,p_2)=1$ provided that $p_1,p_2$ either 
belong to $R_1\eqdef\tfrac13|p_1|\leq|p_1+p_2|<\tfrac13|p_2|$ or $R_2\eqdef\tfrac13|p_1|\leq|p_1+p_2|<\tfrac13|p_2|$. 
We will split the above integral in these two regions, but since the argument is the same, we will only explicitly treat 
the first. On $R_1$, $|p_1+p_2|\gtrsim |p_1|$ and, by triangle inequality $|p_2|<\tfrac32|p_1|$, hence~\eqref{e:Intermediate} 
is bounded above by 
\begin{equs}
\coup^4 \int \dd p_1 \frac{\VHat(p_1)}{(\lambda + \half|p_1|^2)^2}\int_{|p_2|\lesssim|p_1|}\dd p_2\VHat(p_2) \lesssim \coup^2\,.
\end{equs}
As a consequence, we are left with, writing $N = N(p_1,p_2)$,
\begin{equs}
&\coup^2 \int_N  \frac{\dd p _{1:2}\VHat(p_1)\VHat(p_2)}{\lambda + \half|p_1+p_2|^2} \Big\lvert\frac{ \frakf(p_2)}{\lambda+\half|p_2|^2[1+\calg^\lambda(p_2)]} + \frac{\frakf(p_1)}{\lambda+\half|p_1|^2[1+\calg^\lambda (p_1)]}\Big\rvert^2\\
&=\coup^2 \int_N  \frac{\dd p _{1:2}\VHat(p_1)\VHat(p_2)}{\lambda + \half|p_1+p_2|^2} \Big\lvert\frac{ \frakf(p_2)}{\lambda+\half|p_2|^2[1+\calg^\lambda(p_2)]} - \frac{\frakf(-p_1)}{\lambda+\half|p_1|^2[1+\calg^\lambda (p_1)]}\Big\rvert^2
\end{equs}
where we used that $\frakf$ is odd. The map $p\mapsto \frakf(p)/(\lambda+\half|p|^2[1+\calg^\lambda (p)])$ has a gradient that can be easily seen to be bounded in modulus by 
$\gamma (\lambda+|p|^2)^{-1}$, which, by mean value theorem implies 
\begin{equ}
\Big\lvert\frac{\frakf(p_2)}{\lambda+\half|p_2|^2[1+\calg^\lambda(p_2)]} - \frac{\frakf(-p_1)}{\lambda+\half|-p_1|^2[1+\calg^\lambda(-p_1)]}\Big\rvert \lesssim \frac{\coup|p_1+p_2|}{\lambda+|p_2|^2}\,.
\end{equ}
and we further exploited the fact that for every $p$ in the segment connecting $p_2$ and $-p_1$, 
we have $||p|-|p_2||<\tfrac13|p_2|$. 
In conclusion, we have shown that 
\begin{equ}
\vertiii{ \frakf }^2\lesssim \coup^4 \int_N  \frac{\VHat(p_1)\VHat(p_2)}{(\lambda + \half|p_2|^2)^2} \dd p _{1:2}\leq \coup^4 \int \dd p_2  \frac{\VHat(p_2)}{(\lambda + \half|p_2|^2)^2} \int_{|p_1|\lesssim |p_2|} \VHat(p_1)\dd p _{1}
\end{equ}
and the last integral is $O(\gamma^2)$, from which the statement follows at once. 
\end{proof}

Next, we focus on the diffusivity. 

\begin{lemma}\label{lem:ApproxMean}
There exists $\bar\lambda=\bar\lambda(n)\in(0,1)$ such that for every $\lambda<\bar\lambda$
\begin{equ}
 |\lVert \calS^\half v^{\lambda, n} \rVert^2 - \tfrac12 \sigma^2(\alpha)| \lesssim \coup^2n + n^{-2} \,.
\end{equ}
\end{lemma}
\begin{proof}
At first we want to bound the difference between $\lVert \calS^\half v^{\lambda, n} \rVert^2$ and 
$\langle \frakf, v^{\lambda,n}\rangle$. For this, recall that $\frakf\in\fock_1$ so that 
$\calA^R_-(\lambda +\calS + \calS\calG^\lambda)^{-1}\frakf=0$. Hence, 
testing both sides of~\eqref{eq:27} by $v^{\lambda, n}$ 
and arguing as in~\eqref{eq:24}, we deduce 
\begin{equs}
\big|\lVert \calS^\half v^{\lambda, n} \rVert^2-\langle \frakf, v^{\lambda,n}\rangle\big|&\leq \lambda\|v^{\lambda,n}\|^2+C n \gamma^2\|\calS^\half v^{\lambda,n}\| +|\langle v^{\lambda,n}_n, \calS\calG^\lambda v^{\lambda,n}_n\rangle|\\
&\lesssim \gamma^2 +n \gamma^2 + \|\calS^{\half} v^{\lambda,n}_n\|^2\lesssim n \gamma^2 +\frac{1}{n^2}\label{e:FirstDifference}
\end{equs}
where to bound the first term at the right hand side of the first line we used~\eqref{eq:L2Bound} first and 
both~\eqref{e:BoundVanNorm} and~\eqref{e:hone} after, for the second~\eqref{eq:77} with $k=0$, while 
for the last first the fact that $\calG^\lambda$ is bounded, then~\eqref{eq:77} with $k=1$ and at last 
both~\eqref{e:BoundVanNorm} and~\eqref{e:hone}. 
As a consequence, we only need to focus on the scalar product $\langle \frakf, v^{\lambda,n}\rangle$. 
Notice that by orthogonality of Fock spaces with different indices and the definition 
of $v^{\lambda,n}_1$ in~\eqref{eq:114}, we have 
\begin{equ}[e:SecondDifference]
\langle \frakf, v^{\lambda,n}\rangle=\langle \frakf, v^{\lambda,n}_1\rangle=\lVert(\lambda + \calS + \calS\calG^\lambda)^{-\half}\frakf\rVert^2
\end{equ}
and the last term can be explicitly computed
\begin{equs}
\lVert(\lambda &+ \calS + \calS\calG^\lambda)^{-\half}\frakf\rVert^2= \coup^2 \int \frac{\VHat(p)}{|p|^2}\frac{|-\iota e_1p|^2}{\lambda+\half|p|^2[1 + \calg^\lambda(\half|p|^2)]} \dd p \\
&= \frac{ \coup^2}{2} \int \frac{\VHat(p)}{\lambda+\half|p|^2[1 + \calg^\lambda(\half|p|^2)]} \dd p \\
&= \frac{ \coup^2}{2} \int_{|p| \le 1} \frac{\VHat(p)}{\lambda+\half|p|^2(1 + \calg^\lambda(\half|p|^2))} \dd p + O(\coup^2) \\
&= \frac{ \coup^2}{2} \int_{|p| \le 1} \frac{1}{\lambda+\half|p|^2(1 + \calg^\lambda(\half|p|^2))} \dd p  + O(\coup^2) \\
&= \pi \coup^2 \int_0^\half \frac{1}{\lambda+\rho(1 + \calg^\lambda(\rho))} \dd p  + O(\coup^2) \\
&= \pi \coup^2 \int_0^\half \frac{1}{(\lambda+\rho)(\lambda+\rho+1)(1 + \calg^\lambda(\rho))} \dd p  + O(\coup^2) =\half \sigma^2(\alpha) +O(\coup^2) 
\end{equs}
where in the various steps above we used that uniformly on $|p| \le 1$ by the mean value theorem, 
$|\VHat(p) - 1| \lesssim |p|$, passed to polar coordinates, and argued as in the proof of Lemma \ref{lem:ReplEst}. 
The last step is a consequence of~\eqref{eq:87} together with the fact that $\sigma^2(\alpha)=g(\ell^\lambda(0))$. 
The statement then follows by the above,~\eqref{e:SecondDifference} and~\eqref{e:FirstDifference}. 
\end{proof}

\subsection{Variance bounds for the quadratic variation}
\label{sec:vari-bounds-quadr}

In this section, we derive the bound which is necessary to verify \eqref{eq:ApproxVariance}. 

\begin{proposition}
\label{prp:4}
For $n \in \bbN$ and $\lambda\in(0,1)$, let $q^{\lambda,n} \in \fock_{\le2n}$ and $\sigma_{\lambda,n}^2 > 0$ 
be defined according to~\eqref{e:gsigma}. 
Then, there exists $\bar\lambda=\bar\lambda(n)\in(0,1)$ and a constant $C=C(n)>0$ 
such that for all $\lambda<\bar\lambda$, 
we have 
\begin{equ}[e:VarBound]
\lambda \lVert (\lambda+\calS)^{-\half}\big(q^{\lambda,n} - \sigma_{\lambda,n}^2\big) \rVert^2 \leq C\coup^2\,.
\end{equ}
\end{proposition}
\begin{proof}
We write $q^{\lambda,n}$ in terms of its chaos decomposition $q^{\lambda,n} = \sum_{m=0}^{2n}q^{\lambda,n}_m$ 
and note that $q^{\lambda,n}_0 = \sigma^2_{\lambda,n}$. 
Therefore it is enough to show that each component $q^{\lambda,n}_m$, 
with $m\neq0$, is such that $\lambda \lVert (\lambda+\calS)^{-\half}q^{\lambda,n}_m\rVert^2$ is bounded by 
the right hand side of~\eqref{e:VarBound}. To that end, let $m\in\{1,\dots,2n\}$ and write
\begin{equs} 
q^{\lambda,n}_m = \sum_{l=1}^2 \sum_{j,k = 1}^n \Pi_m \left((\nabla_l v^{\lambda,n}_j)(\nabla_l v^{\lambda,n}_k)\right)=:\sum_{l=1}^2 \sum_{j,k = 1}^n h_{m,\ell,j,k}\,. 
\end{equs}
By triangle inequality, the statement follows once we prove that for each given $m\in\{1,\dots,2n\}$, 
$l\in\{1,2\}$, $j,k\in\{1,\dots,n\}$ we have 
\begin{equ}[e:Bound]
\lambda \lVert (\lambda+\calS)^{-\half} h_{m,\ell,j,k} \rVert^2 \leq C\coup^2
\end{equ}
for some constant $C>0$ which might depend on $m,l,j,k$. Therefore, let $\ell,m,j,k$ be fixed once and for all, assume, 
without loss of generality, $j\geq k$ 
and drop the corresponding subscripts from $h_{m,\ell,j,k}$, i.e. set $h= h_{m,\ell,j,k}$. 
By \cite[Theorem 3.15]{Janson97_GaussianHilbert} $h$ is non-zero only if 
$m = j + k - 2r$ for some $r \in \{1,...,j\}$, 
in which case it is given by the symmetrisation of the function $\hTil$ defined according to
\begin{equation}
\label{eq:88}
\hTil(p_{1:m})\eqdef C_{j,k,r} \int  \nabla_l v_j(q_{1:r},p_{1:j-r}) \nabla_l v_k(q_{1:r},p_{j-r+1:m}) \mu_r(\dd q_{1:r})
\end{equation}
where $v = v^{\lambda, n}$, and $C_{j,k,r}$ is an explicit combinatorial constant. Moreover, since $j\geq k$ and $m\neq0$, it holds that $r<j$. 
Then, 
\begin{equs}
&\lambda \lVert (\lambda+\calS)^{-\half} h\rVert^2 = \lambda \int \frac{h(p_{1:m})^2}{\lambda+|p_{[1:m]}|^2} \mu_m(\dd p) \lesssim \lambda \int \frac{\hTil(p_{1:m})^2}{\lambda+|p_{[1:m]}|^2} \mu_m(\dd p) \\
&\lesssim \lambda \int \frac{1}{\lambda+|p_{[1:m]}|^2} \nabla_l v_j(q_{1:r},p_{1:j-r}) \, \nabla_l v_k(q_{1:r},p_{j-r+1:m}) \label{e:Uffa}\\
& \qquad \qquad \times\nabla_l v_j(q^\prime_{1:r},p_{1:j-r})\,\nabla_l v_k(q^\prime_{1:r},p_{j-r+1:m})\big)  \mu_m(\dd p_{1:m}) \mu_r(\dd q_{1:r}) \mu_r(\dd q^\prime_{1:r})\\
&\lesssim \lambda\int  \frac{\big|\calS^\half v_j(q_{1:r},p_{1:j-r})\big|^2 \big|\calS^\half v_k(q^\prime_{1:r},p_{j-r+1:m})\big|^2}{\lambda+|p_{[1:m]}|^2} \mu_m(\dd p_{1:m}) \mu_r(\dd q_{1:r}) \mu_r(\dd q^\prime_{1:r})
\end{equs}
where we omit the dependence on $m,l,j,k$ and, in the last step, we used $2ab\leq a^2+b^2$, the fact that, 
upon changing variables the two summands are the same and that, by the definition of $\nabla_l$ and $\calS$ 
in Lemma~\ref{lem:26} $|\nabla_l v_j|\leq |\calS^{1/2}v_j|$. 

Modulo a constant depending only on $m, j$ and $r$, we can write the product of the measures 
$\mu_m(\dd p_{1:m}) \mu_r(\dd q_{1:r}) \mu_r(\dd q^\prime_{1:r})$ as $\mu_j(\dd q_{1:r}\dd p_{1:j-r}) \mu_k(\dd q^\prime_{1:r}\dd p_{j-r+1:m})$ and consider first the integral in $\mu_j$. This in turn is controlled 
via Lemma~\eqref{lem:23}, so that~\eqref{e:Uffa} is bounded above by 
\begin{equ}[e:Uffa2]
\coup^2\int \big|\calS^\half v_j(q^\prime_{1:r},p_{j-r+1:m})\big|^2 \mu_{k}(\dd q^\prime_{1:r}\dd p_{j-r+1:m})=\gamma^2\|\calS^\half v_j\|^2
\end{equ}
and, thanks to~\ref{eq:77} with $k=0$ and~\eqref{e:hone}, the norm at right hand side is uniformly bounded, so that 
the proof of the statement is concluded.  
\end{proof}

\begin{lemma}
\label{lem:23}
In the setting of Proposition~\ref{prp:4}, for any $n \in \bbN$, $j \in \{1,...,n\}$, $r \in \{0,...,j-1\}$, for $\lambda$ sufficiently 
small, it holds that
\begin{equation}
\label{eq:92}
\lambda \int   \frac{\big|\calS^\half v^{\lambda,n}_j(p_{1:j})\big|^2}{\lambda+|p_{[1:j-r]}+p^\prime|^2} \mu_j(\dd p _{1:j}) \lesssim \coup^2\,
\end{equation}
uniformly over $p'\in\R^2$. 
\end{lemma}
\begin{proof}
To lighten the notation, we set $v=v^{\lambda,n}$ throughout the proof. We prove the result by induction on $j$. For $j=1$, $r=0$. Therefore, by~\eqref{eq:114}, we get 
\begin{equ}
\lambda \int   \frac{\big|\calS^\half v_1(p)\big|^2}{\lambda+|p+p'|^2} \mu_1(\dd p )\lesssim  \lambda \int\frac{\VHat(p)}{|p|^2} \frac{\coup^2|p|^4}{(\lambda+|p+p^\prime|^2)(\lambda+|p|^2)^2} \dd p\lesssim \gamma^2
\end{equ}
where the last bound follows by Lemma \ref{lem:22}. Assume now~\eqref{eq:92} holds for $j-1$, 
and, for $\ell\in\{1,\dots,j\}$, define $v_j[\ell]$ as the right hand side of~\eqref{eq:114} but with $\calA^B_+$ 
replaced by $\calA^B_+[\ell]$, the latter being given in~\eqref{eq:ARi}. 
Clearly, it suffices to show~\eqref{eq:92} with $v_j[\ell]$ in place of $v_j$. 
Then,
\begin{equs}
&\lambda \int   \frac{\big|\calS^\half v_j[\ell](p_{1:j})\big|^2}{\lambda+|p_{[1:j-r]}+p^\prime|^2} \mu_j(\dd p _{1:j})\\
&\lesssim \coup^2\int \mu_{j-1}(\dd p _{1:j\setminus \ell})\big|\calS^\half v_{j-1}(p_{1:j\setminus\ell})\big|^2 \Big(\lambda\int \frac{\VHat(p_\ell) \dd p _\ell}{(\lambda+|p_{[1:j-r]}+p^\prime|^2)(\lambda+|p_{[1:j]}|^2)} \Big) 
\end{equs} 
which holds since $\cg^\lambda$ is non-negative. For $\ell \in \{1,...,j-r\}$, we can upper bound the quantity in parenthesis 
via Lemma~\ref{lem:22} and then argue as in~\eqref{e:Uffa2}. For $\ell\in\{j-r+1,j\}$, we instead bound the 
integral over $p_\ell$ as in~\eqref{e:UBWeak}, so that we are left with 
\begin{equ}
\lambda\int \frac{\big|\calS^\half v_{j-1}(p_{1:j\setminus\ell})\big|^2}{\lambda+|p_{[1:j-r]}+p^\prime|^2} \mu_{j-1}(\dd p _{1:j\setminus \ell})
\end{equ}
and the bound follows by the induction hypothesis. 
\end{proof}

\subsection{Proof of Theorem~\ref{thm:1}, the invariance principle}

We can now complete the proof of Theorem~\ref{thm:1}, for which we only need to collect the bounds 
in the previous subsections. 

\begin{proof}[of Theorem \ref{thm:1}]
The family $(v^{\lambda,n}\colon \lambda\in(0,1),\,n\in\N)$ in Definition~\ref{def:3} with input function 
$\frakf=\coup f_1$ satisfies the conditions of Theorem~\ref{thm:3}, which then imply the result. 
Indeed,~\cref{eq:L2,eq:ApproxH1} follow respectively by~\eqref{eq:L2Bound} and~\eqref{eq:ApproxH1Bound} 
in Proposition~\ref{p:MainEstimates} (note that $\frakf=\gamma f_1$ trivially satisfies~\eqref{a:RTI}) together 
with~\eqref{e:BoundVanNorm} and~\eqref{e:hone}. 
Condition~\eqref{eq:ApproxMean} is a consequence of Lemma~\ref{lem:ApproxMean} while~\eqref{eq:ApproxVariance} 
can be deduced by Proposition~\ref{prp:4}. 
\end{proof}

\subsection{A comparison to the classical Fluctuation theory for Markov Processes}
\label{sec:no-weak-convergent}
This subsection is meant to highlight the novelty of our approach with respect to the classical theory  
presented in~\cite{KomoLandOlla12_FluctuationsMarkov} and complement the discussion in Section~\ref{sec:BCC}. 
In particular, our goal is to show that the conditions~\cref{eq:L2,eq:ApproxH1} imposed on the family 
$(v^{\lambda,n}\colon \lambda\in(0,1),\,n\in\N)$ do not imply the assumptions 
of~\cite[Theorem 2.7]{KomoLandOlla12_FluctuationsMarkov}. 
The main assumption required therein is~\cite[eq. (2.23)]{KomoLandOlla12_FluctuationsMarkov}, 
which in the language of the present paper, reads 
\begin{equ}[e:AssKLO]
\lim_{\lambda\to 0} \lambda\|u^\lambda\|^2=0\qquad\text{and}\qquad \lim_{\lambda\to 0} \|\calS^{\half}(u^\lambda-u)\|=0
\end{equ}
for some some $u\in\fH\eqdef\{\psi\in\fock\colon \|\psi\|_\fH\eqdef\|\calS^{1/2}\psi\|<\infty\}$, 
where $u^\lambda$ is the solution of the resolvent equation $(\lambda-\calL)u^\lambda=\gamma f_1$, 
with $f_1$ given by~\eqref{e:fdef}. 
In the next proposition, we show that only the former among the two limits holds in our setting, while the 
latter does not. 

\begin{proposition}\label{p:USvsKLO}
For $\lambda\in(0,1)$, let $u^\lambda$ be the solution of the resolvent equation $(\lambda-\calL)u^\lambda=\gamma f_1$ 
with $f_1$ defined as in~\eqref{e:fdef}. Then, the first of the two limits in~\eqref{e:AssKLO} holds 
but there exists no $u\in\fH$ for which the second does. 
\end{proposition}

The proof of the previous is based on the following lemma which shows that, for $n$ large enough and 
$\lambda$ sufficiently small, the solution of the replacement 
equation provides a good approximation.

\begin{lemma}
For $\lambda\in(0,1)$ and $n\in\N$, let $v^{\lambda,n}$ be the solution 
of the replacement equation in Definition~\ref{def:3} both with input $\frakf\eqdef\gamma f_1$. 
Then, there exists a constant $C>0$ and $\bar\lambda=\bar\lambda(n)\in(0,1)$ such that for all $\lambda<\bar\lambda$, 
we have  
\begin{equ}[e:ApproxRERE]
\|(\lambda+\calS)^{\half}(u^\lambda-v^{\lambda,n})\|^2\leq C(n^{-\half} + C_N\gamma)
\end{equ}
where $C_N>0$ is a constant for which~\eqref{eq:17} holds. 
\end{lemma}
\begin{proof}
Let $u^{\lambda,n}$ be the solution of the truncated generator equation~\eqref{eq:5} with input $\frakf\eqdef\gamma f_1$. 
At first, we show that we can approximate $u^\lambda$ with $u^{\lambda,n}$. 
By definition of $u^\lambda$ and $u^{\lambda,n}$, we have 
$(\lambda-\calL)(u^\lambda-u^{\lambda,n})=\calA_+ u^{\lambda,n}_n$,
so that testing both sides by $u^\lambda-u^{\lambda,n}$ we obtain 
\begin{equs}
\|(\lambda+\calS)^\half(u^\lambda-u^{\lambda,n})\|^2 &=\langle u^\lambda-u^{\lambda,n}, \calA_+ u^{\lambda,n}_n\rangle\\
&\leq \|(\lambda+\calS)^\half(u^\lambda-u^{\lambda,n})\|\|(\lambda+\calS)^{-\half}\calA_+ u^{\lambda,n}_n\|\,.
\end{equs}
Now, by~\eqref{eq:AR7} with $\calT\equiv 0$, we get 
\begin{equs}
\|(\lambda+\calS)^{-\half}\calA_+ u^{\lambda,n}_n\|\lesssim \sqrt{n} \|\calS^\half u^{\lambda,n}_n\|\lesssim n^{-\half}\|\calN(\lambda+\calS)^\half u^{\lambda,n}\|\lesssim n^{-\half}
\end{equs} 
the last step being a consequence of~\eqref{eq:41} with $k=1$, and~\eqref{e:hone}. 
Therefore, we obtain 
\begin{equs}
\|(\lambda+\calS)^\half(u^\lambda-u^{\lambda,n})\|\lesssim \frac{1}{\sqrt{n}}\,.
\end{equs}
Further using~\eqref{e:BoundDifference},~\eqref{e:hone} and~\eqref{e:BoundVanNorm}, we conclude that 
\begin{equs}
\|(\lambda+\calS)^{\half}(u^\lambda-v^{\lambda,n})\|^2&\lesssim \|(\lambda+\calS)^{\half}(u^\lambda-u^{\lambda,n})\|^2 +\|(\lambda+\calS)^{\half}(u^{\lambda,n}-v^{\lambda,n})\|^2\\
&\lesssim n^{-\half} + (n^{-2}+C_N\coup^2)+\vertiii{ \frakf }\lesssim n^{-\half} + C_N\gamma
\end{equs}
which concludes the proof. 
\end{proof}

\begin{proof}[of Proposition~\ref{p:USvsKLO}]
Thanks to~\eqref{e:ApproxRERE}, we immediately obtain
\begin{equ}[e:FirstLim]
\limsup_{\lambda\to 0}\lambda\|u^\lambda\|^2\lesssim \limsup_{\lambda\to 0}\lambda\|u^\lambda-v^{\lambda,n}\|^2+\lambda\|v^{\lambda,n}\|^2\lesssim n^{-\half}+\limsup_{\lambda\to 0}\lambda\|v^{\lambda,n}\|^2\,.
\end{equ}
Since the left hand side is independent of $n$, we can pass to the limit as $n\to\infty$, so that 
the first limit in~\eqref{e:AssKLO} follows by~\eqref{eq:L2}. 
To show that the other limit does not hold, it suffices to prove that, as $\lambda$ goes to $0$, 
$\|u^\lambda\|_{\fH}^2$ converges to $\half\sigma^2(\alpha)>0$ while $u^\lambda$ converges to $0$ weakly in $\fH$. 
For the first, by~\eqref{e:ApproxRERE}, we have 
\begin{equs}[e:SecondLimNorm]
\limsup_{\lambda\to 0}&|\|\calS^\half u^\lambda\|^2-\tfrac12\sigma^2(\alpha)|\\
&\lesssim \limsup_{\lambda\to 0}\|\calS^\half (u^\lambda- v^{\lambda,n})\|^2+|\|\calS^\half v^{\lambda,n}\|^2-\tfrac12\sigma^2(\alpha)|\\
&\lesssim n^{-\half}+\limsup_{\lambda\to 0}|\|\calS^\half v^{\lambda,n}\|^2-\tfrac12\sigma^2(\alpha)|
\end{equs}
and, arguing as for~\eqref{e:FirstLim} and using~\eqref{eq:ApproxMean}, we see that the left hand side 
exists as a limit and such limit is $0$. 
Concerning the weak convergence, by density, it suffices to show that 
for every given $m\in\N$ and $\psi\in\fock_m\cap\fH$, the scalar product in $\fH$ of $u^\lambda$ and $\psi$ vanishes. 
Fixing $m$ and $\psi$ as above, let $\Psi\eqdef\calS^\half\psi$. For $\delta>0$, let $\kappa>0$ be such that $\|\Psi \1_{\calS^\half\leq \kappa}\|\leq \delta$. 
We write
\begin{equ}
\langle u^\lambda, \psi\rangle_\fH=\langle \calS^\half u^\lambda, \Psi\rangle=\langle \calS^\half u^\lambda, \1_{\calS^\half\leq \kappa}\Psi\rangle+\langle \calS^\half u^\lambda, \1_{\calS^\half> \kappa}\Psi\rangle\,.
\end{equ}
For the first summand, we have 
\begin{equ}[e:LastLim1]
\langle \calS^\half u^\lambda, \1_{\calS^\half\leq \kappa}\Psi\rangle\leq \|\calS^\half u^\lambda\|\|\1_{\calS^\half\leq \kappa}\Psi\|\lesssim \delta \sigma(\alpha)
\end{equ}
where we used~\eqref{e:SecondLimNorm} to control the norm of $u^\lambda$. 
For the second instead, we notice first that, by orthogonality of Fock spaces with different indices, 
we can replace $u^\lambda$ with $u^\lambda_m$ and then we bound 
\begin{equs}
\limsup_{\lambda\to 0}|\langle \calS^\half u^\lambda_m, \1_{\calS^\half> \kappa}\Psi\rangle|&\leq \|\Psi\|\limsup_{\lambda\to 0}\| \1_{\calS^\half> \kappa}\calS^\half u^\lambda_m\|\\
&\leq \|\Psi\| \limsup_{\lambda\to 0} \big( \|\calS^\half (u^\lambda-v^{\lambda,n})\|+\| \1_{\calS^\half> \kappa}\calS^\half v^{\lambda,n}_m\|\big) \\
&\lesssim \|\Psi\| ( n^{-\half}+ \limsup_{\lambda\to 0}\| \1_{\calS^\half> \kappa}\calS^\half v^{\lambda,n}_m\|)
\end{equs}
where we used~\eqref{e:ApproxRERE} once again. At this point, by~\eqref{eq:114} and the fact that 
$\cG^\lambda$ is non-negative, we get 
\begin{equs}
\|& \1_{\calS^\half> \kappa}\calS^\half v^{\lambda,n}_m \|^2=\|\1_{\calS^\half> \kappa}\calS^\half (\lambda+\calS+\calS\calG^\lambda)^{-1}\calA^B_+v^{\lambda,n}_{m-1} \|^2\\
&\lesssim n \coup^2 \int\mu_{m-1}(\dd p_{2:m})|p_{[2:m]}|^2 |u^\lambda_{m-1}(p_{2:m})|^2 \int \dd p_1 \frac{\hat V(p_1)\1_{|p_{[1:m]}|>\sqrt{2}\kappa}}{\lambda +|p_{[1:m]}|^2}\\
&\lesssim n\frac{\coup^2}{\kappa^2}\|\calS^\half v^{\lambda,n}_{m-1}\|^2\lesssim n\frac{\coup^2}{\kappa^2}
\end{equs}
where, in the last step we used~\eqref{eq:77} with $k=0$, and~\eqref{e:hone}. Since this last term converges to 
$0$ as $\lambda$ goes to $0$, we conclude that 
\begin{equ}
\limsup_{\lambda\to0}|\langle u^\lambda, \psi\rangle_\fH|\lesssim \delta \sigma(\alpha)+n^{-\half} \|\psi\|_\fH\,.
\end{equ}
As the left hand side is independent of both $\delta$ and $n$, the proof is concluded. 
\end{proof}

\section{Convergence of the environment}
\label{sec:Env}

The goal of this section is to prove Theorem \ref{thm:2}, i.e. that the rescaled environment process $\eta^\eps$
converges to the solution of the stochastic linear transport equation (SLTE)~\eqref{eq:131}. 
As mentioned in the introduction, 
the structure of the argument is similar to that of the proof of Theorem~\ref{thm:1}, namely, 
we will prove tightness via the It\^o trick (see Section~\ref{sec:h-valued-ito-1}) 
and then identify the limit by using a martingale characterisation of the solution of SLTE in Definition~\ref{def:2}, 
for which in turn we will need the results in Section~\ref{sec:appr-solv-gener}. 
That said, since the solution of the SLTE is a distribution, we begin by introducing some preliminary 
tools and definitions better tailored to rigorously determine the convergence in this setting. 
\medskip

At first we need a Hilbert space $H$ whose embedding into the classical Sobolev space $H^k(\R^2,\R^2)$, 
$k=0,1,2$, is Hilbert-Schmidt. For this, let $H=H_w^{50}(\bbR^2,\bbR^2)$ be the weighted Sobolev space 
of $\R^2$-valued functions $g$ on $\R^2$ such that 
\begin{equ}[e:H]
\|g\|^2_H\eqdef \int_{\bbR^2} (1+|x|^2)^2 \big((1-\Delta)^{25}g(x)\cdot (1-\Delta)^{25}g(x)\big) \dd x<\infty
\end{equ}
endowed with the scalar product induced by the above norm, and let $H^\ast$ be its dual. 
Let $(g^i)_{i\in\N}$ be an orthonormal basis of $H$, which, without loss of generality we can take  
be given by $g^i=(1-\Delta)^{-25}(1+|x|^2)^{-1}f^i$ with $(f^i)_{i\in\N}$ Schwarz functions forming 
an orthonormal basis of $L^2(\R^2,\R^2)$.  
Clearly, we have 
\begin{equation}
\label{eq:8}
\sum_{i=1}^\infty \lVert g^i \rVert_{H^k(\bbR^2,\bbR^2)}^2 < \infty \qquad k \in \{0,1,2\}\,.
\end{equation}
We also have the natural embedding $H \hookrightarrow \fock_1$ given by  
$H\ni g \mapsto\frakg \in \fock_1$ for $\frakg$ as in~\eqref{eq:42} with $n=1$. 
Let $\bfL^2(\pi)$ be the space of $H^*$-valued random variables $\bfh$ on $\Omega$  such that 
\begin{equation}
\label{eq:35}
\lVert \bfh \rVert_{\bfL^2(\pi)}^2 = \bbE_\pi[\lVert \bfh(\omega) \rVert_{H^*}^2] =\sum_{i=1}^\infty\bbE_\pi\Big[| \bfh(\omega)[g^i] |^2\Big] 
<\infty\,,
\end{equation}
and, for $n \in \bbN$, let $\bm{\calH}_n$ be the set of $\bfh\in \bfL^2(\pi)$ for which $\bfh[g] \in \calH_n$ for all $g \in H$, and let $\bm{\calH}_{\le n}$ be $\oplus_{j=1}^n \bm{\calH}_j$. With a slight abuse of notation, we will denote with the 
same symbol operators on $L^2(\pi)$ and $\bfL^2(\pi)$, i.e. if $T\colon L^2(\pi) \rightarrow L^2(\pi)$ 
then its action on $\bfL^2(\pi)$ is defined according to $T\bfh(\omega)[g] = T(\bfh[g])(\omega)$ 
for every $g \in H$ and $\bfh\in\bfL^2(\pi)$ such that $\bfh[g]\in\mathrm{dom}(T)$. 
\medskip

In what follows, we want to study the diffusively rescaled environment process given by  
$\eta^\eps_t(x)\eqdef \eps^{-1}\eta_{t/\eps^2}(x/\eps)$ for $\eps>0$. 
To prove its convergence, we need to interpret $\eta^\eps$ as a generalised function, which we will 
denote by $\bm{\eta}^\eps_t=\bfg^\eps(\eta_{t/\eps^2})$ where $\bfg^\eps$ is the element of $\bm{\calH}_1$ defined according to 
\begin{equ}[e:scaling]
\bfg^\eps(\omega)[g]\eqdef \int \omega(x)\cdot g^\eps(x)\dd x\,,\qquad \text{for $g\in H$ and $g^\eps(x)\eqdef \eps g(\eps x)$}. 
\end{equ}
As a last remark, notice that $\bfg^\eps\in\bfL^2(\pi)$ is bounded {\it uniformly} over $\eps$. 
Indeed, slightly more is true since, for $\delta\in[0,2]$, we have 
\begin{equ}[e:gepsB]
\lVert \eps^{-\delta}\calS^{\frac{\delta}{2}}\bfg^\eps \rVert_{\bfL^2(\pi)}^2 
=\sum_{i=1}^\infty\bbE_\pi\big[\big|\eps^{-\delta}\calS^{\frac{\delta}{2}} (\omega[g^{i,\eps}] )\big|^2\big] \lesssim \sum_{i=1}^\infty\|g^i\|^2_{H^{\delta}(\R^2,\R^2)}
\end{equ}
and the right hand side is bounded by~\eqref{eq:8}. 
The estimate in the third step is a consequence of~\eqref{e:ScalarProd}. Indeed, 
setting $\frakg^{i}\in\fock_1$ to be given as in~\eqref{eq:42} with $g^i$ in place of $f$ and $n=1$, and 
$\frakg^{i,\eps}(p)\eqdef\frakg(\eps^{-1}p)$, we have 
\begin{equs}
\bbE_\pi\big[\big|\eps^{-\delta}\calS^{\frac{\delta}{2}} (\omega[g^{i,\eps}]) \big|^2\big]&= \lVert\eps^{-\delta}\calS^{\frac{\delta}{2}}\frakg^{i,\eps} \rVert^2\\
&=\frac{1}{2^\delta}\int \frac{\hat V(p)}{|p|^2}|p/\eps|^{2\delta} |e_1p \hat g^i_1(\eps^{-1}p)+e_2p \hat g^i_2(\eps^{-1}p)|^2\frac{\dd p}{\eps^{2}}\\
&\lesssim  \int |p|^{2\delta}\hat g^i(p)\cdot \hat g^i(p)\dd p = \|g^i\|^2_{H^\delta(\R^2,\R^2)}\,.\label{e:gepsB2}
\end{equs}

\subsection{The $H^*$-valued It\^{o} trick and tightness}
\label{sec:h-valued-ito-1}

In this section, we will prove convergence of the initial condition and tightness of the environment process 
$\bm{\eta}^\eps$ in the space $C_TH^\ast\eqdef C([0,T], H^\ast)$ for any $T>0$. 
To do so, let us first define the limit law of the former which, as argued in the introduction, is given by 
that of the gradient of a two-dimensional GFF.  

\begin{definition}
\label{def:5}
Let $\bm{\piBar}$ be the probability measure on $H^*$ under which $\{\bfh[g] : g \in H\}$ 
is a Gaussian process with covariance 
\begin{equ}
\int \bfh[g_1] \bfh[g_2] \bm{\piBar}(\dd \bfh) = \int_{\bbR^2} \int_{\bbR^2} \ddiv(g_1)(x) \ddiv(g_2)(y) G(x-y) \dd x \dd y
\end{equ}
where $G$ is the Green's function given in \eqref{e:CovOmega}. 
\end{definition}

We are now ready to state the main result of this section. 

\begin{proposition}
\label{lem:13}
Let $(\eta_t)_{t\geq 0}$ be the environment seen by the particle process defined in~\eqref{e:envNew}, 
$\eta^\eps$ be its diffusively rescaled version for $\eps>0$, and $\bm{\eta}^\eps$ 
the associated generalised function. 
Then, under the annealed measure $\bfP$, 
the law of $\bm{\eta}^\eps_0$ converges in $H^\ast$ as $\eps\to 0$ to $\bm{\piBar}$ in Definition~\ref{def:5} 
and, for any $T>0$, the sequence $\{(\bm{\eta}^\varepsilon_t)_{t\in[0,T]} \colon \eps\in(0,1)\}$ is tight in $C_T H^\ast$. 
\end{proposition}

As in the case of the SRBP, tightness is an immediate consequence of the It\^o trick which we now 
state in its $H^\ast$-valued version.  

\begin{lemma}[$H^*$-valued It\^{o} trick]
\label{lem:2}
For $n \in \bbN$, $\bfh \in \bm{\calH}_{\le n}$, $T > 0$, $p \ge 1,\lambda>0$, it holds that 
\begin{equation*}
\bfE\Big[ \sup_{t \in [0,T]} \Big\lVert \int_0^t \bfh(\eta_s) \dd s \Big\rVert_{H^*}^p \Big]^{\frac1p} \lesssim_{n, p} (T^{\half}+\lambda^{\half} T)\lVert (\lambda+\calS)^{-\half}\bfh\rVert_{\bfL^2(\pi)}
\end{equation*}
Moreover, in the case $p = 2$, the estimate is uniform in $n \in \bbN$. 
\end{lemma}

Central to the proof of the previous statements as well as that of Theorem \ref{thm:2}, are 
$H^\ast$-valued Dynkin martingales, analogous to those in~\eqref{e:Dynkin}. 
That is, for nice enough $\bfu \in \bfL^2(\pi)$, by applying It\^o's formula, it holds that 
\begin{equation}
\label{eq:119}
\bfM_t(\bfu)= \bfu(\eta_t) - \bfu(\eta_0) - \int_0^t \calL \bfu(\eta_s) \dd s
\end{equation}
is the $H^*$-valued martingale given by
\begin{equation}
\label{eq:4}
\bfM_t(\bfu) = \sum_{i=1}^2 \int_0^t \nabla_i \bfu(\eta_s) \dd B^i_s\,, \quad \langle\!\langle \bfM(\bfu)\rangle\!\rangle_t = \sum_{i=1}^2 \int_0^t (\nabla_i \bfu(\eta_s))(\nabla_i \bfu(\eta_s))^* \dd s\,.
\end{equation}
For a thorough introduction on $H^\ast$-valued martingales, we refer the reader to~\cite[Chapter 3]{DaPratoZabczyk14_StochasticEquations}.
We will now first briefly sketch the proof of the It\^o trick and then turn to Proposition~\ref{lem:13}. 

\begin{proof}[of Lemma~\ref{lem:2}]
For $T>0$ fixed, let $(\etaHat_t)_{t \in [0,T]}$ be the reversed process defined by $\etaHat_t \eqdef \eta_{T-t}$. 
It is a standard fact that $\hat\eta$ is again a Markov process with stationary measure $\pi$ and generator 
$\calLHat = (\calL)^*=\calS-\calA$.  In particular, for nice enough $\bfu \in \bfL^2(\pi)$, the 
process $(\bfMHat_t(\bfu))_{t\in[0,T]}$ defined according to the right hand side of~\eqref{eq:119} but with $\hat\eta$ and 
$\calLHat$ in place of $\eta$ and $\calL$ respectively, is a martingale with quadratic variation as in~\eqref{eq:4}. 

For $\lambda\in(0,1)$ and $\bfh$ as in the statement, set 
$\bfw^\lambda = (\lambda + \calS)^{-1} \bfh \in \bm{\calH}_{\le n}$. By adding up
$\bfw^\lambda(\eta_t)+\bfw^\lambda(\hat\eta_T)-\bfw^\lambda(\hat\eta_{T-t})$ in the corresponding 
formulas~\eqref{eq:119}, we obtain 
\begin{equ}
\int_0^t \bfh(\eta_s)\dd s = \lambda \int_0^t \bfw^\lambda(\eta_s)\dd s + \half\Big(\bfM_t(\bfw^\lambda)+\bfMHat_T(\bfw^\lambda) - \bfMHat_{T-t}(\bfw^\lambda)\Big)\,.
\end{equ}
For the martingale part, the three terms can be estimated similarly, so we focus on the first. 
The infinite dimensional version of 
Burkholder-Davis-Gundy inequality~\cite[Theorem 1.1]{MarinelRockner16_MaximalInequalities} 
implies
\begin{equs}
\bfE&\Big[  \sup_{t \in [0,T]} \lVert \bfM_t(\bfw^\lambda) \rVert_{H^*}^p \Big] \lesssim_p\bfE\Big[ \text{Tr} \langle\!\langle\bfM \rangle\!\rangle_T^{p/2} \Big]\\
&= \bfE\Big[ \Big( \sum_{i=1}^2 \int_0^T \lVert \nabla_i \bfw^\lambda(\eta_s) \rVert_{H^*}^2 \dd s \Big)^{p/2} \Big] 
\lesssim_p T^{p/2} \sum_{i=1}^2 \bbE_\pi\big[ \lVert \nabla_i \bfw^\lambda(\omega) \rVert_{H^*}^p\big] \\
&\lesssim_n T^{p/2} \sum_{i=1}^2 \lVert \nabla_i \bfw^\lambda(\omega) \rVert_{\bfL^2(\pi)}^p \lesssim T^{p/2} \lVert (\lambda+\calS)^{-\half}\bfh \rVert_{\bfL^2(\pi)}^p
\end{equs}
where we used Jensen's inequality, stationarity, Gaussian hypercontractivity \cite[Theorem 5.10]{Janson97_GaussianHilbert}, \eqref{eq:35} and the definition of $\bfw^\lambda$. 
Concerning the finite variation term, we follow similar steps and get
\begin{align*}
\bfE&\Big[  \sup_{t \in [0,T]} \lVert  \lambda \int_0^t \bfw^\lambda(\eta_s)ds \rVert_{H^*}^p \Big] \le \lambda^pT^p \bfE\Big[ \Big( \frac{1}{T} \int_0^T \lVert \bfw^\lambda(\eta_s)ds \rVert_{H^*}\Big)^p \Big] \\
&= \lambda^pT^p \bbE_\pi \big[ \lVert \bfw^\lambda(\omega) \rVert_{H^*}^p \big] \lesssim_n \lambda^pT^p \lVert \bfw^\lambda \rVert_{\bfL^2(\pi)}^p \lesssim \lambda^{p/2}T^p \lVert (\lambda+\calS)^{-\half}\bfh \rVert_{\bfL^2(\pi)}^p\,.
\end{align*}
As for the statement concerning the case $p = 2$, we don't need to use hypercontractivity and 
therefore there is no dependence in $n$.
\end{proof}

\begin{proof}[of Proposition~\ref{lem:13}]
Convergence (and therefore tightness) of the initial distribution is standard and therefore omitted. 
Concerning the process, we use Kolmogorov's criterion \cite[Theorem 23.7]{Kallenberg21_FoundationsModern}, 
according to which, and thanks to stationarity, we only need to control the $p$-th moment of 
$\lVert \bm{\eta}^\varepsilon_t-\bm{\eta}^\varepsilon_0\rVert_{H^*}$ for some $p>2$. 
Recalling the definition of $\bfg^\eps$ in~\eqref{e:scaling} and using~\eqref{eq:119} and~\eqref{eq:4}, we have 
\begin{equation}
\label{eq:122}
\bfg^\eps(\eta_{t/\eps^2}) -  \bfg^\eps(\eta_0) = \int_0^{t/\varepsilon^2} \calL \bfg^\varepsilon(\eta_s)\dd s + \sum_{i=1}^2 \int_0^{t/\varepsilon^2} \nabla_i \bfg^\varepsilon(\eta_t) \dd B^i_t\,.
\end{equation}
and we will control each term at the right hand side separately. 
For the first, we use the It\^o trick in Lemma~\ref{lem:2}, which is applicable since $\calL \bfg^\eps\in\bm{\calH}_{\le 2}$, 
and take $\lambda=\eps^2$, so that 
\begin{equs}
\bfE \Big\lVert \int_0^{t/\varepsilon^2} \calL \bfg^\varepsilon(\eta_s)\dd s \Big\rVert_{H^*}^p &\lesssim t^{\frac{p}2}\lVert (\eps^2+\calS)^{-\half} \varepsilon^{-1}\calL \bfg^\varepsilon\rVert_{\bfL^2(\pi)}^p
\lesssim t^{\frac{p}2} \lVert  \varepsilon^{-1}\calS^\half \bfg^\varepsilon\rVert_{\bfL^2(\pi)}^p
\end{equs}
and the right hand side is bounded above by $t^{\frac{p}{2}}$ by~\eqref{e:gepsB}. In the last step above, 
we decomposed $\calL$ into $\calS$ and $\calA$, bounded $(\eps^2+\calS)^{-\half}\calS$ by 
$\calS^\half$ and, for the term containing $\calA$, used that, as in~\eqref{e:gepsB}, we have 
\begin{equs}
\lVert (\eps^2+&\calS)^{-\half} \varepsilon^{-1}\calA \bfg^\varepsilon\rVert_{\bfL^2(\pi)}^2=\sum_{i=1}^\infty\E_\pi\Big[\Big|\Big(\varepsilon^{-1}(\eps^2+\calS)^{-\half}\calA \bfg^\varepsilon(\omega)\Big)[g^i]\Big|^2\Big]\\
&=\sum_{i=1}^\infty\|\varepsilon^{-1}(\eps^2+\calS)^{-\half}\calA \frakg^{i,\eps}\|^2\lesssim \sum_{i=1}^\infty  \|\varepsilon^{-1}\calS^{\half} \frakg^{i,\eps}\|^2=\lVert \eps^{-1}\calS^{\frac{1}{2}}\bfg^\eps \rVert_{\bfL^2(\pi)}^2
\end{equs}
where we further exploited Lemma~\ref{lem:32}. Concerning the martingale, 
we apply Burkholder-Davis-Gundy inequality~\cite[Theorem 1.1]{MarinelRockner16_MaximalInequalities}, and get 
\begin{equ}
\bfE \Big\lVert \sum_{i=1}^2 \int_0^{t/\varepsilon^2} \nabla_i \bfg^\varepsilon(\eta_t) \dd B^i_t \Big\rVert_{H^*}^p \lesssim_{p,T} t^{\frac{p}{2}} \lVert \varepsilon^{-1} \calS^\half \bfg^\varepsilon \rVert_{\bfL^2(\pi)}^p  \lesssim  t^{\frac{p}{2}}
\end{equ}
where again we have invoked \eqref{e:gepsB}. Therefore, the proof is concluded. 
\end{proof}

\subsection{The martingale problem for the stochastic linear transport equation}
\label{sec:mart-probl-stoch}

Given the tightness obtained in the previous section, we are left to uniquely identify the limit points. 
For this, we derive a martingale problem characterisation of the law of the SLTE in~\eqref{eq:131} 
initialised by the gradient of the GFF.

\begin{definition}
\label{def:2}
Let $\bm{\piBar}$ be the measure on $H^*$ in Definition~\ref{def:5} and $T>0$. We say that the probability measure 
$\bfPBar$ on $(C_TH^*,\calB)$, with $\calB$ the canonical Borel $\sigma$-algebra, 
solves the martingale problem for the stochastic linear transport equation~\eqref{eq:131} on $[0,T]$ 
with initial distribution $\bm{\piBar}$ and diffusivity $\varsigma^2>0$, if the canonical process $(\bm{\etaBar}_t)_{t \ge 0}$ 
under $\bfPBar$ satisfies
\begin{enumerate}
\item\label{itm:1} the law of $\bm{\etaBar}_0$ under $\bfPBar$ is $\bm{\piBar}$.  
\item\label{itm:2} $\bfPBar$-a.s. for all $t\in[0,T]$, $\bm{\etaBar}_t \in \text{dom}(\Delta)$ and 
\begin{equ}
\bfEBar\Big[\int_0^T \lVert \Delta \bm{\etaBar}_t \rVert_{H^*} \dd t\Big] \vee \bfEBar\Big[\int_0^T \lVert \partial_i \bm{\etaBar}_t \rVert_{H^*}^2 \dd t\Big] < \infty, \quad \forall i \in \{1,2\}
\end{equ}
where $\Delta,\partial_i$ for $i \in \{1, 2\}$ are defined on $H^*$ by duality.  
\item\label{itm:3} The $H^*$-valued process 
\begin{equation}
\label{eq:112}
\bfM_t \eqdef \bm{\etaBar}_t-\bm{\etaBar}_0-\frac{\varsigma^2}{2} \int_0^t \Delta\bm{\etaBar}_s ds
\end{equation}
is a continuous martingale with respect to the natural filtration of $\bm{\etaBar}$, and its quadratic variation is 
\begin{equation}
\label{eq:79}
\langle\!\langle \bfM \rangle\!\rangle_t = \varsigma^2\sum_{i=1}^2 \int_0^t (\partial_i\bm{\etaBar}_s)(\partial_i\bm{\etaBar}_s)^* \dd s\,.
\end{equation}
\end{enumerate}
\end{definition}

The requirement on the quadratic variation in~\eqref{eq:79} can be replaced by a control over quadratic functionals, 
as the next lemma shows. 

\begin{lemma}
\label{lem:11}
The quadratic variation condition \eqref{eq:79} is equivalent to requiring that for any fixed $g_1,g_2 \in H$  
\begin{equation}\label{eq:2}
\UBar_t \eqdef \, \bm{\etaBar}_t\otimes\bm{\etaBar}_t[g_1\otimes g_2] - \frac{\varsigma^2}{2} \int_0^t \bm{\etaBar}_r\otimes\bm{\etaBar}_r[\Delta(g_1\otimes g_2)]\dd r
\end{equation}
is a martingale with respect to the natural filtration of $\bm{\etaBar}$. 
\end{lemma}
\begin{proof}
By definition of quadratic variation for $H^\ast$-valued martingales~\cite[Proposition 3.13]{DaPratoZabczyk14_StochasticEquations}, the formula at the right hand side of~\eqref{eq:79} is the quadratic variation 
of $\bfM$ if and only if for every $g_1, g_2 \in H$ the process 
\begin{equ}
\VBar_t \eqdef \bfM_t[g_1]\bfM_t[g_2]  - \varsigma^2\sum_{i=1}^2 \int_0^t \partial_i\bm{\etaBar}_s[g_1] \partial_i\bm{\etaBar}_s[g_2] ds 
\end{equ}
is a martingale. Therefore, the proof is completed provided that $\UBar_t - \VBar_t $ is a martingale. 
This in turn can be argued as in the proof of~\cite[Theorem 3.4]{CannGubiToni23_GaussianFluctuations}. 
\end{proof}

The following tells us that the martingale problem for SLTE is well posed.
\begin{proposition}
\label{prp:1}
For every $T>0$, the martingale problem in Definition~\ref{def:2} is well-posed, i.e. there exists a unique 
probability measure $\bfPBar$ on $(C_TH^*,\calB)$ satisfying items~\ref{itm:1}-\ref{itm:3}. 
\end{proposition}

We believe the proof of the above proposition is classical and follows from standard arguments 
(see e.g.~\cite{DaPratoZabczyk14_StochasticEquations}). That said, we were unable to find a specific 
reference and therefore we provide a proof in Appendix~\ref{sec:well-posedn-mart}. 

\subsection{Proof of Theorem \ref{thm:2}, convergence to SLTE}
\label{sec:conv-cond-good}

The goal of this section is to complete the proof of Theorem \ref{thm:2}. As done for the SRBP with Theorem~\ref{thm:3}, 
we will first identify a set of conditions under which the statement holds (Theorem~\ref{thm:4}) and then focus  
on the main result. 

In the following, for a sequence of (not necessarily basis) elements 
$(g_i)_{i = 1}^m \subset H \setminus \{0\}$, $m \in \bbN$, and for $\varepsilon>0$, define 
$\frakh^\varepsilon \in \fock_m$ as 
\begin{equ}[e:frakh]
\frakh^\varepsilon(p_{1:m})= \frakh^\varepsilon_{g_1,\dots,g_m}(p_{1:m}) \eqdef \frac{1}{m!} \sum_{\sigma \in S_m} \frakg_1^\varepsilon(p_{\sigma(1)}) \ldots \frakg_m^\varepsilon(p_{\sigma(m)})
\end{equ}
where $\frakg_i\in\fock_1$ is the element associated to $g_i$ via the embedding $H \hookrightarrow \fock_1$ 
(see~\eqref{eq:42} with $n=1$) and, as above, $\frakg_i^\eps(\cdot)=\frakg_i(\eps^{-1}\cdot)$. 
Since we will need to invoke the estimates on the nuisance region, Proposition \ref{p:MainEstimates}, 
for $c\in(0,\half)$, we also define 
\begin{equ}[e:frakhcut]
\frakh^{\varepsilon,c}(p_{1:m}) \eqdef \frakh^\varepsilon(p_{1:m}) \1_{\{|\sum_{i=1}^m p_i| \ge c\sum_{i=1}^m|p_i|\}}\,,
\end{equ}
if $m\geq 2$, and $\frakh^{\varepsilon,c}=\frakh^{\varepsilon}=\frakg_1^\varepsilon$ if $m=1$.

\begin{theorem}
\label{thm:4}
Suppose that for each $m\in \N$ and $\frakh \in \fock_m$, there exists a collection of elements 
$\{v^{\lambda, n} \colon \lambda \in (0,1), n \in \bbN\} \subset \text{dom}(\calL)$ such that for each 
$\lambda \in (0,1), n \in \bbN$, the map $\frakh \mapsto v^{\lambda,n}[\frakh]$ is linear from 
$\fock_m \rightarrow\fock_{\le n}$. 

Assume further that for all $c \in (0,\half)$, $m\in\{1,2\}$, we have 
\begin{align}
\lim_{n\rightarrow\infty} \limsup_{\lambda\rightarrow0} \sup_{g_1, ..., g_m} \lVert v^{\lambda,n} - \frakh\rVert^2 &= 0\,,\label{eq:32} \\
\lim_{n\rightarrow\infty} \limsup_{\lambda\rightarrow0} \sup_{g_1, ..., g_m} \lambda^{-1} \lVert (\lambda+\calS)^{-\half} \big[-\calL v^{\lambda,n} - (1+\sigma^2(\alpha)) \calS  \frakh \big] \rVert^2 &= 0\label{eq:34}
\end{align}
where, for $g_1,\dots,g_m\in H$, 
$\frakh= \big(\prod_{i=1}^m \lVert g_i\rVert_{H^2(\bbR^2,\bbR^2)} \big)^{-1} \frakh^{\sqrt{\lambda},c}$ with 
$\frakh^{\sqrt{\lambda},c}=\frakh^{\sqrt{\lambda},c}_{g_1,\dots g_m}$ given as in~\eqref{e:frakhcut} and~\eqref{e:frakh}, 
$v^{\lambda,n}=v^{\lambda,n}[\frakh]$ and 
$\sigma^2(\alpha)$ is defined according to~\eqref{eq:87}. 
Then the conclusion of Theorem \ref{thm:2} holds.
\end{theorem}

\begin{proof}
By tightness, Lemma \ref{lem:13}, there exists a subsequence of $\{\bm{\eta}^\varepsilon\}_\eps$ (which, 
by a slight abuse of notation we will still denote by $\bm{\eta}^\varepsilon$) which converges 
almost surely in $C_T H^\ast$. Let $\bm{\etaBar}$ be the limit and $\bfPBar$ its law. 
Thanks to Proposition \ref{prp:1}, 
it suffices to show that $\bfPBar$ solves the martingale problem in Definition \ref{def:2} with 
$\varsigma^2=1+\sigma^2(\alpha)$. 
Item \ref{itm:1} is verified in Lemma \ref{lem:13}, while Item \ref{itm:2} follows immediately 
by the fact that, since $\bm{\eta}^\varepsilon\to\bm{\etaBar}$ in $C_TH^\ast$ and $\bm{\eta}^\varepsilon$ 
is stationary, so is $\bm{\etaBar}$ and $\text{law}(\bm{\etaBar})=\bm{\piBar}$. Hence, for $j=0,1,2$ and $i=1,2$,
\begin{equ}
\bfE\Big[\int_0^T \lVert\partial_i^j\bm{\etaBar}_t\rVert_{H^*}^2 \dd t\Big]=\int_0^T \bfE\lVert\partial_i^j\bm{\etaBar}_t\rVert_{H^*}^2 \dd t \lesssim T \sum_k\|g^k\|_{H^j(\R^2,\R^2)}^2 < \infty
\end{equ}
the last step being a consequence of~\eqref{eq:8}. 
For Item \ref{itm:3}, we set $\varsigma^2=1+\sigma^2(\alpha)$ and decompose its verification into two steps. 
\medskip

\noindent\textit{Step 1. The martingale property.} 
To see that the process in~\eqref{eq:112} is a martingale we are required to show that for any 
$0\le s\le t\le T$ and $G : C([0,s],H^*) \rightarrow \bbR$ bounded and continuous, we have 
\begin{equ} 
\bfEBar\Big[ \Big( \bm{\etaBar}_t - \bm{\etaBar}_s - \frac{\varsigma^2}{2} \int_s^t \Delta \bm{\etaBar}_r \dd r \Big) G((\bm{\etaBar}_r)_{r \in [0,s]})\Big] = 0\,, 
\end{equ}
which in turn follows upon proving that 
\begin{align}
\limsup_{\varepsilon\rightarrow0} \bfE \Big[ \Big\lVert &\big( \bm{\eta}^\varepsilon_t - \bm{\eta}^\varepsilon_s - \frac{\varsigma^2}{2} \int_s^t \Delta \bm{\eta}^\varepsilon_r \dd r \big) G((\bm{\eta}^\varepsilon_r)_{r \in [0,s]}) \Big\rVert_{H^*}^2 \Big] < \infty \label{eq:129}\\
\lim_{\varepsilon\rightarrow0} \Big\lVert \bfE \Big[ &\Big( \bm{\eta}^\varepsilon_t - \bm{\eta}^\varepsilon_s - \frac{\varsigma^2}{2} \int_s^t \Delta \bm{\eta}^\varepsilon_r \dd r \Big) G((\bm{\eta}^\varepsilon_r)_{r \in [0,s]}) \Big] \Big\rVert_{H^*} = 0\,. \label{eq:130}
\end{align}
The former follows by boundedness of $G$ together with stationarity of $\bm{\eta}^\varepsilon$ and~\eqref{e:gepsB}. 
Concerning the latter, define, for $n \in \bbN, \lambda \in (0,1)$, $\bfv^{\lambda,n}\in\bfL^2(\pi)$ 
according to $\bfv^{\lambda,n}[g] = v^{\lambda,n}[\frakg^{\sqrt{\lambda}}]$, $g\in H$. 
Let $(\bfM_t)_{t\ge0}$ be the $H^*$-valued Dynkin martingale associated to $\bfv^{\eps^2,n}$ 
as given in~\eqref{eq:119} and set $\bfM^\varepsilon_t = \bfM_{t/\varepsilon^2}$, 
which is a martingale adapted to $(\calF^\varepsilon_t)_{t \in [0,T]}$, where $\calF^\varepsilon_t = \calF_{t/\varepsilon^2}$. 

Since $\bfM^\varepsilon$ is a martingale, \eqref{eq:130} is implied by 
\begin{equation}
\label{eq:29}
\lim_{n\rightarrow\infty} \limsup_{\varepsilon\rightarrow0} \bfE \Big\lVert \bfM^\varepsilon_t - \bfM^\varepsilon_s - \Big( \bm{\eta}^\varepsilon_t - \bm{\eta}^\varepsilon_s - \frac{\varsigma^2}{2} \int_s^t \Delta \bm{\eta}^\varepsilon_r \dd r \Big) \Big\rVert_{H^*}^2 = 0
\end{equation}
By~\eqref{e:scaling} stationarity and Lemma \ref{lem:2} with the choice $p=2, \lambda = \varepsilon^2$, we obtain
\begin{align*}
\bfE & \Big\lVert \bfM^\varepsilon_t - \bfM^\varepsilon_s - \Big( \bfg^\eps(\eta^\varepsilon_t)- \bfg^\eps(\eta^\varepsilon_s) - \frac{\varsigma^2}{2} \int_s^t \Delta \bfg^\eps(\eta^\varepsilon_r) \dd r \Big) \Big\rVert_{H^*}^2 \\
&\lesssim_{s, t} \lVert \bfv^{\varepsilon^2,n} - \bfg^\varepsilon \rVert_{\bfL^2(\pi)}^2 + \varepsilon^{-2} \lVert (\varepsilon^2+\calS)^{-\half} \big(-\calL\bfv^{\varepsilon^2,n} - \varsigma^2\calS \bfg^\varepsilon \big) \rVert_{\bfL^2(\pi)}^2 \\
&= \sum_{i=1}^\infty \big( \lVert v^{\lambda,n}[\frakg^{i,\varepsilon}] - \frakg^{i,\varepsilon}\rVert^2 + \varepsilon^{-2} \lVert (\varepsilon^2+\calS)^{-\half} \big(-\calL v^{\lambda,n}[\frakg^{i,\varepsilon}] - \varsigma^2\calS \frakg^{i,\varepsilon}\big) \rVert^2 \big) 
\end{align*}
Conditions \cref{eq:32,eq:34} imply that, for each $i$, the $i$-th summand is bounded above by 
$ \lVert g_i\rVert_{H^2(\bbR^2,\bbR^2)}^2$, which is summable by~\eqref{eq:8}. 
Hence, by dominated convergence, the same conditions also imply~\eqref{eq:29} from which 
we conclude that the right hand side of~\eqref{eq:112} is a martingale. 
\medskip

\noindent \textit{Step 2. Quadratic variation.} In view of Lemma \ref{lem:11}, we are required to show that $\UBar$ 
in~\eqref{eq:2} is a martingale. 
Arguing as in the step above, this follows provided that for all $0\le s\le t\le T$ and 
$G : C([0,s],H^*) \rightarrow \bbR$ bounded continuous we have 
\begin{align}
\limsup_{\varepsilon\rightarrow0} \bfE | ( U^\varepsilon_t - U^\varepsilon_s ) G((\bm{\eta}^\varepsilon_r)_{r \in [0,s]}) |^2 &< \infty \label{eq:136} \\
\limsup_{\varepsilon\rightarrow0} | \bfE [( U^\varepsilon_t - U^\varepsilon_s ) G((\bm{\eta}^\varepsilon_r)_{r \in [0,s]})] | &= 0  \label{eq:137}
\end{align}
where $U^\varepsilon$ is defined as $\UBar$ but with $\bm{\eta}^\varepsilon$ replacing $\bm{\etaBar}$. 
Note that $U^\varepsilon_t = U_{t/\varepsilon^2}$, where $U_t$ is given by
\begin{equation}
\label{eq:138}
U_t = \frakh^\varepsilon(\eta_t) - \frakh^\varepsilon(\eta_0) + \varsigma^2\int_0^t\calS \frakh^\varepsilon(\eta_r) \dd r
\end{equation}
for $\frakh^\eps$ as in~\eqref{e:frakh} and more specifically, $\frakh^\varepsilon(\eta_{t/\varepsilon^2})= \bm{\eta}^\varepsilon_t\otimes\bm{\eta}^\varepsilon_t[g_1\otimes g_2]- \langle g_1^\varepsilon,g_2^\varepsilon\rangle
=\bm{\eta}^\varepsilon_t[g_1]\bm{\eta}^\varepsilon_t[g_2]-\langle\frakg_1^\varepsilon,\frakg_2^\varepsilon\rangle$. 

Now, the verification of~\eqref{eq:136} follows the same steps as those for~\eqref{eq:129}, and is therefore omitted. 
For~\eqref{eq:137}, define $(U^c_t)_{t\ge0}$ as in \eqref{eq:138} but with $\frakh^{\varepsilon,c}$ replacing $\frakh^\varepsilon$, and $(U^{\varepsilon,c}_t)_{t\ge0}$ according to $U^{\varepsilon,c}_t = U^c_{t/\varepsilon^2}$. 
We are therefore reduced to show that for all $c\in(0,\half)$, 
\begin{align}
\lim_{c\rightarrow0} \limsup_{\varepsilon\rightarrow0} \bfE | ( U^\varepsilon_t - U^\varepsilon_s ) - ( U^{\varepsilon,c}_t - U^{\varepsilon,c}_s ) |^2 &= 0\,, \label{eq:37} \\
\limsup_{\varepsilon\rightarrow0} | \bfE ( U^{\varepsilon,c}_t - U^{\varepsilon,c}_s ) G((\bm{\eta}^\varepsilon_r)_{r \in [0,s]}) | &= 0\,.   \label{eq:40} 
\end{align}
Let us begin with the latter. For $n\in\N$ and $c\in(0,\half)$, let $M^c$ be the Dynkin martingale as in~\eqref{e:Dynkin}
associated to $v^{\varepsilon^2, n}[\frakh^{\varepsilon,c}]$ and $M^{\varepsilon,c}_t \eqdef M_{t/\varepsilon^2}$. 
By stationarity and Lemma \ref{lem:1}, we have 
\begin{align*}
\bfE |(M^{\varepsilon,c}_t-&M^{\varepsilon,c}_s) - ( U^{\varepsilon,c}_t - U^{\varepsilon,c}_s )|^2  \\
&\lesssim_{s, t}\lVert v^{\varepsilon^2,n} - \frakh^{\varepsilon,c} \rVert^2 + \varepsilon^{-2} \lVert (\varepsilon^2+\calS)^{-\half} \big(-\calL v^{\varepsilon^2,n} - \varsigma^2\calS \frakh^{\varepsilon,c} \big) \rVert^2
\end{align*}
and the right hand side converges to $0$ in view of~\cref{eq:32,eq:34}. Hence,~\eqref{eq:40} follows. 

For~\eqref{eq:37}, we need to control both the boundary terms and the time integral, for which we will 
use Lemma~\ref{lem:2}. Summarising, we need to control   
\begin{equs}
\lVert \frakh^\varepsilon - \frakh^{\varepsilon,c} \rVert^2 &\lesssim \int_{|p_1+p_2| \le c(|p_1|+|p_2|)} |\gHat_1(p_1)|^2 |\gHat_2(p_2)|^2 \dd p _{1:2}\\
\varepsilon^{-2}\lVert \calS^\half (\frakh^\varepsilon - \frakh^{\varepsilon,c}) \rVert^2 &\lesssim  \int_{|p_1+p_2| \le c(|p_1|+|p_2|)}  |p_1 + p_2|^2 |\gHat_1(p_1)|^2 |\gHat_2(p_2)|^2 \dd p _{1:2}
\end{equs} 
where the bounds above are uniform in $\eps$ and converge to $0$ as $c\rightarrow0$ by dominated convergence. 
Therefore, the proof of~\eqref{eq:37}, and consequently that of the theorem, is concluded. 
\end{proof}

We are now ready to complete the proof of Theorem \ref{thm:2}.  

\begin{proof}[of Theorem \ref{thm:2}]
In view of Theorem~\ref{thm:4}, all we have to do is to identify a family of observables for which \cref{eq:32,eq:34} hold. 
For $m=1,2$, $\frakh \in \fock_m$, these will be given by the solution of the replacement 
equation $v^{\lambda,n} = v^{\lambda,n}[\frakh]$ in Definition~\ref{def:3} corresponding to 
$\frakf = (\lambda+\calS+\calS\calG^\lambda) \frakh \in \fock_m$. 
For $c\in(0,\half)$ and $g_1,\dots,g_m\in H$, set 
$\frakf^{\sqrt{\lambda},c}\eqdef (\lambda+\calS+\calS\calG^\lambda) \frakh^{\sqrt{\lambda},c}$ with 
$\frakh^{\sqrt{\lambda},c}$ defined according to~\eqref{e:frakh} and~\eqref{e:frakhcut}. 

By construction, $\frakf^{\sqrt{\lambda},c}$ satisfies the reverse triangle inequality with respect to $c$, 
therefore we may invoke Proposition \ref{p:MainEstimates}, according to which, for $\lambda\in(0,\bar\lambda)$, 
we have 
\begin{equ}[e:L2Final?]
\lVert v^{\lambda,n} - \frakh^{\sqrt{\lambda},c} \rVert \lesssim \sqrt{n} \coup \lambda^{-\half}\lVert (\lambda+\calS)^{\half} \frakh^{\sqrt{\lambda},c} \rVert
\end{equ}
and 
\begin{equs}[e:H1Final?]
\lambda^{-\half} \lVert &(\lambda+\calS)^{-\half} \big[-\calL v^{\lambda,n} - (1+\sigma^2(\alpha)) \calS  \frakh^{\sqrt{\lambda},c} \big] \rVert\\
\lesssim& \lVert v^{\lambda,n} - \frakh^{\sqrt{\lambda},c} \rVert+\lambda^{-\half}\|(\lambda+\calS)^{-\half}\calS\big[\calG^\lambda-\sigma^2(\alpha)\big]\frakh^{\sqrt{\lambda},c}\|\\
&+\lambda^{-\half}\lVert (\lambda+\calS)^{-\half} \big[ (\lambda-\calL)v^{\lambda, n} - \frakf^{\sqrt{\lambda},c} \big] \rVert\\
\lesssim&
\lambda^{-\half}\|\calS^\half\big[\calG^\lambda-\sigma^2(\alpha)\big]\frakh^{\sqrt{\lambda},c}\| \\
&+\sqrt{n}\lambda^{-\half}\vertiii{\frakf^{\sqrt{\lambda},c}} +\lambda^{-\half}\Big(\frac{1}{\sqrt{n}}+\gamma C\Big)\|(\lambda+\calS)^{\half}\frakh^{\sqrt{\lambda},c}\|
\end{equs}
where we further used~\eqref{e:L2Final?} and the seminorm $\vertiii{\cdot}$ was defined in~\eqref{e:VanNorm}. 
To conclude, we need to control the right hand sides of~\cref{e:L2Final?,e:H1Final?} in terms of the $H^2$-norms 
of $g_1,\dots,g_m$, and a constant which vanishes in the double limit $\lambda\to 0$ first and $n\to\infty$ after. 
We will focus on the case $m=2$ as the case $m=1$ follows similar steps and is simpler. 
Arguing as in~\eqref{e:gepsB2}, we have
\begin{equs}
\lambda^{-1}&\lVert (\lambda+\calS)^{\half} \frakh^{\sqrt{\lambda},c} \rVert^2\lesssim \lambda^{-1}\int \frac{\lambda+|p_{[1:2]}|^2}{|p_1|^2|p_2|^2}|\frakg_1^{\sqrt{\lambda}}(p_1)\frakg_2^{\sqrt{\lambda}}(p_2)|^2\dd p_{1:2}\\
&=\int \frac{1+|p_{[1:2]}|^2}{|p_1|^2|p_2|^2}|\frakg_1(p_1)\frakg_2(p_2)|^2\dd p_{1:2}\lesssim \lVert g_1\rVert_{H^1(\bbR^2,\bbR^2)}^2\lVert g_2\rVert_{H^1(\bbR^2,\bbR^2)}^2\,.
\end{equs}
The explicit form of the operator $\calG^\lambda$ in~\eqref{e:ApproxFPO} gives
\begin{equs}
\lambda^{-1}&\|\calS^\half\big[\calG^\lambda-\sigma^2(\alpha)\big]\frakh^{\sqrt{\lambda},c}\|^2\\
&\lesssim \lambda^{-1}\int |p_{[1:2]}|^2\frac{|g(\ell^\lambda(\tfrac12 |p_{[1:2]}|^2))-\sigma^2(\alpha)| }{|p_1|^2|p_2|^2}|\frakg_1^{\sqrt{\lambda}}(p_1)\frakg_2^{\sqrt{\lambda}}(p_2)|^2\dd p_{1:2}\\
&=\int |p_{[1:2]}|^2\frac{|g(\ell^\lambda(\tfrac{\lambda}2 |p_{[1:2]}|^2))-\sigma^2(\alpha)| }{|p_1|^2|p_2|^2}|\frakg_1(p_1)\frakg_2(p_2)|^2\dd p_{1:2}\\
&\lesssim \coup^2 \int \frac{|p_{[1:2]}|^4}{|p_1|^2|p_2|^2}|\frakg_1(p_1)\frakg_2(p_2)|^2\dd p_{1:2}\lesssim \coup^2\lVert g_1\rVert_{H^2(\bbR^2,\bbR^2)}^2\lVert g_2\rVert_{H^2(\bbR^2,\bbR^2)}^2
\end{equs}
where we applied mean value theorem and~\eqref{e:gbasic} in the third step. 
Therefore, we are left to consider $\lambda^{-\half}\vertiii{\frakf^{\sqrt{\lambda},c}}$ for which 
the result follows from the claim 
\begin{equ}
\lambda^{-1} \Big(\lVert (\lambda+\calS)^{-\half} \calA^N_+ \frakh^{\sqrt{\lambda},c}\rVert^2\vee \|(\lambda+\calS)^{-\half}\calA_- \frakh^{\sqrt{\lambda},c}\|^2\Big)  \lesssim \gamma^2 \prod_{i=1}^2 \lVert g_j\rVert_{H^1(\bbR^2,\bbR^2)}^2\,.
\end{equ}
The term containing $\calA^N_+$ can be treated using the decomposition \eqref{eq:ARi}, bounding the off-diagonal terms with the diagonal terms, we obtain the upper bound
\begin{align*}
\coup^2 \lambda^{-2} \int |p_{[1:2]}|^2 |\frakg_1^{\sqrt{\lambda}}(p_1)|^2 |\frakg_2^{\sqrt{\lambda}}(p_2)|^2 \big( \int N(p_3, p_{[1:2]}) \dd p_3 \big) \mu_2(\dd p _{1:2}) \, .
\end{align*}
where we have invoked the bound $(\lambda + \calS)^{-1} \le \lambda^{-1}$. Performing the inner integral in $p_3$ produces an additional factor of $|p_{[1:2]}|^2$, and the bound $\gamma^2 \prod_{i=1}^2 \lVert g_j\rVert_{H^1(\bbR^2,\bbR^2)}^2$ follows from a computation analogous to \eqref{e:gepsB2} in the case $\delta = 2$. For the other term, by~\eqref{eq:ops}, we have 
\begin{align*}
\lVert(\lambda+\calS)^{-\half}\calA_- \frakh^{\sqrt{\lambda},c} \rVert^2 &\lesssim \coup^2 \int |p_2| |p_3| |\frakg_1^{\sqrt{\lambda}}(p_1)|^2 |\frakg_2^{\sqrt{\lambda}}(p_2)| |\frakg_2^{\sqrt{\lambda}}(p_3)| \mu_3(\dd p _{1:3}) 
\end{align*}
where we omitted the symmetric term obtained by swapping $1$ and $2$ at the right hand side as it can be 
similarly bounded. 
We rewrite the above in terms of 
\begin{equ}
 \Phi = \Big(\frac{\lambda+|p_2|^2}{\lambda+|p_3|^2}\Big) |p_3| |\frakg_2^{\sqrt{\lambda}}(p_2)|, \qquad \Phi^\prime =  \Big(\frac{\lambda+|p_3|^2}{\lambda+|p_2|^2}\Big) |p_2| |\frakg_2^{\sqrt{\lambda}}(p_3)| 
 \end{equ}
and estimate their product via $2\Phi\Phi'\leq (\Phi)^2+(\Phi')^2$. The two summands we obtain are the same, 
and can be bounded by 
\begin{align*}
\coup^2 \int |\frakg_1^{\sqrt{\lambda}}(p_1)|^2 |\frakg_2^{\sqrt{\lambda}}(p_2)|^2 \Big(\frac{\lambda+|p_2|^2}{\lambda+|p_3|^2}\Big)^2|p_3|^2 \mu_3(\dd p _{1:3}) \lesssim \coup^2 \lambda \prod_{i=1}^2 \lVert g_j\rVert_{H^2(\bbR^2,\bbR^2)}^2
\end{align*}
where in the last line we have used Lemma~\ref{lem:22} with $q=r=0$. 
The statement then follows by collecting the bounds obtained so far.

\end{proof}

\appendix 

\section{Integral estimates} 
\label{sec:integral-estimates}
In this appendix, we collect estimates and integral computations which are used throughout the paper. 

\subsection{Basic Estimates}

The following lemma is a simplified version of what can be found in the appendix of \cite[Lemma A.3]{CannHaunToni22_SqrtLog}. 

\begin{lemma}
\label{lem:4}
Uniformly over $p,r \in \bbR^2$, we have 
\begin{equ}
 |p| \int_{\bbR^2} \frac{\VHat(q)}{|q+r|(\lambda+|q|^2+|p|^2)} \dd q \lesssim 1 \,.
 \end{equ}
\end{lemma}
\begin{proof}
We split the integral in three regions, $R_1 = \{|q+r|<|p|\}$, $R_2 := \{|q|<|p|\}\cap\{|q+r| \ge |p|\}$ 
and the complement of their union, and we treat each of them separately. For the first, we have  
\begin{equ} 
|p| \int_{R_1} \frac{\VHat(q)}{|q+r|(\lambda+|q|^2+|p|^2)} \dd q \le \frac{1}{|p|} \int_{R_1} \frac{1}{|q+r|} \dd q \lesssim 1\,.
\end{equ}
For $R_2$ instead, 
\begin{equ}
|p| \int_{R_2} \frac{\VHat(q)}{|q+r|(\lambda+|q|^2+|p|^2)} \dd q \le \frac{1}{|p|^2} \int_{R_2} \VHat(q) \dd q \lesssim 1 \,.
\end{equ}
At last notice that $(R_1 \cup R_2)^c = \{|q+r| \ge |p|\} \cap \{|q|\ge|p|\}$, so that 
Holder's inequality with exponents $(\frac{1}{3},\frac{2}{3})$ gives 
\begin{align*}
|p| &\int_{(R_1 \cup R_2)^c} \frac{\VHat(q)}{|q+r|(\lambda+|q|^2+|p|^2)} \dd q \\
&\le |p| \Big( \int_{\{|q+r| \ge |p|\}} \frac{\VHat(q)}{|q+r|^3} \dd q \Big)^{\frac{1}{3}} \Big( \int_{\{|q| \ge |p|\}} \frac{\VHat(q)}{(\lambda+|q|^2)^{\frac{3}{2}}} \dd q \Big)^{\frac{2}{3}}\lesssim 1\,.
\end{align*}
\end{proof}

\begin{lemma}
\label{lem:10}
Let $\kappa\in(0,1)$ be fixed, $N^\kappa$ be the nuisance region in~\eqref{e:BNregions} and 
$\gamma$ be given by~\eqref{eq:Weak}. 
Then, uniformly over $\lambda \in (0,1)$ and $r \in \bbR^2$, we have
\begin{equation}
\label{eq:113}
\coup^2\int \frac{\VHat(p)\VHat(q)N^\kappa(q,p) |p+q| }{(\lambda+|p+q+r|^2)(\lambda+|q|^2+|r|^2)^{\frac32}} \dd p \dd q \lesssim 1\,.
\end{equation}
\end{lemma}
\begin{proof}
Recall that, by~\eqref{e:BNregionsBound}, if $N^\kappa(p,q)=1$, then $|p+q|\lesssim |p|$ 
and $|q|\gtrsim |p|$. Now, consider first the restriction of the integral in~\eqref{eq:113} to the region $\{|p+q+r| \ge 1\}$. 
This is bounded above by
\begin{equ}
\coup^2 \int\dd p \VHat(p) |p|\int_{|q|\gtrsim |p|}\frac{\dd q}{(\lambda+|q|^2)^{\frac32}} \lesssim \gamma^2\,.
\end{equ}
Instead, on $\{|p+q+r| \le 1\}$, we apply the change of variables $p+q\mapsto p$ and get 
\begin{equ}
\coup^2  \int_{|p+r| \le 1}\frac{|p|\dd q}{(\lambda+|p+r|^2)} \int_{|q|\gtrsim |p|} \frac{\dd p}{(\lambda+|q|^2)^{\frac32}}   \lesssim 1 \,,
\end{equ}
from which~\eqref{eq:113} follows at once. 
\end{proof}

\begin{lemma}
\label{lem:22}
Uniformly over $\lambda \in (0,1)$, $q,r \in \bbR^2$, we have
\begin{equation}
\label{eq:51}
\lambda \int\frac{\VHat(p)}{(\lambda+|p+q|^2)(\lambda+|p+r|^2)} \dd p  \lesssim 1\,.
\end{equation}
\end{lemma}
\begin{proof}
Applying Cauchy-Schwarz to~\eqref{eq:51}, the integral is bounded by
\begin{equ}
\lambda \Big(\int\frac{\VHat(p)}{(\lambda+|p+q|^2)^2} \dd p \Big)^\half  \Big(\int\frac{\VHat(p)}{(\lambda+|p+r|^2)^2} \dd p \Big)^\half\,. 
\end{equ}
The two factors can be treated similarly, and can be controlled by
\begin{equ}
\int\frac{\VHat(p)}{(\lambda+|p+q|^2)^2} \dd p  \lesssim 1 + \int_{|p+q| \le 1}\frac{1}{(\lambda+|p+q|^2)^2} \dd p  \lesssim 1+\frac{1}{\lambda}\,.
\end{equ}
\end{proof}

\subsection{Estimates for the replacement lemma}
\label{sec:estim-repl-lemma}

The goal of this appendix is to show the crucial estimate in the proof of the Replacement Lemma~\ref{lem:3}, 
i.e.~\eqref{eq:10}. For the reader's convenience, let us recall some notation. 
Let $\cg^\lambda$ be the multiplier associated to the operator $\cG^\lambda$ in~\eqref{e:ApproxFPO}, 
and $\cm_\lambda$ be given as in~\eqref{eq:DiKer} with 
$\tau(p)=\tfrac12|p|^2\cg^\lambda(p)$, $p\in\R^2$ (see~\eqref{e:ApproxFPO}) and 
$R$ replaced by the bulk region $B=B^{\frac13}$ in~\eqref{e:BNregions}. 
More explicitly, for $p\in\R^2$, we have 
\begin{equs}
\cg^\lambda(p)&\eqdef g(\ell^\lambda(\tfrac12|p|^2))=\sqrt{4\pi\gamma^2 \log\Big(1+\frac{1}{\lambda +\frac12|p|^2}\Big)+1}-1\label{eq:cglambda}\\
\cm^\lambda(p)&\eqdef 2\coup^2 \int_{\bbR^2} \frac{\VHat(q) B(q, p) \cos^2 \theta}{\lambda+\half|p+q|^2[1 + \cg^\lambda(p+q)]} \dd q\label{eq:mlambdaApp}
\end{equs}
where $g$ and $\ell^\lambda$ are defined according to~\eqref{e:g}. 

We are now ready to state and prove the next proposition. The proof follows closely that of similar 
results as~\cite[Lemma 2.6]{CannGubiToni23_GaussianFluctuations} and we therefore limit ourselves to outline the main steps. 

\begin{lemma}\label{lem:ReplEst}
For $\lambda\in(0,1)$ and $\gamma$ as in~\eqref{eq:Weak}, let $\cg^\lambda$ and $\cm^\lambda$ 
be as in~\eqref{eq:cglambda} and~\eqref{eq:mlambdaApp} respectively. Then~\eqref{eq:10} holds, i.e. 
\begin{equ}
\sup_{p}|\cm^\lambda(p)-\cg^\lambda(p)|\lesssim \gamma^2\,.
\end{equ}
\end{lemma}
\begin{proof}
Since $\gamma$ is given according to weak coupling, the following properties of $\cg^\lambda$, 
which will be used throughout, hold uniformly over $\lambda\in(0,1)$, 
\begin{equs}[e:gbasic]
0 \le \cg^\lambda(p) \le \cg^\lambda(0) = \sqrt{4\pi \alpha^2 +1} - 1\,,\qquad &\text{for all $p\in\R^2$}\\
 |\partial_\rho g(\ell^\lambda(\rho))|  \lesssim \frac{\coup^2}{(\lambda+\rho)(1+g(\ell^\lambda(\rho)))}\qquad &\text{for all $\rho\geq 0$}\,.
\end{equs} 

The proof consists of massaging the expression for $\cm^\lambda$ by successive replacements until we obtain 
a consistency relation with $\cg^\lambda$, and more specifically with $g$. All the replacements will be such that 
the error made in the difference is bounded by $\gamma^2$, uniformly in $p\in\R^2$ and $\lambda\in(0,1)$. 
First, we claim that~\eqref{eq:mlambdaApp} can be substituted, up to an error of order $\gamma^2$, with 
\begin{equation}
\label{eq:11}
2\coup^2 \int\frac{\VHat(q)B(q,p)\cos^2\theta}{\lambda+\half (|p|^2+|q|^2)[1+\cg^\lambda( p+q)]} \dd q\,.
\end{equation}
Indeed, their difference is bounded above by 
\begin{equs}
\coup^2 &\int \frac{\VHat(q)B(q,p) |p| |q| [1+\cg^\lambda( p+q)]}{(\lambda+(|p|^2+|q|^2)[1+\cg^\lambda( p+q)])(\lambda+|p+q|^2[1+\cg^\lambda( p+q)])} \dd q \nonumber \\
 &\le \coup^2 \int\frac{\VHat(q)B(q,p)|p| |q|}{(\lambda+|p|^2+|q|^2)|p+q|^2}\dd q\lesssim\coup^2 \int\frac{\VHat(q)|p| }{|p+q|(\lambda+|p|^2+|q|^2)} \dd q \lesssim \coup^2 
\end{equs}
where we used that $\cg^\lambda$ is non-negative, that on $B$, $|p+q| \gtrsim |q|$, and Lemma \ref{lem:4}. 

Second, we replace~\eqref{eq:11}, up to an error of order $O(\gamma^4)$, with
\begin{equation}
\label{eq:13}
2\coup^2 \int\frac{\VHat(q)B(q,p)\cos^2\theta}{\lambda+\half (|p|^2+|q|^2)[1+g(\ell^\lambda(\half (|p|^2+|q|^2)))]} \dd q
\end{equation}
where $g$ and $\ell^\lambda$ are defined according to~\eqref{e:g}. 
The difference can be controlled by 
\begin{equs}
\coup^2& \int\frac{\VHat(q)B(q,p)|g(\ell^\lambda(\half (|p|+|q|^2)))-g(\ell^\lambda(\half |p+q|^2))|}{\lambda+|p|^2+|q|^2} \dd q\\
&\lesssim \coup^4 \int\frac{\VHat(q)|p||q| }{(\lambda+|p|^2+|q|^2)^2} \dd q \lesssim  \coup^4 |p| \int\frac{\VHat(q)}{|q|(\lambda+|p|^2+|q|^2)} \dd q \lesssim \coup^4\,.
\end{equs}
where in the last step we used Lemma \ref{lem:4}, while in the first, 
we applied the mean value theorem and~\eqref{e:gbasic}, which give
\begin{equs}
|g(\ell^\lambda(\half (|p|+|q|^2)))-g(\ell^\lambda(\half |p+q|^2))|&\lesssim |p||q|\sup_{\rho\in I_{p,q}}\frac{\coup^2}{(\lambda+\rho)(1+g(\ell^\lambda(\rho)))}
\end{equs}
for $I_{p,q}\eqdef[\half |p+q|^2\wedge\half (|p|^2+|q|^2),\half |p+q|^2\vee\half (|p|^2+|q|^2)]$. Now, since 
on $B$, $|p+q|^2 \gtrsim |p|^2+|q|^2$, the stated bound follows. 

Third, we remove the indicator function of the bulk region in~\eqref{eq:13}. 
By definition, the difference of the two has the same expression but with the nuisance region in place of the 
bulk one, and can be bounded by 
\begin{equ}
\coup^2 \int\frac{\VHat(q) N(q, p)}{\lambda+|p|^2+|q|^2} \dd q \lesssim \frac{\coup^2}{|p|^2}  \int \VHat(q) N(q, p) \dd q \lesssim \coup^2\,.
\end{equ}
Moreover, it is not hard to see that, again up to an error of order $\gamma^2$, we can restrict the integral to $|q|\leq 1$. 

Fourth, we first restrict the integral to $|q|\leq 1$ 
and then remove $\hat V$, which is possible since, by Assumption~\ref{a:V}, $\hat V$ is smooth and therefore 
satisfies $|\hat V(q)-1|\lesssim |q|$ uniformly withing $|q|\leq 1$. Both operations can be easily seen to produce an error of order $\gamma^2$. 
In this way, we obtained
\begin{equs}
2\coup^2& \int_{|q| \le 1} \frac{\cos^2\theta}{\lambda+\half (|p|^2+|q|^2)[1+g(\ell^\lambda(\half (|p|^2+|q|^2)))]} \dd q\\
&=2\coup^2\int_0^{2\pi} \cos^2\theta\dd \theta \int_0^1\frac{r\dd r}{\lambda +\half (|p|^2+r^2)[1+g(\ell^\lambda(\half (|p|^2+r^2)))]}\\
&=2\pi \gamma^2\int_{\frac12|p|^2}^{\frac12(1+|p|^2)} \frac{\dd\rho}{\lambda+\rho[1+g(\ell^\lambda(\rho))]}\,.
\end{equs}
Fifth, we replace the previous  with 
\begin{equation}
\label{eq:23}
2\pi \coup^2  \int_{\half |p|^2}^{\frac12(1+|p|^2)} \frac{\dd \rho}{(\lambda+\rho)(\lambda + \rho + 1)[1+g(\ell^\lambda(\rho))]} 
\end{equation}
which is allowed since the difference is bounded by 
\begin{equ}
\coup^2 \Big(  \int_{\half |p|^2}^{\frac12(1+|p|^2)} \frac{1}{\lambda + \rho + 1} d\rho + \lambda  \int_{\half |p|^2}^{\frac12(1+|p|^2)} \frac{1}{(\lambda+\rho)^2(\lambda + \rho + 1)} d\rho \Big) \lesssim \coup^2\,. 
\end{equ}
At last, we extend the integral to the interval $[\tfrac12|p|^2,\infty)$, which is possible since 
\begin{equ}
\coup^2  \int_{\frac12(1+|p|^2)}^\infty \frac{\dd \rho}{(\lambda+\rho)(\lambda + \rho + 1)[1+g(\ell^\lambda(\rho))]} \leq \gamma^2\int_{\half}\frac{\dd\rho}{\rho^2}\lesssim \gamma^2\,.
\end{equ}

Upon performing the additional change of variables $y\eqdef \ell^\lambda(\rho)$, we have shown 
\begin{equ}
\sup_p\Big|\cm^\lambda(p)-2\pi   \int_{0}^{\ell^\lambda(\tfrac12|p|^2)} \frac{\dd y}{1+g(y)} \Big|\lesssim\gamma^2
\end{equ}
so that the conclusion follows since $g$ satisfies $\partial g=2\pi/(1+g)$. 
\end{proof}

\subsection{Estimates for the nuisance region}\label{app:EstNregion}

The goal of this subsection is to complete the proof of Lemma~\ref{lem:nuisance} for $\sigma=+$ and 
show that the function $\calI$ in~\eqref{e:IBoundFour} on $\R^{2(n-1)}$ has bounded sup-norm. 
Let us recall its definition
\begin{equs}[e:IApp]
\calI(p_{3:n+1}) \eqdef& \coup^2\int \dd p_{1:2} \frac{\hat V(p_1)\hat V(p_2)}{|p_1||p_2||p_{[3:n+1]}|} \frac{N(p_1,p_{[2:n+1]}) B(p_2, p_{[3:n+1]})}{(\lambda +|p_{[1:n+1]}|^2)(\lambda +|p_2|^2 +|p_{[3:n+1]}|^2)^\half}\\
&\times \Big|\frac{ (p_1\cdot p_{[2:n+1]})(p_2\cdot p_{[3:n+1]})}{\lambda +\half |p_{[2:n+1]}|^2[1+\cg^\lambda(p_{[2:n+1]})]}\\
&\qquad+\frac{N(p_2,p_{[1:n+1\setminus 2]}) B(p_1, p_{[3:n+1]}) (p_2\cdot p_{[1:n+1\setminus2]})(p_1\cdot p_{[3:n+1]})}{\lambda +\half |p_{[1:n+1\setminus 2]}|^2[1+\cg^\lambda(p_{[1:n+1\setminus 2]})]}\Big| 
\end{equs}
where $p_{3:n+1}\in\R^{2(n-1)}$, $\lambda\in(0,1)$ and $\cg^\lambda$ is given by~\eqref{eq:cglambda}. 
Note that we removed the product of nuisance and bulk region in the second line since this is already present in the 
first. 

\begin{lemma}
\label{lem:35}
Let $\calI$ be defined according to~\eqref{e:IApp}. Then, uniformly over $\lambda\in(0,1)$, we have 
\begin{equ}
\sup_{p_{3:n+1}\in\R^{2(n-1)}} \calI(p_{3:n+1}) \lesssim 1\,.
\end{equ}
\end{lemma}
\begin{proof}
As in the proof of Lemma~\ref{lem:ReplEst}, we will massage the expression in~\eqref{e:IApp} 
so to ultimately see a cancellation between the terms in the absolute value. Our first goal is to show that the difference 
between $\calI$ and $\calI_{\rm fin}$ given by
\begin{equation}
\label{e:IFApp}
\begin{split}
&\calI_{\rm fin}(p_{3:n+1}) \eqdef \coup^2\int \dd p_{1:2} \frac{\hat V(p_1)\hat V(p_2)}{|p_1||p_2||p_{[3:n+1]}|}\times \\
& \frac{N(p_1,p_{[2:n+1]}) B\cap G(p_2, p_{[3:n+1]})}{(\lambda +|p_{[1:n+1]}|^2)(\lambda +|p_2|^2 +|p_{[3:n+1]}|^2)^\half} \Big|\frac{ (p_1\cdot p_{2})([p_1+p_2]\cdot p_{[3:n+1]})}{\lambda +\half |p_{[2:n+1]}|^2[1+\cg^\lambda(p_{[2:n+1]})]}\Big| 
\end{split}
\end{equation}
is order $1$, where $B\cap G(p_2, p_{[3:n+1]})\eqdef B(p_2, p_{[3:n+1]})G(p_2, p_{[3:n+1]})$ 
and the second factor 
is the indicator function on $(\R^2)^2$ of the set $G\eqdef \{(p,q)\colon |q|\leq \tfrac{1}{16}|p|\}\subset (\R^2)^2$. 
This will be achieved by subsequently (inserting or) replacing certain terms and controlling the error 
made at each step. 

Recall the definition of the bulk $B=B^{\frac13}$ and nuisance $N=N^{\frac13}$ regions from~\eqref{e:BNregions} 
and notice that, by triangle inequality, for any $i,j\in\{1,2\}$, $i\neq j$, we have 
\begin{equ}[e:Triangle1]
N(p_i,p_{[1:n+1\setminus i]}) B(p_j, p_{[1:n+1\setminus i,j]})\leq \1_{\{\frac16|p_j|<|p_i|\}}\,.
\end{equ}

\noindent {\it Step 1.} We want to insert the function $G(p_2,p_{[3:n+1]})$ in the first line 
at the right hand side of~\eqref{e:IApp}. Bounding each summand in the absolute value separately, 
using that $\cg^\lambda$ is non-negative and the bulk region in each case, 
we see that the difference can be controlled by 
\begin{equs}\label{e:Step1}
 &\coup^2\int \dd p_{1:2} \hat V(p_1)\hat V(p_2) \frac{N(p_1,p_{[2:n+1]}) B(p_2, p_{[3:n+1]}) (1-G(p_2, p_{[3:n+1]}))}{(\lambda +|p_{[1:n+1]}|^2)(\lambda +|p_2|^2 +|p_{[3:n+1]}|^2)^\half}\\
&\qquad\times \Big(\frac{ |p_{[2:n+1]}|}{\lambda +|p_2|^2 +|p_{[3:n+1]}|^2}+\frac{N(p_2,p_{[1:n+1\setminus 2]}) B(p_1, p_{[3:n+1]}) |p_{[1:n+1\setminus2]}|}{\lambda +|p_1|^2+|p_{[3:n+1]}|^2}\Big)\,.
\end{equs}
We argue that thanks to the regions under consideration, both of the terms $|p_{[2:n+1]}|, |p_{[1:n+1\setminus2]}|$ appearing in the the second line in \eqref{e:Step1} can be bounded by a constant times $|p_{[3:n+1]}|$. For the first term, we have $1-G(p_2, p_{[3:n+1]}) =  \1_{\{\tfrac1{16}|p_2|<|p_{[3:n+1]}|\}}$, therefore $|p_{[2:n+1]}|\leq |p_2|+|p_{[3:n+1]}|\lesssim |p_{[3:n+1]}|$. For the second, by~\eqref{e:Triangle1}, we deduce that 
\begin{equ}
\prod_{(i,j)=(1,2), (2,1)}N(p_i,p_{[1:n+1\setminus i]}) B(p_j, p_{[3:n+1]})\leq \1_{\{\frac16|p_1|<|p_2|< 6 |p_1|\}}
\end{equ}
and therefore, we also have $|p_{[1:n+1\setminus2]}|\leq |p_1|+|p_{[3:n+1]}|\lesssim |p_{[3:n+1]}|$. 
As a consequence,~\eqref{e:Step1} is bounded above by 
\begin{equs}[e:Bound32]
\coup^2|p_{[3:n+1]}|\int \frac{\hat V(p_1)\hat V(p_2)\dd p_{1:2}  }{(\lambda +|p_{[1:n+1]}|^2)(\lambda +|p_2|^2 +|p_{[3:n+1]}|^2)^{\frac32}}\lesssim 1
\end{equs}
where, in the last step, we first integrated in $p_1$ and then applied Lemma~\ref{lem:4} to the integral in $p_2$. 

Before proceeding, let us point out that the insertion of the indicator function of $G$ ensures that 
\begin{equ}[e:Triangle2]
N(p_1,p_{[2:n+1]}) B\cap G(p_2, p_{[3:n+1]}) \leq N^{\frac5{12}}(p_1,p_2)
\end{equ}
so that, in particular, $p_1$ and $p_2$ are comparable.  
\medskip

\noindent {\it Step 2.} We want to replace both scalar products $p_1\cdot p_{[2:n+1]}$ and 
$p_2\cdot p_{[1:n+1\setminus2]}$, with $p_1\cdot p_2$. To do so, we add and subtract 
the corresponding terms inside the absolute value and use triangle inequality. By~\eqref{e:Triangle2}, 
it is not hard to see that the error terms are bounded by~\eqref{e:Bound32}. 
\medskip

\noindent {\it Step 3.} In this and the next step, we focus on the second summand in the absolute value in~\eqref{e:IApp}. 
At first we want to replace its denominator with the denominator of the first summand, i.e. with 
$\lambda +\half |p_{[2:n+1]}|^2[1+\cg^\lambda(p_{[2:n+1]})]$, for which we add and subtract the corresponding term in the 
absolute value and bound the error. To do so, let $f$ be the map on $\R^2$ given by 
$f(p)\eqdef (\lambda +\half|p|^2[1+\cg^\lambda(p)])^{-1}$ and note that $| \nabla f(p)|\lesssim (\lambda +|p|^2)^{3/2}$. 
Hence, by mean value theorem, we get 
\begin{equs}
\Big|f(p_{[1:n+1\setminus 2]})-f(-p_{[2:n+1]})\Big|\lesssim |p_{[1:n+1\setminus 2]}+p_{[2:n+1]}|
\sup_{p\in I}\frac{1}{(\lambda+|p|^2)^{\frac32}}
\end{equs}
where $I\subset\R^2$ is the segment connecting $p_{[1:n+1\setminus 2]}$ and 
$-p_{[2:n+1]}$. Thanks to the presence of the indicator function $B(p_2,p_{[3:n+1]})$ we know that 
$|p_{[2:n+1]}|^2\gtrsim |p_2|^2+|p_{[3:n+1]}|^2$, while $B(p_1,p_{[3:n+1]})$ and~\eqref{e:Triangle2} give 
$|p_{[1:n+1\setminus 2]}|\gtrsim |p_1|^2+|p_{[3:n+1]}|^2\gtrsim |p_2|^2+|p_{[3:n+1]}|^2$. 
As a consequence, for any $p\in I$, we have $|p|^2\gtrsim |p_2|^2+|p_{[3:n+1]}|^2$, which 
ultimately delivers a bound on the error term of the form
\begin{equs}
 &\coup^2\int  \dd p_{1:2} \frac{\hat V(p_1)\hat V(p_2) N(p_1,p_{[2:n+1]}) B\cap G(p_2, p_{[3:n+1]})|p_1||p_{[1:n+1\setminus 2]}+p_{[2:n+1]}|}{(\lambda +|p_{[1:n+1]}|^2)(\lambda +|p_2|^2 +|p_{[3:n+1]}|^2)^2}\\
 &\qquad\qquad\qquad\qquad\lesssim \coup^2\int  \dd p_{1:2} \frac{\hat V(p_1)\hat V(p_2) N^{\frac5{12}}(p_1,p_2)(|p_1+p_2|+|p_{[3:n+1]}|)}{(\lambda +|p_{[1:n+1]}|^2)(\lambda +|p_2|^2 +|p_{[3:n+1]}|^2)^{3/2}}
\end{equs}
where we used~\eqref{e:Triangle2} once again. At this point we estimate the term containing $|p_1+p_2|$ 
with Lemma~\ref{lem:10}, while that with $|p_{3:n+1}|$ with Lemma~\ref{lem:4}. 
\medskip

\noindent {\it Step 4.} We want to remove the indicator functions $B(p_1, p_{[3:n+1]})$, $N(p_2,p_{[1:n+1\setminus 2]})$  
from the second summand, starting with the former. Notice that in both cases, we need to estimate a quantity of the form 
\begin{equs}[e:BoundStep4]
\int \dd p_{1:2}  \,\frac{\hat V(p_1)\hat V(p_2) N(p_1,p_{[2:n+1]}) B\cap G(p_2, p_{[3:n+1]}) \BF  |p_1|}{(\lambda +|p_{[1:n+1]}|^2)(\lambda +|p_2|^2 +|p_{[3:n+1]}|^2)^{\frac32}}
\end{equs}
where, to first remove $B(p_1, p_{[3:n+1]})$, we take $\BF=\BF_1=N(p_2,p_{[1:n+1\setminus 2]})$ $ N(p_1, p_{[3:n+1]})$, 
while to then remove $N(p_2,p_{[1:n+1\setminus 2]})$, $\BF=\BF_2=B(p_2,p_{[1:n+1\setminus 2]})$. 
In the first case, if $\BF_1=1$ then $N(p_1, p_{[3:n+1]})=1$, which implies that $|p_1|\lesssim |p_{[3:n+1]}|$ 
(see~\eqref{e:BNregionsBound}). Hence,~\eqref{e:BoundStep4} is bounded above by~\eqref{e:Bound32}.  

On the other hand, for $\BF_2\,N(p_1,p_{[2:n+1]}) B\cap G(p_2, p_{[3:n+1]})=1$, 
both $\BF_2$ and $G(p_2, p_{[3:n+1]})$ must be equal to $1$, therefore 
$\frac13|p_2|\leq |p_{[1:n+1]}|\leq |p_1+p_2|+|p_{[3:n+1]}|\leq |p_1+p_2|+\tfrac{1}{16}|p_2|$,  
so that $|p_2|\lesssim |p_1+p_2|$ and by~\eqref{e:Triangle2}, $|p_1|\lesssim |p_1+p_2|$. 
As a consequence,~\eqref{e:BoundStep4} can be controlled via Lemma~\ref{lem:10}.   
\medskip

\noindent{\it Conclusion. }Thanks to Steps 1-4, we showed that the difference between $\calI$ in~\eqref{e:IApp} and 
$\calI_{\rm fin}$~\eqref{e:IFApp} is uniformly bounded, so that to conclude, it suffices to 
show that the latter is also uniformly bounded. 
For this, we use first the indicator function of the bulk region $B(p_2,p_{[3:n+1]})$, and 
then~\eqref{e:Triangle2}, to get 
\begin{equs}
|\calI_{\rm fin}(p_{3:n+1})|\lesssim \gamma^2\int \dd p_{1:2} \frac{\hat V(p_1)\hat V(p_2) N^{\frac5{12}}(p_1,p_2) |p_1+p_2|}{(\lambda +|p_{[1:n+1]}|^2)(\lambda +|p_2|^2 +|p_{[3:n+1]}|^2)^{\frac32}}\lesssim 1
\end{equs}
where the last step follows by Lemma~\ref{lem:10}. Hence, the proof is concluded. 
\end{proof}

\section{Triangular martingale central limit theorem}
\label{sec:mart-centr-limit}

The following is a simple adaptation of the standard martingale central limit theorem to the case where the sequence 
of martingales depends on the scaling parameter (and therefore we refer to it as `triangular'). 
In the statement, to distinguish between microscopic and macroscopic scales, we use parentheses in the superscript. 

\begin{theorem}
\label{thm:clt}
Let $\{(\calM^{(\varepsilon)}_t)_{t \ge 0}\colon \eps\in(0,1)\}$ be a family of real valued mean zero 
martingales with stationary increments 
on a common filtered probability space $(\bm{\Omega}, \bm{\calF}, \bm{\calF}_t, \bfP)$. 
Let $\varsigma^2_\varepsilon \eqdef \bfE[ \langle \calM^{(\varepsilon)}\rangle_1]$ and suppose that for all $t \ge 0$
\begin{align}
\limsup_{\varepsilon \rightarrow 0} \varsigma^2_\varepsilon &\lesssim 1\,, \label{eq:97} \\
\limsup_{\varepsilon \rightarrow 0} \bfE[ \langle \calM^{(\varepsilon)}\rangle_1^2] &\lesssim 1\,, \label{eq:82} \\
\lim_{\varepsilon \rightarrow 0} \sup_{s \in [0, t/\varepsilon^2]} \varepsilon^4 \Var \langle\calM^{(\varepsilon)}\rangle_s &= 0\,. \label{eq:78} 
\end{align}
For $\varepsilon \in (0,1)$, define the scaled family of martingales $\{(\calM^{\varepsilon}_t)_{t \ge 0}\colon \eps\in(0,1)\}$ 
according to the diffusive rescaling $\calM^\varepsilon_t \eqdef \varepsilon\calM^{(\varepsilon)}_{t/\varepsilon^2}$. 
Then for all $\theta \in \bbR$ it holds that
\begin{equation}
\label{eq:99}
\lim_{\varepsilon \rightarrow 0} \bfE | \bfE \big[ e^{\iota \theta \calM^\varepsilon_t} \mid \bm{\calF}_0 \big] - e^{-\half \theta^2 \varsigma_\varepsilon^2 t} | \rightarrow 0
\end{equation}
\end{theorem}

Our proof follows closely that of \cite[Theorem 2.1]{KomoLandOlla12_FluctuationsMarkov}, 
but the dependence on $\eps$ of $\calM^{(\varepsilon)}$ makes their result not directly applicable so, 
for completeness we detail the argument below. 
For simplicity, we ask for a uniform control over the variance of the quadratic variation~\eqref{eq:78}, 
which is easy to get in the context of the SRBP, but could be weakened. 
Let us also point out that we are not assuming 
\begin{equation}
\label{eq:150}
\varsigma^2 = \lim_{\varepsilon \rightarrow 0} \varsigma^2_\varepsilon
\end{equation}
so that, strictly speaking, the above is not a semi quenched central limit theorem but it would be 
upon replacing~\eqref{eq:97} with~\eqref{eq:150}. 

\begin{proof}
To lighten the notation, let $\bfE_0[\cdot]$ denote the conditional expectation $\bfE[\cdot\mid \bm{\calF}_0]$. Choose $\varepsilon\in(0,1)$ small enough so that \eqref{eq:97}-\eqref{eq:82} hold. Set $\beta_\varepsilon = \half \theta^2 \varsigma_\varepsilon^2$, which is uniformly bounded by~\eqref{eq:97}. It is sufficient to show
\begin{equation}
\label{eq:98}
\lim_{\varepsilon \rightarrow 0} \bfE \big\lvert e^{\beta_\varepsilon}\bfE_0 \big[ e^{\iota \theta\varepsilon \calM^{(\varepsilon)}_{1/\varepsilon^2}} \big] - 1 \big\rvert = 0
\end{equation}
Indeed,~\eqref{eq:99} follows from applying Equation \eqref{eq:98} in the case $\calM^{(\varepsilon)} := \calM^{(\varepsilon\sqrt{t})}, \theta := \theta\sqrt{t}$ and changing variables $\varepsilon := \varepsilon / \sqrt{t}$.

Note that, letting $N\eqdef \lfloor 1/\varepsilon^2 \rfloor$, the previous condition is implied 
by the analogous condition in which the martingale $\calM^{(\varepsilon)}$ is evaluated at time $N\in\N$ instead of 
$1/\eps^2$. Indeed, since for any $a,b\in\R$, $|e^{\iota a}-e^{\iota b}|\leq |a-b|$ and the martingales are mean $0$, we have 
\begin{equ}
\bfE \Big\lvert e^{\beta_\varepsilon} \bfE_0 \Big[ e^{\iota \theta \eps\calM^{(\varepsilon)}_{1/\varepsilon^2}} \Big] - 1 \Big\rvert\lesssim \bfE \Big\lvert e^{\beta_\varepsilon}\bfE_0 \Big[ e^{\iota \theta \eps\calM^{(\varepsilon)}_{N}} \Big] - 1 \Big\rvert+\theta^2\eps^2 e^{\beta_\varepsilon}\bfE[(\calM^{(\varepsilon)}_{1/\varepsilon^2}-\calM^{(\varepsilon)}_{N})^2]
\end{equ}
and $\bfE[\langle \calM^{(\varepsilon)}\rangle_{1/\varepsilon^2} - \langle \calM^{(\varepsilon)}\rangle_N] = (\eps^{-2}-N)\varsigma_\varepsilon^2\lesssim 1$ so that the second summand vanishes. 
We now rewrite the expectation as a telescopic sum. To do so, for $j=0,\dots, N-1$ set 
\begin{equs}
\calX_j&\eqdef e^{\beta_\varepsilon\frac{j+1}{N}} \bfE_0\Big[e^{\iota \theta\varepsilon \calM^{(\varepsilon)}_{j+1}}\Big]-e^{\beta_\varepsilon\frac{j}{N}} \bfE_0\Big[e^{\iota \theta\varepsilon \calM^{(\varepsilon)}_j}\Big]\,,\quad\text{and}\quad\calZ^{(\eps)}_j&\eqdef  \calM^{(\varepsilon)}_{j+1}-\calM^{(\varepsilon)}_{j}\,,
\end{equs} 
so that 
\begin{equs}[e:sum4]
\bfE \Big\lvert e^{\beta_\varepsilon} \bfE_0 \Big[ e^{\iota \theta \eps\calM^{(\varepsilon)}_N} \Big] - 1 \Big\rvert=\bfE \Big\lvert \sum_{j=1}^{N-1}  \calX_j\Big\rvert\leq\sum_{i=1}^4\bfE \Big\lvert \sum_{j=1}^{N-1}  \calX^{(i)}_j\Big\vert
\end{equs}
where
\begin{align*}
\calX^{(1)}_j &= e^{\beta_\varepsilon\frac{j+1}{N}} \big(1-\beta_\varepsilon \tfrac{1}{N}-e^{-\beta_\varepsilon\frac{1}{N}}\big) \bfE_0\Big[e^{\iota \theta\varepsilon \calM^{(\varepsilon)}_j}\Big]\,, \\
\calX^{(2)}_j &= e^{\beta_\varepsilon\frac{j+1}{N}} \big(\beta_\varepsilon \tfrac{1}{N}-\beta_\varepsilon\varepsilon^2\big) \bfE_0\Big[e^{\iota \theta\varepsilon \calM^{(\varepsilon)}_j}\Big]\,, \\
\calX^{(3)}_j &= e^{\beta_\varepsilon\frac{j+1}{N}} \bfE_0\Big[r\big(\theta\varepsilon \calZ^{(\varepsilon)}_{j}\big)e^{\iota \theta\varepsilon \calM^{(\varepsilon)}_j}\Big]\,, \\
\calX^{(4)}_j &= e^{\beta_\varepsilon\frac{j+1}{N}} \bfE_0\Big[\Big(-\tfrac{\theta^2\varepsilon^2 }{2}(\calZ^{(\varepsilon)}_{j})^2 + \beta_\eps\varepsilon^2 \Big) e^{\iota \theta\varepsilon \calM^{(\varepsilon)}_j}\Big]\,.
\end{align*}
in which we used that  $e^{\iota a} = e^{\iota b}+\iota (a-b)-(a-b)^2/2+r(a-b)$, for $r$ such that $|r(x)| \lesssim |x|^3$, 
and any $a, b \in \bbR$. 
Now, in the three summands at the right hand side of~\eqref{e:sum4} corresponding to $i=1,2,3$, 
we bound the complex exponential by $1$, and respectively use that $|1-x-e^{-x}|\lesssim x^2$, 
$|N^{-1}-\eps^2|\lesssim N^{-2}$ and Burkholder-Davis-Gundy inequality which gives 
$\bfE|Z^{(\varepsilon)}_j|^3 = \bfE|Z^{(\varepsilon)}_1|^3 \lesssim \bfE[ \langle \calM^{(\varepsilon)}\rangle_1^{3/2}] \lesssim 1$, the last step being due to~\eqref{eq:82}. 
Therefore, we obtain
\begin{equs}
\sum_{i=1}^3\bfE \Big\lvert \sum_{j=1}^{N-1}  \calX^{(i)}_j\Big\vert \lesssim \sum_{i=1}^{N-1}\Big(2\frac{\beta_\eps^2}{N^2}+\theta^3\eps^3\Big)\lesssim \frac{1}{N}+\eps^3N
\end{equs}
and the right hand side converges to $0$ as $\eps\to 0$ since $N= \lfloor 1/\varepsilon^2 \rfloor$.  

At last we turn to the fourth summand in~\eqref{e:sum4}. Let $V_\eps\eqdef \sup_{s\leq \eps^{-2}} \eps^4 \Var \langle\calM^{(\varepsilon)}\rangle_s$, which, by~\eqref{eq:78}, vanishes, and, for $\alpha<\half$, 
define $K\eqdef \lfloor V^\alpha_\varepsilon N \rfloor$. 
Partition $\{0,...,N-1\}$ into $L\eqdef \lfloor N/K\rfloor$ groups $I_k$ of $K$ or 
$K+1$ consecutive numbers in such a way 
that $\{0,...,N-1\}= \cup_{k=1}^L I_k$.   

Upon defining 
$\calY^{(\eps)}_j\eqdef  \langle\calM^{(\varepsilon)}\rangle_{j+1}-\langle\calM^{(\varepsilon)}\rangle_{j}$, 
we have $\bfE[(\calZ^{(\varepsilon)}_{j})^2 \mid \bm{\calF}_j] = \bfE[\calY^{(\varepsilon)}_{j} \mid \bm{\calF}_j]$, 
so that the quantity we need to bound is 
\begin{equation}
\label{eq:102}
\bfE\Big\lvert\sum_{j=0}^{N-1} \calX^{(4)}_j\Big\lvert = \tfrac{\theta^2}{2}\bfE\Big\lvert\varepsilon^2\sum_{k=1}^\ell\sum_{j \in I_k} e^{\beta_\varepsilon\frac{j+1}{N}} \bfE_0\Big[\Big( \calY^{(\varepsilon)}_{j} - \varsigma_\varepsilon^2 \Big) e^{\iota \theta\varepsilon \calM^{(\varepsilon)}_j}\Big]\Big\lvert
\end{equation}
where we also used that $\beta_\varepsilon = \half \theta^2 \varsigma_\varepsilon^2$. 
Let $j_k\eqdef\min I_k$. We want to replace the $j$'s appearing at the exponentials with $j_k$. To do so, 
notice that 
\begin{equs}[e:ReplBoundCLT]
\Big|e^{\beta_\varepsilon\frac{j+1}{N}+\iota \theta\varepsilon \calM^{(\varepsilon)}_j}-&e^{\beta_\varepsilon\frac{j_k+1}{N}+\iota \theta\varepsilon \calM^{(\varepsilon)}_{j_k}}\Big| \lesssim \beta_\eps \frac{|j-j_k|}{N}+\eps|\calM^{(\varepsilon)}_{j}-\calM^{(\varepsilon)}_{j_k}|\\
&\lesssim \beta_\eps \frac{K}{N}+\eps|\calM^{(\varepsilon)}_{j}-\calM^{(\varepsilon)}_{j_k}|\lesssim \beta_\eps V_\eps^\alpha +\eps|\calM^{(\varepsilon)}_{j}-\calM^{(\varepsilon)}_{j_k}|\,.
\end{equs}
As a consequence,~\eqref{eq:102} is bounded above by 
\begin{equs}[e:Almost]
\frac{\theta^2}{2}\bfE\Big\lvert\varepsilon^2\sum_{k=1}^L e^{\beta_\varepsilon\frac{j_k+1}{N}} \bfE_0\Big[e^{\iota \theta\varepsilon \calM^{(\varepsilon)}_{j_k}}\sum_{j \in I_k} \Big( \calY^{(\varepsilon)}_{j} - \varsigma_\varepsilon^2 \Big) \Big]\Big\lvert +R_\eps\,,
\end{equs}
where, thanks to~\eqref{e:ReplBoundCLT}, $R_\eps$ is controlled by 
\begin{align*}
{} &\eps^2\sum_{k=1}^L\sum_{j \in I_k}\bfE \Big(|\calY^{(\varepsilon)}_{j} - \varsigma_\varepsilon^2| (\beta_\eps V_\eps^\alpha +\eps|\calM^{(\varepsilon)}_{j}-\calM^{(\varepsilon)}_{j_k}|)\Big)\\
&\lesssim \eps^2\sum_{k=1}^L\sum_{j \in I_k}\bfE [|\calY^{(\varepsilon)}_{j} - \varsigma_\varepsilon^2|^2]^{\half}\Big(\beta_\eps V_\eps^\alpha +\eps   \bfE[(\calM^{(\varepsilon)}_j - \calM^{(\varepsilon)}_{j_k})^2]^\half\Big) \\
&\lesssim V_\eps^\alpha+ \eps \sqrt{K}\lesssim V_\eps^\alpha+ V_\eps^{\frac\alpha2}
\end{align*}
and we used that the expectation involving $\calY^{(\varepsilon)}_{j}$ is bounded thanks to~\cref{eq:97,eq:82},  
stationary of the increments of $\calM^{(\eps)}$ to control the corresponding term, and the definition of $K$ and $N$. 
Since $R_\eps$ vanishes, we are left to consider the first summand in~\eqref{e:Almost}, which equals
\begin{equs}
\frac{\theta^2}{2}\bfE\Big\lvert\varepsilon^2\sum_{k=1}^L e^{\beta_\varepsilon\frac{j_k+1}{N}} \bfE_0\Big[e^{\iota \theta\varepsilon \calM^{(\varepsilon)}_{j_k}}\Big( \langle\calM^{(\varepsilon)}\rangle_{K} - K\varsigma_\varepsilon^2 \Big) \Big]\Big\lvert\lesssim  L V^{\half}_\eps \lesssim V^{\half-\alpha}_\eps
\end{equs}
where we applied Cauchy-Schwarz and the definition of $V_\eps$. Recall that we chose $\alpha<1/2$ and 
$V_\eps\to 0$, hence the right hand side goes to $0$ and the proof of the theorem is concluded. 
\end{proof}

\section{Well posedness of the martingale problem}
\label{sec:well-posedn-mart}

Let $H$ be the weighted Sobolev space in~\eqref{e:H} and $H^*$ its dual.  
For $z\in\R^2$, we define the translation map $\tau_z \colon H \rightarrow H$ according to $\tau_z\phi(x) \eqdef \phi(x+z)$, 
and its dual $T_z$ which acts on $H^\ast$. 
The next lemma states some easy properties of the translation map.

\begin{lemma}
\label{lem:31}
For $z\in\bbR^2$, we have the operator bound 
\begin{equation}
\label{eq:164}
\lVert T_z \rVert^2 \lesssim 1 + |z|^4
\end{equation}
Moreover, for fixed $\bfh \in H^*$, the map $T_\cdot \bfh : \bbR^2 \rightarrow H^*$ is Lipschitz-continuous.
\end{lemma}

\begin{proof}
For~\eqref{eq:164}, by duality it suffices to show $\lVert \tau_{-z} \rVert^2 \lesssim 1 + |z|^4$, which in turn is an 
immediate consequence of the weights in the definition of the norm on $H$ (see~\eqref{e:H}).  
To show that $T_\cdot \bfh$ is Lipschitz-continuous, a density argument together with the operator bound~\eqref{eq:164} 
ensure that we can reduce to the case in which there is $h\in\calS(\R^2,\R^2)$ such that $\bfh=\langle h, \cdot\rangle_H$. 
Then, 
\begin{equs}
\lVert T_y \bfh - T_z \bfh \rVert^2 &\le \int (1+|x|^2)^{-2} | (1-\Delta)^{-25} (\tau_y h - \tau_z h)(x) |^2 \dd x\lesssim_h |z-y|^2
\end{equs}
where we used that $\tau_y$ commutes with $(1-\Delta)^{-25}$ and mean-value theorem. Hence, the proof is concluded. 
\end{proof}

We now turn to the main object of this appendix and begin by defining what it means to be an (analytically weak) 
solution of the stochastic linear transport equation in~\eqref{eq:131}. 

\begin{definition}
\label{def:1}
Let $(E,\calE,\bbP)$ be a probability space with a normal filtration $(\calE_t)_{t\ge0}$ carrying a Brownian motion 
$(B_t)_{t\ge0}$ and $\bm{\pi}$ be an arbitrary measure on $H^*$. 
For fixed $T>0$, we say that $(\bm{\eta}_t)_{t\ge0}$ is an (analytically weak) solution of the stochastic linear transport equation (SLTE) in~\eqref{eq:131} with diffusivity $\varsigma^2>0$ and initial distribution $\bm{\pi}$ 
if $(\bm{\eta}_t)_{t\ge0}$ is a continuous adapted $H^*$-valued process such that 
$\bm{\eta}_0$ is independent from $B$ and has law $\bm{\pi}$, and for all 
$g \in \calS(\bbR^2, \bbR^2)$, $\bbP$-a.s. for all $t \in [0,T]$
\begin{equation}
\label{eq:165}
\bm{\eta}_t[g] - \bm{\eta}_0[g] = \varsigma \int_0^t \bm{\eta}_s[-\nabla g] \dd B^i_s + \half \varsigma^2 \int_0^t  \bm{\eta}_s[\Delta g] \dd s\,.
\end{equation}
\end{definition}

We split existence and uniqueness for SLTE in the following two lemmas. 

\begin{lemma}
\label{lem:29}
Let $\bm{\pi}$ be a probability measure on $H^*$ for which $\int \lVert \bfh \rVert_{H^*}^2 \bm{\pi}(d \bfh) < \infty$. 
Then, for all $\varsigma^2>0$ and $T > 0$, there exists a solution to SLTE with diffusivity $\varsigma^2>0$ 
and initial distribution $\bm{\pi}$ given according to Definition \ref{def:1}.
\end{lemma}

\begin{proof}
Let $(E,\calE,\bbP)$ be an arbitrary probability space supporting $\bm{\eta}_0$
with law $\bm{\pi}$, and an independent Brownian motion $(B_t)_{t\ge0}$, and $(\calE_t)_{t\ge0}$ be the usual 
augmented filtration. 
Fix $T>0$ and let $(\bm{\eta}_t)_{t \in [0,T]}$ be the $H^*$-valued process defined by 
\begin{equ}
\bm{\eta}_t \eqdef T_{\varsigma B_t}\bm{\eta}_0\,.
\end{equ}
Our goal is to show that $\bm{\eta}$ solves SLTE. 
Clearly, $(\bm{\eta}_t)_{t\in[0,T]}$ is adapted, continuous (by Lemma \ref{lem:31} and the continuity of Brownian motion) 
and $\bm{\eta}_0$ is distributed according to $\bm{\pi}$. It remains to verify~\eqref{eq:165}. 

Let $g \in \calS(\bbR^2, \bbR^2)$, and note that, by definition, $\bm{\eta}_t[g] = \bm{\eta}_0[\tau_{-\varsigma B_t} g]$. 
It\^o's formula gives that for each $x\in\R^2$, up to $\bbP$-indistinguishability, 
for all $t\in[0,T]$ we have  
\begin{equ}[e:Itog]
\tau_{-\varsigma B_t} g(x) = g(x) - \varsigma\int_0^t \nabla g(x-\varsigma B_s) \cdot \dd B_s + \half\varsigma^2\int_0^t\Delta g(x-\varsigma B_s) \dd s\,.
\end{equ}
Let us now consider a regularised version of $\bm{\eta}_0$, given by 
$\eta_0^{(m)}(x)\eqdef \sum_{i=1}^m  \bm{\eta}_0[ g^i] g^i(x)$, for $(g^i)_{i\in\N}$ the basis of $H$ 
introduced at the beginning of Section~\ref{sec:Env} and $m\in\N$. Denote by $\bm{\eta}_0^{(m)}$ the 
element in $H^\ast$ corresponding to $\eta^{(m)}$. 
We multiply both sides of~\eqref{e:Itog} by $\eta_0^{(m)}(x)$, which is $\calE_0$-measurable, 
and integrate in $x$. To swap the integral in $x$ and the stochastic integral, we invoke 
stochastic version of Fubini's theorem \cite[Theorem 4.33]{DaPratoZabczyk14_StochasticEquations}, 
which is applicable since
\begin{equs}
\int \Big( \bbE\Big[\int_0^T |\eta_0^{(m)}(x)\nabla g(x-\varsigma B_t)|^2\dd t\Big] \Big)^\half \dd x\lesssim_{T,m, g} \Big(\int\|\bfh\|^2_{H^\ast}\bm{\pi}(\dd \bfh)\Big)^\half<\infty
\end{equs}
where we used that we chose the elements of the basis of $H$ to decay at $\infty$ faster than any polynomial. 
Therefore we conclude that, up to $\bbP$-indistinguishability, for all $t \in [0,T]$
\begin{equ}
T_{\varsigma B_t} \bm{\eta_0}^{(m)}[g] - \bm{\eta_0}^{(m)}[g] = - \varsigma \int_0^t T_{\varsigma B_s} \bm{\eta_0}^{(m)}[\nabla g] \cdot \dd B_s + \frac{\varsigma^2}{2}\int_0^t T_{\varsigma B_s}\bm{\eta_0}^{(m)}[\Delta g] \dd s.
\end{equ}
Since by dominated convergence, $\bfE \lVert \bm{\eta_0} - \bm{\eta_0}^{(m)}\rVert^2$ converges to $0$ as $m\to\infty$, 
Doob's and Jensen's inequalities and~\eqref{eq:164} ensure that we can take the limit as $m\to\infty$ in 
the previous equality, thus deducing~\eqref{eq:165} and therefore concluding the proof. 
\end{proof}

\begin{lemma}
\label{lem:12}
Fix an initial distribution $\bm{\pi}$ on $H^\ast$, and $\varsigma^2,\,T>0$, and 
let $(\bm{\eta}_t)_{t\ge0}$ be a solution to SLTE as given in Definition \ref{def:1}. Then, 
$(\bm{\eta}_t)_{t \in [0,T]}$ satisfies $\bm{\eta}_t = T_{\varsigma B_t}\bm{\eta}_0$. 
In particular, solutions to SLTE are unique in law on $C_TH^\ast$.
\end{lemma}

\begin{proof}
By continuity, it is enough to show that for fixed $t \in [0,T]$, $T_{-\varsigma B_t}\bm{\eta}_t = \bm{\eta}_0$ $\bbP$-a.s.. 
It can be easily checked that this is equivalent to show that 
\begin{equation}
\label{eq:18}
\eta^\delta_t(x-\varsigma B_t) = \eta^\delta_0(x)
\end{equation}
where we set $\eta^\delta(x)\eqdef \varrho^\delta\ast\bm{\eta}(x)=\bm{\eta}[\varrho^{\delta}_x]$ for 
$\varrho\in\calS(\R^2,\R^2)$ whose components are non-negative, compactly supported and of mass one, 
and $\varrho_x^\delta(y)=\delta^{-2} \varrho((y-x)/\delta)$. 
Using~\eqref{eq:165} with $g = \varrho^{\delta}_x$, and integration by parts, we deduce that, for $x \in \bbR^2$,
\begin{equ}
\eta^\delta_t(x) = \eta^\delta_0(x) + \varsigma  \int_0^t \nabla \eta^\delta_s(x) \cdot \dd B_s + \half \varsigma^2 \int_0^t  \Delta\eta^\delta_s(x) \dd s\,.
\end{equ}
Since $\eta^\delta(x)$ solves the above, we can apply the generalised It\^{o}'s formula, 
\cite[Theorem 3.3.1]{Kunita90_StochasticFlows}, to $t\mapsto \eta^\delta_t(x-\varsigma B_t)$ 
and immediately see that it agrees with the right hand side of~\eqref{eq:18}. 
\end{proof}

Finally, we may put these pieces together and prove Proposition \ref{prp:1}. 

\begin{proof}[of Proposition \ref{prp:1}]
Existence follows easily by Lemma \ref{lem:29} and a standard approximation procedure. 
For uniqueness, let $\bfPBar$ be as in Definition \ref{def:2}. 
The representation theorem of  \cite[Theorem 8.2]{DaPratoZabczyk14_StochasticEquations} implies that 
on an augmented probability space,
there exists a standard Brownian motion $(B_t)_{t \in [0, T]}$ such that 
\begin{equ}
\bfM_t = \varsigma\int_0^t \nabla \bm{\etaBar}_s \cdot \dd B_s\,. 
\end{equ}
This means that $\bm{\etaBar}$ is a solution of SLTE as given in Definition \ref{def:1} whose law is 
uniquely determined by Lemma \ref{lem:12} and is given by the Brownian transportation of the GFF.
\end{proof}

\section{The role of the Nuisance region and the DCGFF}
\label{a:nuisance}

The goal of this section is to provide some insight regarding how we were lead (and why it is necessary) to distinguish 
between bulk and nuisance region and why in the context of diffusions in divergence free vector fields 
this is not necessary. Our analysis finds its roots in the work of T\'oth and Valk\'o~\cite{TothValko12_SuperdiffusiveBounds}, 
where the authors determine superdiffusive bounds not only for the SRBP but also for a model 
called Diffusion in the Curl of the Gaussian Free Field (DCGFF). 
Despite the similarities of the two models, which we will shortly discuss, the study of the SRBP 
is already seen  therein to be more challenging in that a prototype
of our bulk/nuisance ($B/N$) split is needed, whereas when treating the DCGFF this is not the case. 

The DCGFF is defined as 
\begin{equ}[e:DCGFF]
 \dd \XTil_t = \dd B_t + \gamma \omegaTil(\XTil_t) \dd t \,,\qquad \XTil_0 = 0 
 \end{equ}
where $B$ is a two-dimensional Brownian motion and $\omegaTil$ is an element of $\OmegaTil$, 
which is a space of divergence free (as opposed to $\Omega$, whose elements are rotation free) 
smooth vector fields. We consider the probability measure $\piTil(d \omegaTil)$ on $\OmegaTil$, 
which is the law of $\nabla \times(\sqrt{V}\ast\Phi)$ for $\Phi$ a $2d$ GFF.  
Even though seemingly unrelated, it turns out that the generator of the environment process $\tilde\eta$ associated to 
the DCGFF is exactly the same as that for the SRBP as given in~\eqref{eq:ops} but for a crucial 
difference, namely, the scalar products in the definition of $\calA_\pm$ is replaced by the cross 
products~\cite[eq. (3.10)]{CannHaunToni22_SqrtLog}
From a technical viewpoint, this is the reason why the off-diagonal terms 
are lower order in one case but not on the other. Let us verify the previous statement 
in a simple but informative example concerning the function $f_1$ defined in \eqref{e:fdef}. 

\begin{proposition}\label{prop:Off}
Let $\phi = (\lambda+\calS)^{-\half} \calA_+ \calS^{-\half} \psi$, where $\psi = \gamma (\lambda+\calS)^{-1/2} f_1$ and $f_1 \in \fock_1$ be given as 
\begin{equ}
\text{SRBP:} \quad f_1 \eqdef -\frac{\iota}{2\pi} e_1p_2\,,\qquad\qquad
\text{DCGFF:} \quad f_1 \eqdef -\frac{\iota}{2\pi} e_2p_2\,.
\end{equ}
Consider the decomposition $\phi = \phi[1]+\phi[2]$ as given in  \eqref{e:phii} with $\ST\equiv0$ and $R=\R^2$. 
Then it holds that 
\begin{align}
\text{SRBP:}& \quad \langle\phi[1],\phi[2]\rangle = -\frac{1}{32} + O(\gamma^2)\,, \label{eq:58} \\
\text{DCGFF:}& \quad \langle\phi[1],\phi[2]\rangle = O(\gamma^2)\,. \label{eq:43} 
\end{align}
\end{proposition}
\begin{proof}
We begin with \eqref{eq:58}. By definition, we have 
\begin{align}
&\langle\phi[1],\phi[2]\rangle = \frac{\gamma^4}{32 \pi^2} \int \frac{(p_1 \cdot p_2)^2 e_1p_1 e_1p_2 \mu_2(\dd p_{1:2})}{(\lambda+\half|p_1+p_2|^2)(\lambda+\half|p_1|^2)^\half(\lambda+\half|p_2|^2)^\half|p_1||p_2|} \nonumber \\
&= \frac{\gamma^4}{64 \pi^2} \int \frac{(p_1 \cdot p_2)^3 \mu_2(\dd p_{1:2})}{(\lambda+\half|p_1+p_2|^2)(\lambda+\half|p_1|^2)^\half(\lambda+\half|p_2|^2)^\half |p_1||p_2|} \nonumber
\end{align}
where the last step is a consequence of the fact that, by rotating both variables $p_1,p_2$ by $\pi/2$, 
the integrand in the previous line is the same but with $e_2p_1 e_2p_2$ in place of $e_1p_1 e_1p_2$. 
Then, it is not hard to see that, by Lemma~\ref{lem:4}, the previous is equal, up to $O(\gamma^2)$, to 
\begin{equation}
\label{eq:63}
\frac{\gamma^4}{64 \pi^2} \int \frac{(p_1 \cdot p_2)^3 N(p_2, p_1) \mu_2(\dd p_{1:2})}{(\lambda+\half|p_1+p_2|^2)(\lambda+\half|p_1|^2)^\half(\lambda+\half|p_2|^2)^\half |p_1||p_2|}+O(\gamma^2) \,.
\end{equation}
Moreover, a sequence of basic computations utilising both the properties of $V$ as stated in Assumption \ref{a:V} and 
the fact that in the nuisance region $p_1 \approx -p_2$, show that \eqref{eq:63} reduces to
\begin{align*}
&\frac{\gamma^4}{64 \pi^2} \int \frac{(p_1 \cdot p_2)^3 N(p_2, p_1) }{(\lambda+\half|p_1+p_2|^2)(\lambda+\half|p_1|^2) |p_1|^2} \mu_2(\dd p_{1:2}) + O(\gamma^2)\\
&= -\frac{\gamma^4}{64 \pi^2} \int \frac{|p_1|^4 |p_2|^2 N(p_2, p_1)}{(\lambda+\half|p_1+p_2|^2)(\lambda+\half|p_1|^2) |p_1|^2}\mu_2(\dd p_{1:2}) + O(\gamma^2) \\
&=  -\frac{\gamma^4}{64 \pi^2} \int \frac{\VHat(p_1)^2}{\lambda+\half|p_1|^2} \int_{N(p_2, p_1)} \frac{1}{\lambda+\half|p_1+p_2|^2} dp_2 \Big) \dd p_1 + O(\gamma^2) \\
&=  -\frac{\gamma^4}{32 \pi} \int \frac{\VHat(p_1)^2}{\lambda+\half|p_1|^2} \Big( \log( \lambda + \tfrac{1}{9} |p_1|^2) - \log \lambda \Big) \dd p_1+ O(\gamma^2) \\
&=  -\frac{\gamma^4}{16} \int_0^1 \frac{1}{\lambda+\rho} \Big( \log( \lambda + \rho) - \log \lambda \Big) \dd \rho + O(\gamma^2) = - \frac{1}{32} + O(\gamma^2)
\end{align*}
which implies the result for the SRBP. 

Turning to \eqref{eq:43}, we may adapt the proof of Lemma \ref{lem:6} to the case of DCGFF, 
simply using the bound $|p \times q | \le |p| |q|$ instead of the bound $|p \cdot q | \le |p| |q|$, 
and we obtain $\langle\phi^B[1],\phi^B[2]\rangle = O(\gamma^2)$, where we have also invoked Lemma \ref{lem:5}. Moreover, the cross terms $\langle\phi^B[1],\phi^N[2]\rangle, \langle\phi^B[2],\phi^N[1]\rangle$ are seen to be negligible, for example, swapping the nuisance regions, similarly to what was done in \eqref{e:Intermediate}
\begin{align*}
|\langle\phi^B[1],\phi^N[2]\rangle| \le \gamma^4 \int \frac{(p_1 \times p_2)^2 B(p_1, p_2) N(p_2, p_1)}{(\lambda+|p_1+p_2|^2)(\lambda+|p_1|^2)^\half(\lambda+|p_2|^2)^\half} \mu_2(\dd p_{1:2}) \\
= \gamma^4 \int \frac{(p_1 \times p_2)^2 B(p_1, p_2) N(p_1, p_2)}{(\lambda+|p_1+p_2|^2)(\lambda+|p_1|^2)^\half(\lambda+|p_2|^2)^\half} \mu_2(\dd p_{1:2}) + O(\gamma^2) 
\end{align*}
and this final term is zero because $B(p_1, p_2) N(p_1, p_2) = 0$. 

Therefore, it remains to consider
\begin{equ}
|\langle\phi^N[1],\phi^N[2]\rangle| \le \gamma^4 \int \frac{(p_1 \times p_2)^2 N(p_1, p_2) N(p_2, p_1)}{(\lambda+|p_1+p_2|^2)(\lambda+|p_1|^2)^\half(\lambda+|p_2|^2)^\half} \mu_2(\dd p_{1:2})\,.
\end{equ}
We have $(p_1 \times p_2)^2 = |p_1|^2|p_2|^2\sin^2(\theta)$ where $\theta$ is the angle between $p_1$ and $p_2$ 
and  $\sin^2(\theta) \le |p_1+p_2|^2/|p_1|^2$. 
By symmetry, we can restrict to $|p_1| \le |p_2|$, in which case we obtain the upper bound 
\begin{equ}
\gamma^4 \int  \frac{\VHat(p_1)}{(\lambda+|p_1|^2)|p_1|^2} \Big(\int_{|p_1+p_2|< \frac{1}{3} |p_1|} \frac{\VHat(p_2) |p_1+p_2|^2}{\lambda+|p_1+p_2|^2} \dd p_2\Big) \dd p_1 = O(\gamma^2)
\end{equ}
and the proof of~\eqref{eq:43} is completed. 
\end{proof}

What the previous statement together with the observations made above concerning the DCGFF hints at, is that 
the nuisance region is not relevant for the dynamics of DCGFF, while it is for that of the SRBP. 
Therefore, we expect not only that our techniques would  work for~\eqref{e:DCGFF} but further, that it would 
be much simpler to obtain the analogue of Theorem~\ref{thm:1} (and Theorem~\ref{thm:2}) in that setting,
since the whole of Section~\ref{sec:contr-nusi-regi} could be avoided. 
\endappendix

\section*{Acknowledgements} 
The authors would like to thank B. T\`oth for elucidating discussions concerning the SRBP and his work on it, 
and M. Maurelli for explaining the solution theory for stochastic linear transport equation. 
G.~C. gratefully acknowledges financial support
via the UKRI FL fellowship ``Large-scale universal behaviour of Random
Interfaces and Stochastic Operators'' MR/W008246/1.  

\bibliography{srbp_refs,srbp_refs_manual}
\bibliographystyle{Martin}

\end{document}